\documentclass[twoside, a4paper,10pt]{amsart}

\usepackage{amsmath}  
\usepackage{amsfonts}
\usepackage{amsthm}  
\usepackage{amssymb}
\usepackage{enumerate}
\usepackage[all]{xy}
\SelectTips{cm}{}
\SilentMatrices

\usepackage[pdfauthor={Steffen Sagave and Christian Schlichtkrull}]{hyperref}

\numberwithin{equation}{section}
\theoremstyle{plain}
\newtheorem{lemma}[subsection]{Lemma}
\newtheorem{theorem}[subsection]{Theorem}

\newtheorem{corollary}[subsection]{Corollary}
\newtheorem{proposition}[subsection]{Proposition}
\theoremstyle{definition}

\newtheorem{construction}[subsection]{Construction}
\newtheorem{definition}[subsection]{Definition}
\newtheorem{example}[subsection]{Example}
\newtheorem{remark}[subsection]{Remark}

\newcommand{\mC}{{\mathbb C}}

\newcommand{\mL}{{\mathbb L}}
\newcommand{\mN}{{\mathbb N}}

\newcommand{\mR}{{\mathbb R}}
\newcommand{\mS}{{\mathbb S}}
\newcommand{\mZ}{{\mathbb Z}}

\newcommand{\cA}{{\mathcal A}}
\newcommand{\cB}{{\mathcal B}}
\newcommand{\cC}{{\mathcal C}}
\newcommand{\cD}{{\mathcal D}}
\newcommand{\cE}{{\mathcal E}}

\newcommand{\cI}{{\mathcal I}}

\newcommand{\cK}{{\mathcal K}}
\newcommand{\cM}{{\mathcal M}}
\newcommand{\cN}{{\mathcal N}}

\newcommand{\cP}{{\mathcal P}}
\newcommand{\cS}{{\mathcal S}}
\newcommand{\cT}{{\mathcal T}}
\newcommand{\cU}{{\mathcal U}}

\DeclareMathOperator{\id}{id}

\DeclareMathOperator{\Mod}{Mod}
\DeclareMathOperator{\Map}{Map}
\DeclareMathOperator{\Fun}{Fun}
\DeclareMathOperator{\Ho}{Ho}
\DeclareMathOperator{\Ev}{Ev}

\DeclareMathOperator{\const}{const}
\DeclareMathOperator{\colim}{colim}
\DeclareMathOperator{\hocolim}{hocolim}

\DeclareMathOperator{\Sing}{Sing}

\DeclareMathOperator{\concat}{\sqcup}

\newcommand{\eins}{\mathbf{1}}
\newcommand{\tensor}{\otimes}
\newcommand{\ovl}{\overline}

\newcommand{\ot}{\leftarrow}
\newcommand{\iso}{\cong}

\newcommand{\op}{{\mathrm{op}}}
\newcommand{\sm}{\wedge}
\newcommand{\wdg}{\vee}

\newcommand{\Sp}{\mathit{Sp}}
\newcommand{\Spsym}{{\mathrm{\textrm{Sp}^{\Sigma}}}}
\newcommand{\bld}[1]{{\mathbf{#1}}}

\newcommand{\gp}{\mathrm{gp}}
\newcommand{\un}{\mathrm{un}}
\newcommand{\GL}{\mathrm{GL}}
\newcommand{\gl}{\mathrm{gl}}

\newcommand{\geor}[1]{\lvert#1\rvert}

\DeclareMathOperator{\diag}{diag}

\newcommand{\xr}{\xrightarrow}
\newcommand{\xl}{\xleftarrow}

\newcommand{\arxivlink}[1]{\href{http://arxiv.org/abs/#1}{\texttt{arXiv:#1}}}

\begin{document} \title{Group completion and units in
  \texorpdfstring{$\cI$}{I}-spaces}

\author{Steffen Sagave} \address{Mathematical
Institute, University of Bonn, Endenicher Allee 60, 53115 Bonn,
Germany} \email{sagave@math.uni-bonn.de}
\author{Christian Schlichtkrull}
\address{Department of Mathematics, University of Bergen, Johannes
  Brunsgate 12, 5008 Bergen, Norway}
\email{krull@math.uib.no}

\date{\today}
\begin{abstract}
The category of $\cI$-spaces is the diagram category of spaces indexed by finite sets and injections. This is a symmetric monoidal category whose commutative monoids model all $E_{\infty}$ spaces. Working in the category of $\cI$-spaces enables us to simplify and strengthen previous work on group completion and units of $E_{\infty}$ spaces. As an application we clarify the relation to $\Gamma$-spaces and show how the spectrum of units associated with a commutative symmetric ring spectrum arises through a chain of Quillen adjunctions.
\end{abstract}
\maketitle
\section{Introduction}
In homotopy theory, an $E_{\infty}$ space is a space equipped with a multiplicative structure arising from the action of an $E_{\infty}$ operad. Such an operad action encodes all higher coherence homotopies between iterated multiplications in the space. 
For the purpose of homotopy theory, this is often the right way to express commutativity.
 In contrast, the notion of a strictly commutative monoid in spaces is usually too rigid since it does not model enough homotopy types.  

However, there is a different way to get at a notion of commutativity suitable for doing homotopy theory: Instead of changing the meaning of ``commutative'' to ``$E_{\infty}$'', one may change the meaning of ``space''. Specifically, working in the category of 
\emph{$\cI$-spaces} studied by the authors in \cite{Sagave-S_diagram}, one obtains a setting in which the commutative monoids do model all $E_{\infty}$ spaces.  

In more detail, let $\cI$ the category with objects the finite sets $\bld{n}
= \{1,\dots,n\}$, including the empty set $\bld0$,  and morphisms the injective maps.  By definition, an $\cI$-space is a
functor from $\cI$ to the category of (unbased) spaces $\cS$. As it is generally the case for a category of diagrams in spaces indexed by a small symmetric monoidal index category,
the resulting category $\cS^{\cI}$ of $\cI$-spaces inherits a
symmetric monoidal structure from the concatenation of finite sets in $\cI$. A commutative monoid with respect to this structure will be called 
a \emph{commutative $\cI$-space monoid}, and we write $\cC\cS^{\cI}$ for the category of such commutative monoids. The statement that commutative $\cI$-space monoids model all $E_{\infty}$ spaces is made precise in \cite{Sagave-S_diagram} where it is shown that the category $\cC\cS^{\cI}$ has a model structure, called the \emph{positive 
$\cI$-model structure}, which makes it Quillen equivalent to the category of $E_{\infty}$ spaces. The weak equivalences in the positive $\cI$-model structure are the 
\emph{$\cI$-equivalences}, that is, the maps 
$A\to B$ that induce a weak equivalence $A_{h\cI}\to B_{h\cI}$ of the associated homotopy colimits. For an $\cI$-space monoid $A$, the homotopy colimit $A_{h\cI}$ inherits a monoid structure which extends to an $E_{\infty}$ structure provided that $A$ is commutative. Based on this we think of the homotopy colimit functor from $\cS^{\cI}$ to $\cS$ as a forgetful functor taking commutative $\cI$-space monoids to $E_{\infty}$ spaces. The simple and explicit combinatorics underlying the category $\cC\cS^{\cI}$ often makes it profitable to translate questions about $E_{\infty}$ spaces to questions about commutative $\cI$-space monoids. Examples of this are given in 
\cite{Blumberg-C-S_THH-Thom, Rognes_TLS,Schlichtkrull_units, 
Schlichtkrull_Thom-symm, Schlichtkrull_higher}, and in the present paper, where the focus is on questions related to group completion and units.  

The category $\cC\cS^{\cI}$ is related to the category $\cC\Spsym$ of commutative symmetric ring spectra through a Quillen adjunction
\begin{equation}\label{eq:adj:csi-spsym} 
\mS^{\cI} \colon \cC\cS^{\cI} \rightleftarrows \cC\Spsym \colon
\Omega^{\cI}
\end{equation}
with respect to the positive $\cI$-model structure on
$\cC\cS^{\cI}$ and the positive model structure on $\cC\Spsym$ introduced
in~\cite{MMSS} (see \cite[Proposition~3.19]{Sagave-S_diagram} for details). The right adjoint $\Omega^{\cI}$ sends a commutative symmetric ring spectrum $R$ to the commutative $\cI$-space monoid $\Omega^{\cI}(R)$
with $\Omega^{\cI}(R)(\bld{n})= \Omega^n(R_n)$ and a monoid structure induced by the multiplication in $R$.  

\subsection{Group completion}\label{subsec:grp-cpl}
Recall that a (simplicial or topological) monoid $M$ is \emph{grouplike} if the monoid of connected components $\pi_0(M)$ is a group. We say that a map of homotopy commutative (simplicial or topological) monoids $M\to N$ is a \emph{group completion} if $N$ is grouplike and the map of classifying spaces $B(M)\to B(N)$ is a weak homotopy equivalence.  This implies that $N$ is equivalent to $\Omega(B(M))$ (with an implicit fibrant replacement of $B(M)$ in the simplicial setting and the extra assumption that $M$ and $N$ be well-based in the topological setting).  

A commutative $\cI$-space monoid $A$ is said to be \emph{grouplike} if the underlying $E_{\infty}$ space $A_{h\cI}$ is grouplike, that is, if the commutative monoid 
$\pi_0(A_{h\cI})$ is a group. While there are well-known constructions of group completions within $E_{\infty}$ spaces (e.g.\ by May~\cite[Theorem 2.3]{May_group-completions} and Basterra-Mandell \cite[Theorem 6.5]{Basterr-M_homology-cohomology}), one aim of the present paper is to show that the process of group completion can be conveniently lifted to the category of commutative $\cI$-space monoids.

Our first approach to group completion uses the bar construction. For a commutative $\cI$-space monoid $A$, we write $B(A)$ for the bar construction formed in $\cC\cS^{\cI}$. This construction is left adjoint to the loop functor $\Omega$ on $\cC\cS^{\cI}$ and the unit for the adjunction is a natural map $A\to \Omega(B(A))$. Composing with a functorial fibrant replacement $B(A)\to B(A)^{\textrm{$\cI$-fib}}$ in the positive $\cI$-model structure, we get a natural map of commutative $\cI$-space monoids  
$\eta_A^{\cI} \colon A\to \Omega(B(A)^{\textrm{$\cI$-fib}})$. The next theorem implies that this models a group completion of $A$ provided the latter is cofibrant.

\begin{theorem}\label{thm:intro-bar-completion}
Let $A$ be a cofibrant commutative $\cI$-space monoid. Then the induced map of homotopy colimits $(\eta_A^{\cI})_{h\cI}\colon A_{h\cI}\to \Omega(B(A)^{\textrm{$\cI$-fib}})_{h\cI}$ is a group completion of the $E_{\infty}$ space $A_{h\cI}$.
\end{theorem}
In the statement of the theorem we have included the $\cI$-fibrant replacement of $B(A)$ in order not to make additional assumption on $A$. However, we show in 
Section~\ref{sec:bar-construction} that under a mild ``semistability'' condition on $A$, the $\cI$-fibrant replacement can be dropped. Furthermore, the cofibrancy condition on $A$ can be weakened to a ``flatness'' condition on the underlying $\cI$-space. Thus, under these assumptions the theorem says that the usual group completion for homotopy commutative simplicial or topological monoids in terms of the bar construction lifts to commutative $\cI$-space monoids.  This use of the bar construction is analogous to the suspension of $E_{\infty}$ spaces in \cite[Theorem 6.5]{Basterr-M_homology-cohomology}, but has the advantage of an explicit description in terms of the underlying symmetric monoidal structure. The approach to group completion developed here is used by Rognes in his work on topological logarithmic structures \cite[\S 6]{Rognes_TLS}.

Our second approach to group completion is model categorical. We define a \emph{group completion model structure} $\cC\cS^{\cI}_{\gp}$ on the category of commutative 
$\cI$-space monoids as the left Bousfield localization of $\cC\cS^{\cI}$ with respect to a certain universal group completion map. Here and elsewhere, the notation $\cC\cS^{\cI}$ indicates the category of commutative $\cI$-space monoids equipped  with the positive 
$\cI$-model structure, and a decoration on $\cC\cS^{\cI}$ means that we have kept the underlying category but changed the model structure to something else. The next theorem shows that this localization process has the expected effect on weak equivalences and fibrant objects.

\begin{theorem}\label{thm:csi-gp}
A map $A\to A'$ is a weak equivalence in $\cC\cS^{\cI}_{\gp}$ if and only if the induced map of bar constructions $B(A_{h\cI})\to B(A'_{h\cI})$ is a weak homotopy equivalence. The fibrant objects in $\cC\cS^{\cI}_{\gp}$ are the fibrant objects in $\cC\cS^{\cI}$ which are grouplike, and a fibrant replacement $A\to A^{\gp}$ in $\cC\cS^{\cI}_{\gp}$ induces a group completion $A_{h\cI} \to (A^{\gp})_{h\cI}$ of $A_{h\cI}$. 
\end{theorem} 

Thus, the weak equivalences in $\cC\cS^{\cI}_{\gp}$ are the maps $A\to A'$ for which the underlying map of $E_{\infty}$ spaces $A_{h\cI}\to A'_{h\cI}$ becomes a weak equivalence after group completion. One advantage of the model category approach is that it gives a functorial group completion for all objects without further assumptions. Although group completion in the $E_{\infty}$ context has been known for a long time, we do not know of a reference where it is constructed as a fibrant replacement. 

It is a formal consequence of the definition that the identity functor is both the
left and right adjoint in a Quillen adjunction
\begin{equation}\label{eq:adj:csi-gp}
  L_{\textrm{gp}} \colon \cC\cS^{\cI} \rightleftarrows \cC\cS^{\cI}_{\mathrm{gp}} \colon R_{\mathrm{gp}}.
\end{equation}
Passing to the total derived functors, this induces an adjunction
$(L_{\textrm{gp}}^{\mL}, R_{\textrm{gp}}^{\mR})$ on the level of homotopy
categories restricting to an equivalence between $\Ho(\cC\cS^{\cI}_{\textrm{gp}})$ 
and the full subcategory of grouplike objects in
$\Ho(\cC\cS^{\cI})$. Under this identification, $L_{\textrm{gp}}^{\mL}$ becomes left adjoint to the inclusion of the grouplike objects, just as the group completion of commutative monoids is left adjoint to the forgetful functor from commutative groups to commutative
monoids.

The fibrations in $\cC\cS^{\cI}_{\gp}$ form an interesting class of maps and we show in Section~\ref{sec:repletion} that they generalize the notion of \emph{replete maps} used by Rognes~\cite{Rognes_TLS} in his definition of logarithmic topological Hochschild homology.

\subsection{The relation to \texorpdfstring{$\Gamma$}{Gamma}-spaces}
A classical result in homotopy theory, known as the \emph{recognition
  principle} \cite{Boardman-V_homotopy-invariant,May_geometry}, states that the homotopy theory of grouplike $E_{\infty}$ spaces is equivalent
to the homotopy theory of infinite loop spaces, hence also to the homotopy theory of connective spectra. On the other hand, the work of Segal~\cite{Segal_categories} and Bousfield-Friedlander~\cite{Bousfield-F_Gamma-bisimplicial} shows that Segal's category of $\Gamma$-spaces $\Gamma^{\op}\cS_*$ provides a convenient model
for the homotopy theory of connective spectra. Our next result bypasses the theory of operads and establishes a direct equivalence between the homotopy category of 
$\Gamma$-spaces and the homotopy category of grouplike commutative $\cI$-space monoids.  
\begin{theorem}\label{thm:Gamma-csigp}
There is a Quillen equivalence
\begin{equation}\label{eq:adj:Gamma-csigp}
\Lambda \colon \Gamma^{\op}\cS_* \rightleftarrows \cC\cS^{\cI}_{\mathrm{gp}} \colon \Phi
\end{equation}
between the categories of $\Gamma$-spaces with the stable Q-model
structure and commutative $\cI$-space
monoids with the group completion model structure.
\end{theorem}
The stable $Q$-model structure on $\Gamma^{\op}\cS_*$ was introduced by 
Schwede~\cite{Schwede_Gamma-spaces} and is Quillen equivalent to the stable model structure considered by Bousfield and Friedlander. 
A description of a direct Quillen equivalence between grouplike $E_{\infty}$ spaces
and $\Gamma$-spaces does not seem
to be covered in the literature. One of its advantages is that it not only gives an isomorphism between morphism sets in the respective homotopy categories, but also an equivalence between the derived mapping spaces. 

The spectrum $B^{\infty}(A)$ associated to a commutative $\cI$-space monoid $A$ has the following explicit description: it's $n$th space is the based homotopy colimit $B^n(A)_{h^*\cI}$ of the $n$-fold iterated bar construction $B^n(A)$. This construction is formally very similar to the usual definition of the Eilenberg-Mac Lane spectrum associated to an abelian group. In Section~\ref{subsec:stabilization-proof} we show how to deduce the following corollary from Theorem~\ref{thm:Gamma-csigp}.

\begin{corollary}\label{cor:B-infty-A}
The homotopy category of grouplike commutative $\cI$-space monoids is equivalent to the homotopy category of connective spectra via the functor sending $A$ to the connective spectrum $B^{\infty}(A')$ defined by a cofibrant replacement 
$A'$ of $A$.
\end{corollary}

\subsection{Units}
Recall that an $E_{\infty}$ space $X$ has a subspace of (homotopy) units $X^{\times}$ defined as the union of those path components that represent units in the commutative monoid $\pi_0(X)$. This construction lifts to $\cC\cS^{\cI}$ in the sense that a commutative $\cI$-space monoid $A$ has a submonoid $A^{\times}$ of (homotopy) units such that the inclusion $A^{\times}\to A$ induces an isomorphism 
$(A^{\times})_{h\cI}\cong (A_{h\cI})^{\times}$. An important example is the construction of the units of a commutative symmetric ring spectrum $R$ as the commutative $\cI$-space monoid  $\GL_1(R) =\Omega^{\cI}(R)^{\times}$. This model of the units is useful, for instance, in the study of algebraic $K$-theory~\cite{Schlichtkrull_units} and Thom
spectra~\cite{Blumberg-C-S_THH-Thom,Schlichtkrull_Thom-symm}.

The construction of units in $\cC\cS^{\cI}$ also has a model categorical interpretation. We define the units model structure $\cC\cS^{\cI}_{\un}$ as the right Bousfield localization of $\cC\cS^{\cI}$ with respect to the inclusions $A^{\times}\to A$. The next theorem shows that this localization process has the expected effect on weak equivalences and cofibrant objects.

\begin{theorem}\label{thm:intro-units}
A map $A\to A'$ is a weak equivalence in $\cC\cS^{\cI}_{\un}$ if and only if the induced map $A_{h\cI}^{\times} \to A_{h\cI}'^{\times}$ is a weak homotopy equivalence. The cofibrant objects in $\cC\cS^{\cI}_{\un}$ are the cofibrant objects in $\cC\cS^{\cI}$ which are grouplike, and if $A^{\un}\to A$ is a cofibrant replacement in $\cC\cS^{\cI}$, then there is a canonical $\cI$-equivalence $A^{\un}\to A^{\times}$.
\end{theorem}

As for the group completion model structure, the identity functor participates both as the left and right adjoint in a Quillen adjunction
\begin{equation}\label{eq:adj:un-csi}
  L_{\textrm{un}} \colon \cC\cS^{\cI}_{\textrm{un}} \rightleftarrows \cC\cS^{\cI} \colon R_{\textrm{un}}.
\end{equation}
The induced adjunction $(L_{\textrm{un}}^{\mL},R_{\textrm{un}}^{\mR})$ of homotopy categories restricts to an equivalence between $\Ho(\cC\cS^{\cI}_{\textrm{un}})$ and the full subcategory of grouplike objects in $\Ho(\cC\cS^{\cI})$. 
Under this identification, $R_{\textrm{un}}^{\mR}$ becomes right adjoint to the inclusion of the grouplike objects, just as forming the units of a commutative monoid is right adjoint to the forgetful functor from commutative groups to commutative monoids.

One may argue that there is little use in replacing the simple and explicit 
construction of the units by a cofibrant replacement in a complicated
model category. However, we illustrate below that this is useful for
analyzing spectra of units.
For this we need the following result which is an immediate consequence of the fact that 
the cofibrant-fibrant objects in $\cC\cS^{\cI}_{\textrm{un}}$
and $ \cC\cS^{\cI}_{\textrm{gp}}$ coincide.
\begin{proposition}
  The composite of \eqref{eq:adj:un-csi} and \eqref{eq:adj:csi-gp} is
  a Quillen equivalence
\begin{equation}\label{eq:adj:un-gp}
  L_{\mathrm{un/gp}} \colon \cC\cS^{\cI}_{\mathrm{un}} \rightleftarrows \cC\cS^{\cI}_{\mathrm{gp}} \colon R_{\mathrm{un/gp}}.
\end{equation} 
\end{proposition}

\subsection{Spectra of units}
Assembling the Quillen adjunctions and Quillen equivalences
\eqref{eq:adj:csi-spsym}, \dots, \eqref{eq:adj:un-gp} considered so far,
we get a diagram
\begin{equation}\label{diag:all-adjunctions}
  \xymatrix{
& & \cC\cS^{\cI} \ar@<.4ex>[rr]^{\mS^{\cI}} \ar@<-.4ex>[ld]_(.72){\!\!\!L_{\textrm{gp}}} \ar@<-.4ex>[rd]_(.35){R_{\textrm{un}}\!\!\!\!} & & \cC\Spsym \ar@<.4ex>[ll]^{\Omega^{\cI}} \\
\Gamma^{\op}\cS_* \ar@<.4ex>[r]^{\Lambda} & \cC\cS^{\cI}_{\textrm{gp}} \ar@<.4ex>[l]^{\Phi} \ar@<-.4ex>[rr]_{R_{\textrm{un/gp}}} \ar@<-.4ex>[ru]_(.65){\!\!\!R_{\textrm{gp}}} & & \cC\cS^{\cI}_{\textrm{un}} \ar@<-.4ex>[ll]_{L_{\textrm{un/gp}}} \ar@<-.4ex>[lu]_(.3){\!\!\!L_{\textrm{un}}} 
  }
 \end{equation}
in which the two bottom horizontal Quillen adjunctions are Quillen
equivalences. Although all the functors in the triangle are identity functors, the same does of course not hold for their derived functors because the model structures differ.

For a commutative symmetric ring spectrum $R$, the spectrum of units
$\gl_1(R)$ can be defined as the $\Gamma$-space which the functor
$\Phi$ of \eqref{eq:adj:Gamma-csigp} associates with the grouplike
commutative $\cI$-space monoid $\GL_1(R)$ considered above. In terms of the derived
functors of the Quillen functors in \eqref{diag:all-adjunctions}, the
spectrum of units is the functor
\begin{equation}\label{eq:gl_1-def} \gl_1 = (\Phi^{\mR})(L_{\textrm{un/gp}}^{\mL})(R_{\textrm{un}}^{\mR})(\Omega^{\cI}_{\mR}) \colon
  \Ho(\cC\Spsym) \to \Ho(\Gamma\cS_*).\end{equation}
Since $( L_{\textrm{un/gp}},R_{\textrm{un/gp}})$ is a Quillen equivalence, it induces an
equivalence of homotopy categories.  So the total derived functor $L_{\textrm{un/gp}}^{\mL}$ is both a left and a right adjoint, and we get the following corollary.

\begin{corollary}
  Passing to total derived functors, the adjunctions in
  \eqref{diag:all-adjunctions} exhibit the spectrum of units as the
  right adjoint in an adjunction
  $ \Ho(\Gamma\cS_*) \rightleftarrows \Ho(\cC\Spsym)$.
\end{corollary}
The last corollary gives an independent proof of~\cite[Theorem
3.2]{Ando-B-G-H-R_units-Thom} in the language of diagram spaces and
diagram spectra. We also show that the right adjoint of the adjunction is represented by the explicit $\Gamma$-space model of the units defined by the second author 
in~\cite{Schlichtkrull_units}.

\subsection{Conventions regarding spaces}
For the results stated in the introduction, the category of spaces $\cS$ may be interpreted either as the category of compactly generated weak Hausdorff topological spaces or as the category of simplicial sets. However, starting from 
Section~\ref{sec:I-spaces} we stipulate that $\cS$ be the category of simplicial sets and the body of the paper is written in the simplicial context. This is mainly for the 
sake of the exposition: Working in a simplicial context there is an occasional need for fibrant replacement while working topologically we would sometimes have to impose cofibrancy conditions and require base points to be non-degenerate. We show how to obtain the topological version of our results in Appendix~\ref{app:topological}.

\subsection{Organization}
In Section \ref{sec:I-spaces}, we recall and develop some foundational material
on $\cI$-spaces, and Section \ref{sec:I-space-monoids} collects some
results about commutative $\cI$-space monoids. Section
\ref{sec:bar-construction} features the bar construction for
commutative $\cI$-space monoids and the proof of 
Theorem~\ref{thm:intro-bar-completion}. In Section \ref{sec:csi-gp}, we
construct the group completion model structure and prove Theorem
\ref{thm:csi-gp}.  We study the relation to $\Gamma$-spaces and prove
Theorem \ref{thm:Gamma-csigp} in Section \ref{sec:Gamma-csigp}. The
final Section \ref{sec:un-csi} is about units and contains the proof
of Theorem \ref{thm:intro-units}. In Appendix \ref{app:cellular} we verify that the positive 
$\cI$-model structure on $\cC\cS^{\cI}$ is cellular, Appendix \ref{app:bi-Gamma} is about bi-$\Gamma$-spaces, and in Appendix~\ref{app:topological} we derive the topological version of our results 
 
\subsection{Acknowledgments}
The authors benefited from a visit of the second author to Bonn which
was funded by the HCM in Bonn. They thank Jens Hornbostel, John
Rognes, and Stefan Schwede for helpful conversations related to this
project. The suggestions made by an anonymous referee also helped to
improve the manuscript.

\section{Preliminaries on \texorpdfstring{$\cI$}{I}-spaces}\label{sec:I-spaces}
We begin by recalling some basic facts  about $\cI$-spaces. Let $\cI$ be the
category with objects the finite sets $\bld{n} = \{1, \dots, n\}$, including the 
empty set $\bld{0}$, and with morphisms the injective maps. The
concatenation $\bld{m}\concat\bld{n}$ defined by letting $\bld{m}$
correspond to the first $m$ and $\bld{n}$ to the last $n$ elements of
$\{1, \dots, m+n\}$ gives $\cI$ the structure of a symmetric monoidal 
category with unit $\bld 0$. The symmetry isomorphisms are the obvious
$(m,n)$-shuffles $\tau_{m,n} \colon \bld{m}\concat\bld{n} \to
\bld{n}\concat\bld{m}$.

By definition, an \emph{$\cI$-space} is a functor from $\cI$ to the category $\cS$ of 
unbased simplicial sets, and we write $\cS^{\cI}$ for the (functor) category
of $\cI$-spaces. We will frequently consider the Bousfield-Kan
homotopy colimit of an $\cI$-space $X$  (as defined in \cite{Bousfield_K-homotopy}) and abbreviate it by $X_{h\cI}$,
\[ X_{h\cI} = \hocolim_{\cI}X =
\textrm{diag}\left([s]\mapsto \coprod_{\bld{n}_0\ot \dots \ot
    \bld{n}_s}X(\bld{n}_s)\right), \] where $\diag$ denotes the
diagonal of a bisimplicial set (which is
one of the isomorphic incarnations of the realization functor from bisimplicial
sets to simplicial sets, see e.g.\ ~\cite[Theorem 15.11.6]{Hirschhorn_model}).

Given $\cI$-spaces $X$ and $Y$, the product $X\boxtimes Y$ is the $\cI$-space defined by the left Kan extension of the ($\cI\times\cI$)-diagram $(\bld n_1,\bld n_2)\mapsto 
X(\bld n_1)\times Y(\bld n_2)$ along the concatenation $\concat\colon \cI\times\cI\to \cI$, that is, 
\[ 
(X\boxtimes Y)(\bld{n}) \iso
  \colim_{\bld{n}_1\concat\bld{n}_2\to\bld{n}} X(\bld{n}_1) \times Y(\bld{n}_2)
\] 
with the colimit taken over the comma category $(\concat\downarrow \bld{n})$. 
This defines a symmetric monoidal structure on $\cS^{\cI}$ with the constant 
$\cI$-space $\cI(\bld{0},-)$ as the monoidal unit. 

\subsection{The positive \texorpdfstring{$\cI$}{I}-model structure on 
\texorpdfstring{$\cS^{\cI}$}{SI}}\label{sec:I-space-model-str}
A map of $\cI$-spaces $A\to B$ is said to be an \emph{$\cI$-equivalence} if the induced map of homotopy colimits $A_{h\cI}\to B_{h\cI}$ is a weak equivalence. This is the fundamental notion of equivalence for $\cI$-spaces and participates as the weak equivalences in several model structure on $\cS^{\cI}$. Since we shall eventually consider commutative monoids in $\cS^{\cI}$, it will be appropriate for our purposes to consider a ``positive'' model structure on $\cS^{\cI}$. We say that a map of $\cI$-spaces $X\to Y$ is a 
\begin{itemize}
\item \emph{positive $\cI$-fibration} if for all $n\geq 1$ the map $X(\bld n)\to Y(\bld n)$ is a fibration and the inclusion $\bld n\to \bld{n+1}$ induces a homotopy cartesian square
\[ \xymatrix@-1pc{X(\bld{n}) \ar[r] \ar[d] & X(\bld{n+1}) \ar[d] \\
Y(\bld{n}) \ar[r] & Y(\bld{n+1}); }\]
\item
\emph{positive $\cI$-cofibration} if it has the left lifting property with respect to maps 
of $\cI$-spaces $U\to V$ such that $U(\bld n)\to V(\bld n)$ is an acyclic fibration for
$n\geq 1$.
\end{itemize} 
 
\begin{proposition}[{\cite[Proposition~3.2]{Sagave-S_diagram}}]
  The $\cI$-equivalences, positive $\cI$-fibrations, and positive
  $\cI$-cofibrations comprise a model structure on $\cS^{\cI}$.\qed
\end{proposition}

We refer to this model structure as the \emph{positive $\cI$-model
  structure} (omitting the additional attribute ``projective'' used for it
in \cite{Sagave-S_diagram}). By definition, an $\cI$-space $X$ is positive
$\cI$-fibrant if and only if for all $n\geq 1$ the spaces $X(\bld n)$ are Kan complexes and the maps $ X(\bld{n}) \to X(\bld{n+1})$ weak equivalences. The positive 
$\cI$-cofibrant objects have an explicit description in terms of latching spaces: The $n$th
latching space of an $\cI$-space $X$ is defined by 
$L_{\bld{n}}(X) =\colim_{(\bld{m}\to\bld{n})\in\partial(\cI\downarrow\bld{n})}X(\bld{m})$, where $\partial(\cI\downarrow\bld{n})$ 
is the full subcategory of the comma category $(\cI\downarrow\bld{n})$ 
with objects the non-isomorphisms.  It follows from 
\cite[Proposition~6.8]{Sagave-S_diagram} that 
$X$ is positive $\cI$-cofibrant if and only if $X(\bld{0})=\emptyset$, the canonical map $L_{\bld{n}}(X) \to X(\bld{n})$ is a cofibration for $n\geq 1$, and the symmetric group $\Sigma_n$ acts freely on the complement of the image of this map.

\begin{remark}\label{rem:absolute-model-structure}
There also is an ``absolute'' $\cI$-model structure on $\cS^{\cI}$ in which the weak equivalences are again the $\cI$-equivalences, but where the requirement for a map to be a fibration has been strengthened to hold at all levels 
(see \cite[Section~3]{Sagave-S_diagram}).
This model structure is Quillen equivalent to the positive $\cI$-model structure, but contrary to the latter, it does not lift to a model structure on commutative monoids. 
\end{remark}

It is proved in \cite[Theorem~3.3]{Sagave-S_diagram}
that the adjunction 
$\colim_{\cI} \colon \cS^{\cI} \rightleftarrows \cS : \! \const_{\cI}$   
defines a Quillen equivalence with respect to the positive $\cI$-model structure on $S^{\cI}$ and the standard model structure on $\cS$. The homotopy colimit of an $\cI$-space $X$ represents the ``derived'' colimit, and we view $X_{h\cI}$ as the underlying space of $X$. 

\begin{lemma}\label{lem:hI-of-homotopy-constant}
  If the $\cI$-space $X$ is positive $\cI$-fibrant, then the  map $X(\bld{k}) \to X_{h\cI}$ (induced by $\{\bld k\}\to\cI$) is a weak equivalence for $k \geq 1$.
\end{lemma}
\begin{proof}
  Let $\cI_{\geq 1}$ be the full subcategory of $\cI$ with objects
  $\{\bld{n}\,|\,n\geq 1\}$. It is easy to see that the inclusion $\iota
  \colon \cI_{\geq 1} \to \cI$ is homotopy cofinal so that
  $(\iota^*X)_{h\cI_{\geq 1}} \to X_{h\cI}$ is a weak equivalence.
  Because $B \cI_{\geq 1}$ is contractible,~\cite[Lemma
  IV.5.7]{Goerss-J_simplicial} implies the claim.
\end{proof}

\subsection{Semistable \texorpdfstring{$\cI$}{I}-spaces}
\label{subsec:semistability}
Let $\cN$ be the subcategory of $\cI$ whose morphisms are the subset
inclusions. Thus, $\cN$ may be identified with the ordered set of
non-negative integers. We say that $X\to Y$ is an \emph{
  $\cN$-equivalence} if the induced map $X_{h\cN}\to Y_{h\cN}$ is a
weak homotopy equivalence, where $(-)_{h\cN}$ is the homotopy colimit
over $\cN$.  The following important
observation is due to J. Smith and first appeared in~\cite[Proposition 2.2.9]{Shipley_THH}.

\begin{proposition}
  A map of $\cI$-spaces which is an $\cN$-equivalence is also an
  $\cI$-equivalence.
\end{proposition}  
\begin{proof}
  Let $\omega$ be the set of natural numbers and let $\cI_{\omega}$ be
  the category $\cI$ adjoint the additional object $\omega$ and with morphisms 
  the injective maps. We write
  $M=\cI_{\omega}(\omega,\omega)$ for the endomorphism monoid of
  $\omega$ and view it as a full subcategory of $\cI_{\omega}$.

  For an $\cI$-space $X$, let $L_hX$ be its homotopy left Kan
  extension along the inclusion $\cI \to \cI_{\omega}$. Then $M$ acts
  on $L_hX(\omega)$ and the homotopy cofinality arguments given in the
  proof of~\cite[Proposition 2.2.9]{Shipley_THH} imply that there are
  natural weak equivalences
    \[ X_{h\cI} \simeq (L_h X)_{h\cI_{\omega}} \simeq (L_h X(\omega))_{hM} \quad
  \textrm{ and } \quad (L_h X)(\omega) \simeq X_{h\cN}.\] So if $X \to
  Y$ is an $\cN$-equivalence, then $(L_h X)(\omega) \to (L_hY)(\omega)$ is
  a weak equivalence and hence $X_{h\cI} \to Y_{h\cI}$ is a weak
  equivalence.
\end{proof}

\begin{definition}\label{def:semistability}
  An $\cI$-space $X$ is \emph{semistable} if any $\cI$-equivalence $X\to X'$ with
  $X'$ positive $\cI$-fibrant is an $\cN$-equivalence.
\end{definition}

\begin{remark}
 The $\cN$-equivalences can be viewed as the $\cI$-space analogues of
the $\pi_*$-isomorphisms of symmetric spectra. From this point of view, the above definition is the $\cI$-space analogue of semistability for
symmetric spectrum (see \cite[Section~5.6]{HSS}).
\end{remark}

As we shall see below, there are several equivalent formulations of
semistability. Consider the functor $\eins \sqcup(-)\colon \cI\to\cI$
that takes $\mathbf n$ to $\eins\sqcup \mathbf n$, and let $R$ be the
induced functor
\begin{equation}\label{eq:functor-R}
R\colon \cS^{\cI} \to \cS^{\cI},\quad RX=X(\eins\sqcup(-)).
\end{equation}
This is analogous of the functor $R$ for symmetric spectra  from~\cite[Section 3.1]{HSS}. 
\begin{lemma}
  The functor $R\colon \cS^{\cI}\to\cS^{\cI}$ preserves
  $\cN$-equivalences.
\end{lemma}
\begin{proof}
  The restriction of $RX$ to an $\cN$-diagram may be identified
  with the restriction of $X$ to the category $\cN_{\geq 1}$ of
  positive integers. By cofinality we therefore have a weak homotopy
  equivalence $RX_{h\cN}\to X_{h\cN}$ which is natural when we view both
  sides as functors from $\cS^{\cI}$ to $\cS$. This implies the result.
\end{proof} 
The weak equivalence $RX_{h\cN}\to X_{h\cN}$ in the above proof is not
induced by a map of $\cI$-spaces; hence it does not follow that $RX$
and $X$ are $\cN$-equivalent in general. 

Let $j_X\colon X\to RX$ be
the map of $\cI$-spaces induced from the morphisms $\mathbf n\to\mathbf
1\sqcup\mathbf n$ and let $R^{\infty}X$ be the homotopy colimit of the
sequence of $\cI$-spaces
\[
X\xrightarrow{j_X}RX\xrightarrow{R(j_X)}R^2X\to\dots \to R^kX\xrightarrow{R^k(j_X)} R^{k+1}X \to \dots.
\]
Explicitly, $R^kX(\bld{n})=X(\bld{k}\concat\bld{n})$ and the level
maps of $R^k(j_X)$ are the maps
\[
R^k(j_X)(\bld n)\colon X(\bld{k}\concat\bld{n})\to X(\bld{1}\concat\bld{k}\concat\bld{n})
\]
induced by the morphisms $\mathbf k\sqcup\mathbf n\to\eins\sqcup\mathbf k
\sqcup\mathbf n$.

\begin{proposition}\label{semistabilityprop}
The following conditions on an $\cI$-space $X$ are equivalent.
\begin{enumerate}[(i)]
\item
The canonical map $X_{h\cN}\to X_{h\cI}$ is a weak equivalence.
\item
$X$ is semistable.
\item
The map $j_X\colon X\to RX$ is an $\cN$-equivalence.
\item The map $X\to R^{\infty}X$ is an $\cN$-equivalence and the structure maps of $R^{\infty}X$ are weak equivalences.
\end{enumerate}
\end{proposition}

\begin{proof}
  To see that (i) implies (ii), suppose that $X\to X'$ is an
  $\cI$-equivalence with $X'$ positive $\cI$-fibrant and consider the
  commutative diagram
  \[
  \xymatrix@-1pc{ X_{h\cN} \ar[d] \ar[r] & X_{h\cI}  \ar[d] \\
    X'_{h\cN} \ar[r] & X'_{h\cI}.}
  \]
  The upper horizontal map is a weak equivalence by assumption. The
  bottom horizontal map is a weak equivalence since the fibrancy
  condition on $X'$ implies that both spaces are weakly equivalent to
  $X'(\bld{1})$. Therefore the vertical map on the left is a weak
  equivalence if and only if the map on the right is. 

  Next we show that (ii) implies (iii). Choose an $\cI$-equivalence
  $X\to X'$ with $X'$ positive $\cI$-fibrant (a fibrant replacement)
  and consider the diagram
  \[
  \xymatrix@-1pc{ X_{h\cN} \ar[d] \ar[r] &   RX_{h\cN} \ar[d] \\
    X'_{h\cN}\ar[r] & RX'_{h\cN}. }
  \]
  Here the vertical maps are weak equivalences since $X$ is assumed to
  be semistable and $R$ preserves $\cN$-equivalences. The horizontal
  map on the bottom is induced by a positive levelwise weak
  equivalence, hence is itself a weak equivalence and (iii) follows.
 
  In order to show that (iii) implies (iv) we first observe that there
  are commutative diagrams in $\cI$ of the form
  \[
  \xymatrix@-.5pc{ \mathbf n \ar[r]^-{i} \ar[d]_{\sigma_n} & \mathbf
    {n+1}\ar[d]_{\sigma_{n+1}} \ar[r]^-{i} &
    \mathbf {n+2} \ar[r]^-{i} \ar[d]_{\sigma_{n+2}}& \dots \\
    \mathbf n \ar[r]^-{j} & \mathbf {n+1} \ar[r]^-{j} & \mathbf {n+2}
    \ar[r]^-{j}
    & \dots \\
  }\]
  where $i$ denote the subset inclusions, the maps $j$ are defined by
  $s\mapsto s+1$, and $\sigma_k$ is the permutation of $\mathbf k$
  that maps $s$ to $k+1-s$. By definition, $(R^nX)_{h\cN}$ is the
  homotopy colimit of the $\cN$-diagram obtained by evaluating $X$ on
  the upper sequence and $(R^{\infty}X)(\bld{n})$ is the homotopy colimit of
  the $\cN$-diagram obtained by evaluating $X$ on the bottom sequence.
  It follows that there is a commutative diagram of spaces for each
  $n$,
  \[
  \xymatrix@-1pc{(R^nX)_{h\cN} \ar[d]_{j} \ar[r]^-{\sigma} &
    (R^{\infty}X)(\bld{n}) \ar[d]^{i} \\
    (R^{n+1}X)_{h\cN}\ar[r]^-{\sigma} & (R^{\infty}X)(\bld{n+1})}\]
  where the horizontal maps are isomorphisms. The condition in (iii)
  therefore implies that the structure maps of $R^{\infty}X$ are weak
  equivalences.  We also observe that $(R^{\infty}X)_{h\cN}$
  may be identified with the homotopy colimit of the $\cN\times
  \cN$-diagram $(\mathbf m,\mathbf n)\mapsto R^mX(\bld{n})$. Evaluating the
  homotopy colimits in the $n$-variable first we get the $\cN$-diagram
  $\mathbf m\mapsto (R^mX)_{h\cN}$ and there is a commutative diagram
  of spaces
  \[
  \xymatrix@-1pc{
    & X_{h\cN}\ar[dl] \ar[dr] &\\
    {\displaystyle\hocolim_{m}}R^mX_{h\cN} \ar[rr]^{\sim}& &
    R^{\infty}X_{h\cN}}\]
  where the horizontal map is an isomorphism. The assumption in (iii)
  implies that $\mathbf m\mapsto (R^mX)_{h\cN}$ is a diagram of weak
  homotopy equivalences and (iv) follows. 

  Finally, assuming (iv) there is a commutative diagram
  \[
  \xymatrix@-1pc{
    X_{h\cN} \ar[r]\ar[d]&  X_{h\cI}\ar[d]\\
    R^{\infty}X_{h\cN}\ar[r] & R^{\infty}X_{h\cI} }
  \]
  where the vertical maps are weak equivalences by assumption and the
  bottom horizontal map is a weak equivalence since the structure maps of 
  $R^{\infty}X$ are weak equivalences. Therefore, (iv) implies~(i).
\end{proof}
\begin{remark}
Referring to the ``absolute'' $\cI$-model structure discussed in 
Remark~\ref{rem:absolute-model-structure}, 
one will get an equivalent definition of semistability and the same conclusions as in Proposition~\ref{semistabilityprop} if one requires the $\cI$-space $X'$ in 
Definition~\ref{def:semistability} to be absolute $\cI$-fibrant instead of merely positive 
$\cI$-fibrant. 
\end{remark}

The previous proposition allows us to give a quick proof of B\"okstedt's approximation lemma for homotopy colimits over $\cI$.

\begin{corollary}
Let $X$ be an $\cI$-space and suppose that there exists an unbounded,
non-decreasing sequence of integers $\{\lambda_k | k \geq0\}$ such that any morphism $\bld{m} \to \bld{n}$ in $\cI$ with $m \geq k$ induces a 
$\lambda_k$-connected map $X(\bld{m})\to X(\bld{n})$. Then the canonical map $X(\bld n)\to X_{h\cI}$ is $\lambda_n$-connected for all $n\geq 0$.
\end{corollary}
\begin{proof}
The stated conditions on $X$ implies that $X$ satisfies the criterion (iii) in Proposition \ref{semistabilityprop}, hence is semistable. By the same proposition this in turn implies that $X_{h\cN}\to X_{h\cI}$ is a weak equivalence. The corollary follows from this since $X(\bld n)\to X_{h\cN}$ is clearly $\lambda_n$-connected.  
\end{proof}

\begin{remark}
An $\cI$-space satisfying the condition in the above corollary is said to be \emph{convergent}.  This condition played an
important role in B\"okstedt's original definition of topological
Hochschild homology~\cite{Boekstedt_THH}. The fact that 
$X_{h\cN}\to X_{h\cI}$ is a weak equivalence for $X$ semistable can be viewed as a generalization of B\"okstedt's approximation lemma. 
One of the reasons why the semistability condition is convenient is
that it is preserved under many standard operations on $\cI$-spaces. We shall see some examples of this in the following.
\end{remark}

If $X_{\bullet}$ is a simplicial $\cI$-space, we define its
realization $|X|$ to be the $\cI$-space with $|X|(\bld{n}) = \diag
X_{\bullet}(\bld{n})$. 

\begin{proposition}\label{prop:semistable-realization}
  Let $X_{\bullet}$ be a simplicial $\cI$-space which is semistable in
  each simplicial degree. Then the realization $|X_{\bullet}|$ is also
  semistable.
\end{proposition}
\begin{proof}
  Since the Bousfield-Kan map $\hocolim_{\Delta^{\op}} X_{\bullet} \to
  |X|$ is a level equivalence of $\cI$-spaces by~\cite[Corollary
  18.7.5]{Hirschhorn_model}, this follows by commuting homotopy
  colimits using Proposition \ref{semistabilityprop}(i).
\end{proof}

\subsection{Flat \texorpdfstring{$\cI$}{I}-spaces }
Recall from Section \ref{sec:I-space-model-str} the latching
maps $L_{\bld{n}}(X) \to X(\bld{n})$ associated with an $\cI$-space $X$.
The below flatness condition is the $\cI$-space 
analogue of the $S$-cofibrant symmetric spectra introduced in~\cite[Definition 5.3.6]{HSS} (which are called 
\emph{flat} symmetric spectra in~\cite{Schwede_SymSp}). 
\begin{definition}
An $\cI$-space $X$ is \emph{flat} if the map $L_{\bld{n}}(X) \to X(\bld{n})$ is a cofibration of the underlying (non-equivariant) spaces for every $n\geq 0$.
\end{definition}
It is clear from the definition that every positive $\cI$-cofibrant $\cI$-space is flat. 

\begin{remark} It is a consequence of \cite[Proposition~3.10]{Sagave-S_diagram} 
that the flat $\cI$-spaces are the cofibrant objects in a
  \emph{flat model structure} on $\cS^{\cI}$ whose weak
  equivalences are the $\cI$-equivalences. Although some results
  from~\cite{Sagave-S_diagram} proven using the flat model structure are
  crucial ingredients for the present paper, we only need to
  consider the flat $\cI$-spaces here and refer
  to~\cite{Sagave-S_diagram} for details about the flat model
  structure.
\end{remark}

The following explicit flatness criterion from \cite{Sagave-S_diagram} is often convenient. 
\begin{proposition}[{\cite[Proposition~3.11]{Sagave-S_diagram}}]\label{prop:explicit-flatness}
An $\cI$-space $X$ is flat if and only if each morphism $\bld m\to \bld n$ induces a cofibration $X(\bld m)\to X(\bld n)$ and for each diagram of the following form (with maps induced by the evident order preserving morphisms)
\begin{equation}\label{eq:flat-criterion}
\xymatrix@-1pc{
X(\bld m)\ar[r]\ar[d] & X(\bld m\sqcup\bld n)\ar[d]\\
X(\bld l\sqcup\bld m)\ar[r] & X(\bld l\sqcup\bld m\sqcup\bld n)
}
\end{equation}
the intersection of the images of $X(\bld l\sqcup\bld m)$ and 
$X(\bld m\sqcup\bld n)$ in $X(\bld l\sqcup\bld m\sqcup\bld n)$ equals the image of $X(\bld m)$. \qed
\end{proposition}

There is a further characterization of flat $\cI$-spaces which is analogous to the characterization of flat symmetric spectra in~\cite{Schwede_SymSp}.  We say that a map of $\cI$-spaces $X \to Y$ is a \emph{level
  cofibration} if $X(\bld{n}) \to
Y(\bld{n})$ is a cofibration for every object $\bld{n}$.

\begin{lemma}\label{lem:flat-characterization}
An $\cI$-space $X$ is flat if and only if the functor $X\boxtimes(-)$ preserves
level cofibrations of $\cI$-spaces. 
\end{lemma}
\begin{proof}
  Let $G_{\bld{m}}^{\cI}\colon \cS^{\Sigma_{m}} \to \cS^{\cI}$ be the
  left adjoint of the functor that evaluates an $\cI$-space at $\bld{m}$. 
  By definition (see \cite[Section 6]{Sagave-S_diagram}), a \emph{flat cell complex} is a transfinite composition of a sequence of maps with initial term $\emptyset$ and maps obtained by cobase changes from maps of the form $G_{\bld{m}}^{\cI}(K \to L)$ with $K \to L$ a map of $\Sigma_m$-spaces whose underlying map of spaces is a cofibration.
  As a consequence of the flat model structure~\cite[Proposition~6.7]{Sagave-S_diagram}, we know that every
  flat $\cI$-space is a retract of a flat cell complex.  For a
  $\Sigma_m$-space $L$, \cite[Lemma~5.6]{Sagave-S_diagram} 
  implies that there is a canonical isomorphism
  \[ (G_{\bld{m}}^{\cI}(L) \boxtimes Y)(\bld{m}\sqcup \bld {n}) \iso
  {\Sigma_{m+n}}\times_{(\Sigma_{m}\times\Sigma_{n})} (L \times Y(\bld{n})),\]
  and that $(G_{\bld{m}}^{\cI}(L)\boxtimes Y)(\bld{k})=\emptyset$ if $k < m$.  It
  follows that $G_{\bld{m}}^{\cI}(L) \boxtimes(-)$ preserves level
  cofibrations.

Next suppose that $W$ is an $\cI$-space such that $W\boxtimes(-)$ preserves level cofibrations, let $K\to L$ be a map of $\Sigma_m$-spaces whose underlying map of spaces is a cofibration, and let 
$G_{\bld{m}}^{\cI}(K) \to W$ be a map of $\cI$-spaces. 
Analyzing the pushout of the diagram 
\[
G_{\bld{m}}^{\cI}(L)\boxtimes Y \ot
  G_{\bld{m}}^{\cI}(K)\boxtimes Y \to W\boxtimes Y
\] 
using the above description, we see that
  $(G_{\bld{m}}^{\cI}(L) \cup_{G_{\bld{m}}^{\cI}(K)}W)\boxtimes (-)$
  preserves level cofibrations. Now it follows from an inductive argument that 
$X\boxtimes(-)$ preserves level cofibrations whenever $X$ is a flat cell complex and  since level cofibrations are preserved
under retracts, the same holds for all flat $\cI$-spaces.
    
For the other implication, assume that $X$ is an $\cI$-space such that 
$X\boxtimes(-)$ preserves level cofibrations.     
Let $\cI(\bld{0},-)$ be the monoidal unit in $\cI$-spaces (it equals the
terminal $\cI$-space since $\bld{0}$ is initial) and let
  $\ovl{\cI(\bld{0},-)}$ be $\cI$-space obtained by replacing the
  value of $\cI(\bld{0},-)$ at $\bld{0}$ by the empty space. Then
   $\ovl{\cI(\bld{0},-)} \to \cI(\bld{0},-)$ is a level
  cofibration and evaluating the induced map $X\boxtimes\ovl{\cI(\bld{0},-)} \to
 X\boxtimes \cI(\bld{0},-)$ at $\bld{n}$ we get the map $L_{\bld{n}}(X)
  \to X(\bld{n})$ which is therefore a cofibration.
\end{proof}

We record some useful properties of flat $\cI$-spaces.

\begin{proposition}
\label{prop:flat-boxtimes-preserves-I-and-N-equiv}
  If $X$ is a flat $\cI$-space, then the functor $X\boxtimes(-)$ preserves $\cI$-equivalences and $\cN$-equivalences.
\end{proposition}
\begin{proof}
The statement for $\cI$-equivalences is \cite[Proposition~8.2]{Sagave-S_diagram}. 
For the statement about $\cN$-equivalences, one proceeds as in the proof of Lemma~\ref{lem:flat-characterization} and considers first a flat $\cI$-space of the form $G_{\bld m}^{\cI}(L)$ for a $\Sigma_m$-space $L$. Using the description from the proof of that lemma, one sees that the $\cN$-space underlying $G_{\bld{m}}^{\cI}(L)\boxtimes Y$ decomposes as a coproduct of 
$\cN$-spaces with summands isomorphic to $L\times Y$ after appropriate shifts. Hence $G_{\bld{m}}^{\cI}(L)\boxtimes(-)$ preserves 
$\cN$-equivalences. By an inductive argument, using the $\cN$-analogue of 
\cite[Proposition~7.1]{Sagave-S_diagram}, this implies that $X\boxtimes(-)$ preserves $\cN$-equivalences whenever $X$ is a flat cell complex.  Since $\cN$-equivalences are preserved under retracts, this in turn implies the result for all flat 
$\cI$-spaces. 
\end{proof}

\begin{proposition}\label{prop:boxtimes-times-flat}
If $X$ and $Y$ are flat $\cI$-spaces, then so are $X \boxtimes Y$ and $X\times Y$. 
\end{proposition}
\begin{proof}
  Assuming that $X$ and $Y$ are flat, the pushout-product axiom for
  the flat model structure, \cite[Proposition 3.10]{Sagave-S_diagram},
  implies that $X\boxtimes Y$ is also flat. The statement for $X\times
  Y$ is an immediate consequence of Proposition
  \ref{prop:explicit-flatness}.
\end{proof}

Next we study how the functor $R$ introduced in \eqref{eq:functor-R}
and the natural transformation $j \colon X \to RX$ behave with respect to the
$\boxtimes$-product. First one checks that the maps
\[ X(\bld{k}\concat\bld{m}) \times Y(\bld{l}\concat\bld{n}) \to
(X\boxtimes Y)(\bld{k}\concat\bld{m}\concat\bld{l}\concat\bld{n})
\xrightarrow{\eins \concat \chi_{m,l}\concat \eins} (X\boxtimes
Y)(\bld{k}\concat\bld{l}\concat\bld{m}\concat\bld{n})\] induce a
natural map of $\cI$-spaces
\[ \xi^{k,l} \colon (R^k X) \boxtimes (R^l Y)  \to R^{k+l}(X \boxtimes Y).\]

\begin{lemma}\label{lem:R-vs-boxtimes-pushout}
Let $X$ and $Y$ be $\cI$-spaces. Then there is a pushout square
\begin{equation}\label{eq:R-vs-boxtimes-pushout}
  \xymatrix@-.5pc{
    X \boxtimes Y \ar[r]^{X \boxtimes j_{Y}} \ar[d]_{j_X \boxtimes Y} & 
    X \boxtimes RY \ar[d]^{\xi^{0,1}}\\
    (RX) \boxtimes Y \ar[r]^{\xi^{1,0}} & R(X\boxtimes Y) }
\end{equation}
and the composite $X\boxtimes Y \to R(X\boxtimes Y)$ equals $j_{X\boxtimes Y}$. 
\end{lemma}
\begin{proof}
This can be checked by decomposing the colimit defining $R(X\boxtimes Y)$.
\end{proof}

\begin{proposition}\label{prop:boxtimes_of_semistable}
  If $X$ and $Y$ are flat and semistable $\cI$-spaces, then
  $X\boxtimes Y$ is also semistable. 
\end{proposition}
\begin{proof}
  This is similar to the corresponding statement about symmetric
  spectra~\cite{Schwede_SymSp}. Since $X$ and $Y$ are flat, 
  $j_X$ and $j_Y$ are level cofibrations. Hence the maps $X \boxtimes
  j_Y$ and $j_X \boxtimes Y$ in \eqref{eq:R-vs-boxtimes-pushout}
  are level cofibrations by Lemma~\ref{lem:flat-characterization} and $\cN$-equivalences by Proposition \ref{prop:flat-boxtimes-preserves-I-and-N-equiv}.
Since the homotopy colimit functor over $\cN$ takes level cofibrations to cofibrations and $\cN$-equivalences to weak equivalences, the cobase changes of these maps are also $\cN$-equivalences. 
The claim therefore follows from Lemma~\ref{lem:R-vs-boxtimes-pushout}.
\end{proof}

\subsection{Comparison of the cartesian and the 
\texorpdfstring{$\boxtimes$}{boxtimes}-product}
Let $X$ and $Y$ be $\cI$-spaces and consider the natural transformation 
\[
\mu_{X,Y} \colon X_{h\cI} \times Y_{h\cI} \xrightarrow{\iso} (X \times
Y)_{h(\cI\times\cI)} \to ((-\concat -)^* (X\boxtimes
Y))_{h(\cI\times\cI)} \to (X\boxtimes Y)_{h\cI}
\]
where the second map is induced by the universal natural transformation
of $\cI\times \cI$-diagrams $X(\bld{m}) \times Y(\bld{n}) \to (X\boxtimes Y)(\bld{m}\concat\bld{n})$. These maps gives rise to a monoidal structure on the functor $(-)_{h\cI}$, cf.\ \cite[Proposition 4.17]{Schlichtkrull_Thom-symm}.
 
 \begin{lemma}\label{lem:times-boxtimes}
If one of $X$ and $Y$ is flat, then $\mu_{X,Y}$ is a weak equivalence.
\end{lemma}
\begin{proof}
We may assume without loss of generality that $Y$ is flat. 
As functors of $X$, both the domain and the codomain of $\mu_{X,Y}$ then take 
$\cI$-equivalences to weak equivalences by Proposition \ref{prop:flat-boxtimes-preserves-I-and-N-equiv}. Choosing a cofibrant replacement of $X$ in the positive $\cI$-model structure, it therefore suffices to prove the proposition when $X$ is cofibrant. 
Furthermore, since a cofibrant $\cI$-space is also flat, we may repeat the argument and thereby reduce to the case where both $X$ and $Y$ are cofibrant. Then $X\boxtimes Y$ is also cofibrant because the positive $\cI$-model structure on $\cS^{\cI}$ is monoidal by \cite[Proposition~3.2]{Sagave-S_diagram}. It is proved in 
\cite[Lemma 6.22]{Sagave-S_diagram} that for a cofibrant
$\cI$-space $Z$, the natural map $\hocolim_{\cI}Z \to \colim_{\cI}Z$ is a
weak equivalence. Hence the claim in the lemma follows because the colimit version of the map $\mu_{X, Y}$ is an isomorphism (that is, the colimit functor is strong symmetric monoidal, cf.\  \cite[Lemma~8.8]{Blumberg-C-S_THH-Thom} and the discussion following that lemma).
\end{proof}

\begin{lemma}\label{lem:hI-of-product}
  Let $X$ and $Y$ be semistable $\cI$-spaces. Then the diagonal functor 
  $\cI\to \cI \times \cI$ induces a weak equivalence $(X\times Y)_{h\cI}
  \to X_{h\cI} \times Y_{h\cI}$.
\end{lemma}
\begin{proof}
The assumption that $X$ and $Y$ are semistable
implies that $X\times Y$ is semistable and consequently that the
horizontal maps in the diagram
\[
\xymatrix@-1pc{
(X\times Y)_{h\cN} \ar[r] \ar[d] &  (X\times Y)_{h\cI}\ar[d]\\
X_{h\cN}\times Y_{h\cN}\ar[r] & X_{h\cI}\times Y_{h\cI}
}
\]
are weak equivalences. The vertical map on the left is a weak
equivalence since the diagonal inclusion $\cN\to \cN\times \cN$ is
homotopy cofinal and the conclusion follows.
\end{proof}

Since the terminal object $*$ of $\cS^{\cI}$ is also the monoidal unit for
$\boxtimes$, the projections $X \to *$ and $Y \to *$ induce a map
$\rho_{X,Y} \colon X\boxtimes Y \to X \times Y$.
\begin{proposition}\label{prop:boxtimes-times}
  If $X$ and $Y$ are semistable and one of $X$ and $Y$ is flat, then 
  $\rho_{X,Y}\colon X\boxtimes Y\to X\times Y$ is an $\cI$-equivalence.
\end{proposition} 
\begin{proof}
  The monoidal structure map $\mu_{X,Y}$ fits into a commutative diagram
  \[\xymatrix@-1pc{X_{h\cI} \times B\cI \ar[dd]_{\mu_{X,*}} & X_{h\cI}
    \times Y_{h\cI} \ar[d]_{\mu_{X,Y}} \ar[l] \ar[r] &
    B\cI \times Y_{h\cI}  \ar[dd]^{\mu_{*,Y}}\\
    & (X \boxtimes Y)_{h\cI} \ar[d]_{(\rho_{X,Y})_{h\cI}} \\
    X_{h\cI} & (X \times Y)_{h\cI}  \ar[l]
    \ar[r]&Y_{h\cI}}
\]
in which the horizontal maps are induced by $X \to *$ and $Y \to *$.
Since the diagonal composite in both outer squares is homotopic to the
respective projection,
\begin{equation}\label{eq:mu-eta-composition} 
X_{h\cI} \times Y_{h\cI} \xrightarrow{\mu_{X,Y}} (X \boxtimes
Y)_{h\cI} \xrightarrow{(\rho_{X,Y})_{h\cI}} (X \times Y)_{h\cI} \to
X_{h\cI} \times Y_{h\cI}
\end{equation} 
is homotopic to the identity. The claim now
follows by Lemmas \ref{lem:times-boxtimes} and
\ref{lem:hI-of-product}.
\end{proof} 

\begin{remark}\label{rem:semistablility-hypoth}
  The statement in Proposition \ref{prop:boxtimes-times} does not remain true
  in general without the semistability hypothesis. For instance, given a based space $X$, there is an associated $\cI$-space $X^{\bullet}$ with 
  $X^{\bullet}(\bld n)=X^n$. This is flat but usually not semistable, and in fact it follows from~\cite{Schlichtkrull_infinite} that for a pair of based connected spaces $X$ and $Y$ there are weak equivalences  
\[
(X^{\bullet} \boxtimes Y^{\bullet})_{h\cI} 
\simeq \Omega^{\infty}\Sigma^{\infty}(X \wdg
  Y) \quad \textrm{ and } \quad (X^{\bullet} \times
  Y^{\bullet})_{h\cI} \simeq \Omega^{\infty}\Sigma^{\infty}(X \times Y).
\]
That Proposition~ \ref{prop:boxtimes-times} does not hold in general is related to the fact that a cartesian product of $\cI$-equivalences is not necessarily an $\cI$-equivalence.
\end{remark}

Finally we observe that the argument in the proof of 
Proposition~\ref{prop:boxtimes-times} shows that the composition of the last two maps in \eqref{eq:mu-eta-composition} is a homotopy left inverse of 
$\mu_{X,Y}$. Hence Lemma~\ref{lem:times-boxtimes} has the following corollary.

\begin{corollary}\label{cor:flat-boxtimes-times}
If one of $X$ and $Y$ is flat, then the map $(X\boxtimes Y)_{h\cI}\to 
X_{h\cI}\times Y_{h\cI}$ induced by the projections is a weak equivalence.\qed
\end{corollary}

\section{Commutative \texorpdfstring{$\cI$}{I}-space
  monoids}\label{sec:I-space-monoids}
By definition, a \emph{commutative $\cI$-space monoid} is a commutative monoid in $\cS^{\cI}$ with respect to the $\boxtimes$-product. 
We write $\cC\cS^{\cI}$ for the category of commutative $\cI$-space monoids.
Unraveling the definitions, a commutative $\cI$-space monoid $A$ is an $\cI$-space $A$ together with a unit element in $A(\bld{0})$ and
a natural transformation of ($\cI\times\cI$)-diagrams  $A(\bld{m})\times A(\bld{n}) \to A(\bld{m+n})$ that is associative and unital in the appropriate sense and which makes the diagrams 
\[
\xymatrix@-1pc{
A(\bld m)\times A(\bld n) \ar[r] \ar[d]& A(\bld m\sqcup \bld n)\ar[d]^{\tau_{m,n}}\\
A(\bld n)\times A(\bld m) \ar[r] & A(\bld n\sqcup \bld m)
}
\] 
commutative.
The next result is the main reason for considering the positive $\cI$-model structure on $\cS^{\cI}$.
\begin{proposition}[{\cite[Proposition~3.5]{Sagave-S_diagram}}]\label{prop:I-positive-lift}
The positive $\cI$-model structure on $\cS^{\cI}$ lifts to a proper model structure on $\cC\cS^{\cI}$ in which a map is a weak equivalence or fibration if and only if the underlying map of $\cI$-spaces is. 
\qed
\end{proposition}
We shall also refer to this as the \emph{positive $\cI$-model} structure on $\cC\cS^{\cI}$. It is proved in \cite[Theorem~3.6]{Sagave-S_diagram} that this model structure makes $\cC\cS^{\cI}$ Quillen equivalent to the category of $E_{\infty}$ spaces (for any choice of $E_{\infty}$ operad) and one may think of commutative $\cI$-space monoids as strictly commutative models of $E_{\infty}$ spaces. 

Together with 
Proposition~\ref{prop:flat-boxtimes-preserves-I-and-N-equiv}, the next result ensures that cofibrant commutative $\cI$-space monoids are homotopically well-behaved with respect to the $\boxtimes$-product. 

\begin{proposition}\label{prop:underlying-flat-com-I}
  If a commutative $\cI$-space monoid $A$ is cofibrant in the positive
  $\cI$-model structure, then its underlying $\cI$-space is flat.\qed
\end{proposition}
\begin{proof}
  In~\cite[Proposition~3.15(i)]{Sagave-S_diagram} we establish a
  \emph{positive flat $\cI$-model structure} on $\cC\cS^{\cI}$. It follows
  from~\cite[Proposition~6.20]{Sagave-S_diagram} that $A$ is also
  cofibrant in this model structure. Hence its underlying $\cI$-space
  is flat by~\cite[Proposition~3.15(ii)]{Sagave-S_diagram}.
\end{proof}

\subsection{The simplicial structure on 
\texorpdfstring{$\cC\cS^{\cI}$}{\cC\cS{\cI}}}
The category of $\cI$-spaces is enriched, tensored and cotensored  over simplicial sets. The simplicial mapping spaces are defined by 
\[
\Map(X,Y) \iso \int_{\bld{n}\in\cI}\Map(X(\bld{n}),Y(\bld{n})) \iso
\left\{ [k] \mapsto \cS^{\cI}(X \times \Delta^{k}, Y)\right\},
\]
while for an $\cI$-space $X$ and a simplicial set $K$, the
tensor $X \times K$ and cotensor $X^K$ are the $\cI$-spaces defined by
\[
(X\times K)(\bld{n}) =X(\bld{n})\times K\quad\textrm{ and }\quad 
X^K(\bld n)=\Map(K,X(\bld n)).
\]  
By \cite[Proposition~3.2]{Sagave-S_diagram}, the 
positive $\cI$-model structure on $\cS^{\cI}$ makes the latter a simplicial model category

The category $\cC\cS^{\cI}$ is again enriched, tensored, and cotensored over simplicial sets. While the cotensor is defined on the underlying $\cI$-spaces, the tensor and simplicial mapping spaces are defined respectively by
\[
A \tensor K = |[m] \mapsto A^{\boxtimes K_m}| \quad \textrm{
  and } \quad \Map(A,B) = \left\{[k]\mapsto \cC\cS^{\cI}(A \tensor \Delta^{k}, B) \right\}
\]
for $A$ and $B$ in $\cC\cS^{\cI}$, $K$ in $\cS$, and
$\geor{-}$ the usual restriction to the simplicial diagonal. Since the condition for being a simplicial model category can be expressed in terms of the cotensor structure, 
see e.g.\ \cite[Lemma 4.2.2]{Hovey_model}, the fact that the positive model structure on $\cS^{\cI}$ is simplicial implies that the same holds for $\cC\cS^{\cI}$.
\begin{proposition}
The category $\cC\cS^{\cI}$ is enriched, tensored, and cotensored over 
$\cS$, and the positive $\cI$-model structure is a simplicial model structure.\qed
\end{proposition}

We next observe that the monoidal unit for the $\boxtimes$-product can be identified with both the initial object and the terminal object $*$ in 
$\cC\cS^{\cI}$, so that the latter is a based category. Hence the simplicial mapping spaces $\Map(A,B)$ are canonically based and $\cC\cS^{\cI}$ is a category enriched over the category $\cS_*$ of based simplicial sets. Given a commutative $\cI$-space monoid $A$ and a based simplicial set $(K,v)$, define the tensor $A\otimes(K,v)$ to be the pushout (in $\cC\cS^{\cI}$) of the diagram $*\leftarrow A\otimes\{v\}\to A\otimes K$, and the cotensor $A^{(K,v)}$ to be the pullback of the diagram $*\to A^{\{v\}}\leftarrow A^K$. Thus, $A^{(K,v)}(\bld n)$ is the space of based maps $\Map_*(K,A(\bld n))$.
We claim that this 
structure makes $\cC\cS^{\cI}$ a based simplicial model category. 
This means that given a cofibration 
$A\to B$ in $\cC\cS^{\cI}$ and a cofibration $(K,v)\to (L,w)$ in $\cS_*$, the pushout-product
\[
f\Box g\colon A\otimes(L,w)\boxtimes_{A\otimes(K,v)}B\otimes(K,v)\to
B\otimes(L,w)
\]
is a cofibration in $\cC\cS^{\cI}$ which is acyclic if either $f$ or $g$ is acyclic.
By adjointness (see \cite[Lemma 4.2.2]{Hovey_model}), this condition can be reformulated in terms of the mapping spaces $\Map(A,B)$ so the claim that $\cC\cS^{\cI}$ is a based simplicial model category follows from the fact that it is a simplicial model category, 
cf.\ ~\cite[Proposition 4.2.19]{Hovey_model}. We summarize the above discussion in the next corollary. 
\begin{proposition}\label{prop:based-simplicial}
  The category $\cC\cS^{\cI}$ is enriched, tensored, and cotensored over $\cS_*$, and the positive $\cI$-model structure is a based simplicial model structure.
  \qed
\end{proposition}

\subsection{The monoid of path components}\label{sec:monoid-path-comp}
Recall that for a (not necessarily fibrant) simplicial set $K$, the
set of path components $\pi_0(K)$ is defined to be the coequalizer of 
the maps $d_0, d_1\colon K_1 \rightrightarrows K_0$.

If $A$ is a commutative $\cI$-space monoid, then $A_{h\cI}$ is an
associative and homotopy commutative simplicial monoid. The
multiplication on $A_{h\cI}$ is induced from that of $A$ by means of
the monoidal structure map $\mu$ of 
Lemma~\ref{lem:times-boxtimes}. One can check the homotopy commutativity directly or deduce it from~\cite[\S 6.1]{Schlichtkrull_Thom-symm}, where it is shown that $A_{h\cI}$ is an $E_{\infty}$ space over the Barratt-Eccles operad. It follows in particular that $\pi_0(A_{h\cI})$ is a commutative monoid. Given vertices $x_0 \in A(\bld{m})_0$ and $y_0 \in A(\bld{n})_0$
representing classes $[x_0]$ and $[y_0]$ in $\pi_0(A_{h\cI})$, their
product $[x_0][y_0]$ is represented by $\mu(x_0,y_0) \in
A(\bld{m}\concat\bld{n})_0$.

\begin{example} \label{ex:free_comm_ispace-monoid} 
Let $F_{\bld n}^{\cI}\colon \cS\to \cS^{\cI}$ be the left adjoint of the functor that takes an $\cI$-space $X$ to $X(\bld n)$, and let 
$\mathbb C\colon \cS^{\cI}\to \cC\cS^{\cI}$ be the left adjoint of the forgetful functor. We define $C_1$ to be the free commutative $\cI$-space monoid on a point in $\cI$-space degree one, that is,  
\[
C_1 = \mC (F^{\cI}_{\bld{1}}(*)) = \coprod_{n \geq 0}
F_{\bld 1}^{\cI}(*)^{\boxtimes n}/ \Sigma_n.
\] 
It follows from the definition of the $\boxtimes$-product that 
$F_{\bld 1}^{\cI}(*)^{\boxtimes n}$ is isomorphic to $F^{\cI}_{\bld n}(*)=\cI(\bld n,-)$, and hence that 
 \[\big(F_{\bld{1}}^{\cI}(*)^{\boxtimes n}/ \Sigma_n\big)_{h\cI} \iso
  \left(\big(F_{\bld{1}}^{\cI}(*)^{\boxtimes
        n}\big)_{h\cI}\right)/\Sigma_n \iso B(\bld{n}\downarrow
  \cI)/\Sigma_n.
\]
Since $B(\bld{n}\downarrow\cI)$ is $\Sigma_n$-free and contractible, this in turn implies that the simplicial monoid $(C_1)_{h\cI}$ associated
with $C_1$ is weakly equivalent to $\coprod_{n\geq 0} B\Sigma_{n}$. In
particular, it has $\pi_0((C_1)_{h\cI}) \iso \mN_0$. Given a commutative 
$\cI$-space monoid $A$, it follows from the definitions that vertices in $A(\bld{1})$ correspond to maps of commutative $\cI$-space monoids $C_1 \to A$. Furthermore, if $A$ is positive $\cI$-fibrant, then we know from 
Lemma \ref{lem:hI-of-homotopy-constant} that every class in
$\pi_0(A_{h\cI})$ can be represented by a vertex $a$ in $A(\bld{1})$ 
such that the induced map $\pi_0((C_1)_{h\cI}) \to \pi_0(A_{h\cI})$ sends the generator for  $\pi_0((C_1)_{h\cI})$ to the class $[a]$.
\end{example}

\begin{definition}
A commutative $\cI$-space monoid $A$ is \emph{grouplike} if the monoid
$\pi_0(A_{h\cI})$ is a group.
\end{definition}
It is clear that if $A\to B$ is an $\cI$-equivalence of commutative $\cI$-space monoids, then $A$ is grouplike if and only if $B$ is. 
We shall later define a ``group completion'' 
$C_1\to C_1^{\mathrm{gp}}$ with the property that a positive $\cI$-fibrant commutative $\cI$-space monoid $A$ is grouplike if and only if every map $C_1\to A$ extends to $C_1^{\mathrm{gp}}$ 
(see Lemma~\ref{lem:extension-grouplike}).

\section{Group completion via the bar
  construction}\label{sec:bar-construction}
\subsection{The bar construction in \texorpdfstring{$\cI$}{I}-spaces}
In this section we examine the bar construction of
$\cI$-space monoids and its relationship to the usual space-level bar
construction. It is illuminating to consider the general setting of a
monoid $A$ in a monoidal category $(\mathcal A,\Box,1_{\cA})$ with
monoidal structure $\Box$ and unit~$1_{\cA}$. Given a right $A$-module
$M$ and a left $A$-module $N$ in $\cA$, the two-sided bar construction
is the simplicial object
\[   
B_{\bullet}(M,A,N)\colon [k]\mapsto M\Box\underbrace{A\Box\dots\Box
  A}_k\Box N
\]
with structure maps as for the usual space level bar
construction (see \cite{May_geometry}). Here we suppress an implicit choice
of placement of the parenthesis in the iterated monoidal product. Now
let $\Phi\colon \cA\to\cB$ be a monoidal functor relating the monoidal
categories $(\mathcal A,\Box,1_{\cA})$ and $(\cB,\triangle,1_{\cB})$.
By definition (see \cite[XI.\S 2]{MacLane_categories}), this means that
$\Phi$ is a functor equipped with a morphism $1_{\cB}\to
\Phi(1_{\cA})$ and a natural transformation of functors on $\cA\times
\cA$,
\[
\mu\colon \Phi(A_1)\triangle \Phi(A_2)\to \Phi(A_1\Box A_2),
\]
satisfying the usual unitality and associativity conditions (this is
what is sometimes called lax monoidal). It follows from the
definition, that if $A$ is a monoid in $\cA$, then $\Phi(A)$
inherits the structure of a monoid in $\cB$. If $M$ is a right
$A$-module in $\cA$, then $\Phi(M)$ inherits the structure of a
right $\Phi(A)$-module in $\cB$ and similarly for left
$A$-modules. Applying the functor $\Phi$ degree-wise to
$B_{\bullet}(M,A,N)$ we get a simplicial object in $\cB$ and
the monoidal structure maps give rise to a simplicial map
\[     
\mu\colon B_{\bullet}(\Phi(M),\Phi(A),\Phi(N))\to \Phi B_{\bullet}(M,A,N).
\]

Now we specialize to the monoidal category $\cS^{\cI}$ of
$\cI$-spaces. Let $A$ be an $\cI$-space monoid and notice that the
augmentation $A\to *$ makes the final object $*$ a left and
right $A$-module. We write $B_{\bullet}(A)$ for the simplicial bar
construction $B_{\bullet}(*,A,*)$ and $B(A)$ for its realization.  As
proven in \cite[Proposition 4.17]{Schlichtkrull_Thom-symm}, the
homotopy colimit functor $ (-)_{h\cI}\colon \cS^{\cI}\to \cS $
canonically has the structure of a monoidal functor, where the natural
transformation $\mu$ is the monoidal structure map of Lemma
\ref{lem:times-boxtimes}. Its value on the final object $*$ is the
classifying space $B\cI$ which is contractible since $\cI$ has
an initial object. Implementing the above discussion in the case at
hand we get a chain of simplicial maps
\[
B_{\bullet}(A)_{h\cI}\leftarrow B_{\bullet}(B\cI,A_{h\cI},B\cI)\to
B_{\bullet}(A_{h\cI})
\] 
where the right hand map is induced by the projection $B\cI\to *$.\  
(There is no simplicial map relating
$B_{\bullet}(A_{h\cI})$ and $B_{\bullet}(A)_{h\cI}$ directly.)  
\begin{proposition}\label{bar-bar}
If $A$ is an $\cI$-space monoid with underlying flat $\cI$-space, then the induced maps of realizations 
\[
B(A)_{h\cI}\xleftarrow{\sim} B(B\cI,A_{h\cI},B\cI)\xrightarrow{\sim}
B(A_{h\cI})
\]
are weak equivalences. 
\end{proposition} 
\begin{proof}
  The right hand map is the realization of a map of bisimplicial sets
  which is a weak equivalence at each simplicial degree of the bar construction and is therefore a weak
  equivalence. Lemma \ref{lem:times-boxtimes} enables us to apply the
  same argument to the left hand map.
\end{proof}

We also note the following corollary of Proposition~\ref{prop:semistable-realization}, Proposition~\ref{prop:boxtimes-times-flat}, and  
Proposition~\ref{prop:boxtimes_of_semistable}.
\begin{corollary}\label{B(A)semistable}
  If $A$ is an $\cI$-space monoid with underlying flat and semistable 
  $\cI$-space, then $B(A)$ is semistable.\qed
\end{corollary}

\subsection{The loop functor}\label{subsec:loop-functor}
Given a based space $X$, we write $\Omega(X)$ for the based simplicial mapping space $\Map_*(S^1,X)$ where $S^1$ denotes the simplicial circle $\Delta^1/\partial\Delta^1$. In order for this to represent the correct homotopy type we should of course stipulate that $X$ be fibrant, or otherwise choose a fibrant replacement. For definiteness, we write $X^{\mathrm{fib}}=\Sing|X|$ for the fibrant replacement obtained by evaluating the singular complex of the geometric realization.  

Our next aim is to understand how the loop functor $\Omega$ interacts
with the formation of homotopy colimits over $\cI$. In general, given
a small category $\cK$ and a functor $X\colon \cK \to \cS_*$ to based
spaces, we write $X_{h^*\cK}$ for the based homotopy colimit over
$\cK$. This can be defined as the quotient of the unbased homotopy
colimit $X_{h\cK}$ by the subspace $B\cK$ specified by
the inclusion of the base point in $X$ (for a full discussion, see~\cite[Proposition 18.8.4]{Hirschhorn_model}). If $B\cK$ is contractible then
the projection $X_{h\cK}\to X_{h^*\cK}$ is a weak equivalence. This
applies in particular to the categories $\cN$ and $\cI$.

Now assume that $X \colon \cK \to \cS_*$ is levelwise fibrant in positive
degrees. The advantage of the based homotopy colimit for our purpose
is that the evaluation maps $S^1 \sm \Omega(X(\bld{k})) \to X(\bld{k})$ induce
a based map
\[ 
S^1 \sm \left( \Omega X\right)_{h^*\cK} \to (S^1 \sm \Omega X)_{h^* \cK} \to X_{h^*\cK} \to (X_{h^*\cK})^{\mathrm{fib}} 
\]
with adjoint $\left( \Omega X\right)_{h^*\cK} \to \Omega
((X_{h^*\cK})^{\mathrm{fib}}).$ The levelwise fibrancy assumption and the fibrant replacement ensure that the constructions are homotopically meaningful. 
\begin{lemma}\label{NOmega}
  Let $X\colon \cI \to \cS_*$ be a based $\cI$-space that is levelwise
  fibrant in positive degrees. Then the canonical map 
  $\Omega(X)_{h^*\cN}\to \Omega((X_{h^*\cN})^{\mathrm{fib}})$ is a weak equivalence.
\end{lemma}
\begin{proof}
  Consider the filtration of $\cN$ by the subcategories $\cN_k$ of
  natural numbers less than or equal to $k$. This gives a filtration
  of $X_{h^*\cN}$ by the subspaces $X_{h^*\cN_k}$ and it follows from
  the definition of the homotopy colimit that the inclusions give rise
  to an isomorphism
$
\colim_k X_{h^*\cN_k}\to X_{h^*\cN}.
$
Applying this to $X$ and $\Omega(X)$ and using that $\Omega$ commutes
with sequential colimits of inclusions, we get a commutative diagram
\[
\xymatrix@-1pc{ \colim_{k}\Omega(X)_{h^*\cN_k}\ar[r]\ar[d]_{\iso}&
  \colim_k\Omega((X_{h^*\cN_k})^{\textrm{fib}})\ar[d]^{\iso}\\
  \Omega(X)_{h^*\cN}\ar[r]& \Omega((X_{h^*\cN})^{\textrm{fib}}) }
\]  
where the vertical maps are isomorphisms. Since $\cN_k$ has a terminal
object it is clear that the map in the upper row is a weak equivalence
for each fixed positive $k$, hence the map of colimits is also a weak
equivalence.
\end{proof}

We say that a based $\cI$-space is semistable if the underlying
$\cI$-space (forgetting the base point) is semistable.

\begin{proposition}\label{hocolimloop}
  Let $X$ be a semistable based $\cI$-space that is levelwise fibrant
  in positive degrees. Then $\Omega(X)$ is also semistable and there
  is a natural weak equivalence
\[
\Omega(X)_{h^*\cI}\to \Omega((X_{h^*\cI})^{\mathrm{fib}}).
\] 
\end{proposition}
\begin{proof}
  In order for $\Omega(X)$ to be semistable it suffices that $j\colon
  \Omega(X)\to R\Omega(X)$ induces a weak equivalence of based
  homotopy colimits over $\cN$. Since $R\Omega(X)$ is the same as
  $\Omega(RX)$, there is a commutative diagram
\[
\xymatrix@-1pc{
  \Omega(X)_{h^*\cN}\ar[r] \ar[d]_{\sim} & R\Omega(X)_{h^*\cN}\ar[d]^{\sim}\\
  \Omega((X_{h^*\cN})^{\textrm{fib}})\ar[r]& \Omega((RX_{h^*\cN})^{\mathrm{fib}}) }
\]
where the vertical maps are the weak equivalences from Lemma
\ref{NOmega}. The lower horizontal map is a weak equivalence by
assumption and the result follows. Using this the second statement in
the proposition follows from the commutative diagram
\[
\xymatrix@-1pc{
\Omega(X)_{h^*\cN}\ar[r]\ar[d] & \Omega((X_{h^*\cN})^{\textrm{fib}})\ar[d]\\
\Omega(X)_{h^*\cI} \ar[r] &\Omega((X_{h^*\cI})^{\textrm{fib}})
}
\]
where the vertical maps are weak equivalences since $X$ and
$\Omega(X)$ are semistable and the upper horizontal map is a weak
equivalence by Lemma \ref{NOmega}.
\end{proof}

\subsection{Group completion in \texorpdfstring{$\cI$}{I}-spaces}
In this section we show that the usual procedure for group completing
a simplicial monoid lifts to $\cI$-space monoids and we use this to prove 
Theorem~\ref{thm:intro-bar-completion} from the introduction.

Consider in general a simplicial monoid $M$, and recall that the classifying space $B(M)$ is isomorphic to the coend of the diagram 
$([k],[l])\mapsto M^{\times k}\times \Delta^l$ 
(see e.g.\ \cite[Chapter IV.1]{Goerss-J_simplicial}).  
The canonical map $M\times \Delta[1]\to B(M)$ induces a based map 
$M\wedge S^1\to B(M)$ and composing with the fibrant replacement 
$B(M)\to B(M)^{\mathrm{fib}}$ we get a based map 
$M\wedge S^1\to B(M)^{\mathrm{fib}}$. We define the group completion map 
\[
\eta_M\colon M\to\Omega(B(M)^{\mathrm{fib}})
\]
to be the adjoint of this map.
Likewise, given an $\cI$-space monoid $A$, the canonical map 
$A\times \Delta^1\to B(A)$ induces a map of based $\cI$-spaces 
$A\wedge S^1\to B(A)$ and applying the fibrant replacement functor 
$(-)^{\mathrm{fib}}$ levelwise to $B(A)$ we get a map of based $\cI$-spaces 
$A\wedge S^1\to B(A)^{\mathrm{fib}}$. We define the group completion map 
\begin{equation}\label{eq:fib-completion}
\eta_A\colon A\to \Omega(B(A)^{\mathrm{fib}})
\end{equation}
to be the adjoint of this map. The underlying map $(\eta_A)_{h\cI}$ can be compared to $\eta_{A_{h\cI}}$ as we now show.
\begin{proposition}\label{prop:fib-space-completion}
If $A$ is an $\cI$-space monoid with underlying flat and semistable $\cI$-space, then there is a chain of natural weak equivalences relating
$\Omega(B(A)^{\mathrm{fib}})_{h\cI}$ and 
$\Omega(B(A_{h\cI})^{\mathrm{fib}})$ such that the diagram
\[
\xymatrix@-1pc{
& A_{hI}\ar[dl]_{(\eta_A)_{h\cI}} \ar[dr]^{\eta_{A_{h\cI}}} &\\
\Omega(B(A)^{\mathrm{fib}})_{h\cI} \ar[rr]^{\sim}& & 
\Omega(B(A_{h\cI})^{\mathrm{fib}})
}
\]
is commutative in the homotopy category.
\end{proposition}
\begin{proof}
  We first consider a reduced version of the equivalences relating
  $B(A)_{h\cI}$ and $B(A_{h\cI})$ in Proposition \ref{bar-bar}. Notice
  that $B(A)_{h^*\cI}$ is the quotient of $B(A)_{h\cI}$ by the
  subspace $B_0(A)_{h\cI}=B\cI$. We similarly define a reduced version
  of $B(B\cI,A_{h\cI},B\cI)$ by letting
  $\widetilde{B}(B\cI,A_{h\cI},B\cI)$ be the quotient by the subspace
  $B\cI\times B\cI$. It follows from Proposition \ref{bar-bar} that
  there are weak equivalences
  \[
  B(A)_{h^*\cI}\xleftarrow{\sim}
  \widetilde{B}(B\cI,A_{h\cI},B\cI)\xrightarrow{\sim} B(A_{h\cI})
  \]
  which give rise to weak equivalences of the associated loop spaces
  of their fibrant replacements. The conclusion then follows from the
  commutativity of the diagram
  \[
  \xymatrix{ \Omega\left(B(A)^{\mathrm{fib}}\right)_{h\cI}
    \ar[d]^{\sim}& \ar[l]_-{\left(\eta_A\right)_{h\cI}} A_{h\cI}
    \ar[r]^-{\eta_{A_{h\cI}}} \ar[dl]
    \ar[d]\ar[dr] & \Omega\left(B(A_{h\cI})^{\textrm{fib}}\right)\\
    \Omega\left(B(A)^{\mathrm{fib}}\right)_{h^*\cI}\ar[r]^-{\sim} &
    \Omega\left(\left(\left(B(A)^{\mathrm{fib}}\right)_{h^*\cI}\right)^{\textrm{fib}}\right)
    &
    \Omega\left(\widetilde{B}(B\cI,A_{h\cI},B\cI)^{\textrm{fib}}\right)
    \ar[l]_-{\sim}\ar[u]_{\sim} }
  \]
  where the vertical equivalence on the left is the canonical
  projection from the unbased to the based homotopy colimit and the
  first map in the bottom row is the equivalence from Proposition
  \ref{hocolimloop}.  Here we use that $B(A)$ is semistable by
  Corollary~\ref{B(A)semistable}.
\end{proof}

For the rest of this section we specialize to the category $\cC\cS^{\cI}$ of commutative $\cI$-space monoids. 
Since $\cC\cS^{\cI}$ is a based simplicial category, the functors
$B$ and $\Omega$ admit a categorical description when applied to a
commutative $\cI$-space monoid $A$: They are given by the tensor and
cotensor of $A$ with the based simplicial set $S^1$. The next result is therefore an immediate consequence of Proposition \ref{prop:based-simplicial}.
\begin{corollary}\label{cor:B-Omega-Qadjunction}
The functors $B$ and $\Omega$ define a Quillen adjunction
\[
B\colon\cC\cS^{\cI}\rightleftarrows\cC\cS^{\cI}\colon\Omega.
\eqno\qed
\]
\end{corollary}
The unit of the adjunction is a map $A\to \Omega(B(A))$ of commutative 
$\cI$-space monoids which may be identified with the explicit map considered above for general $\cI$-space monoids. Notice, that since the space-level fibrant replacement 
$(-)^{\mathrm{fib}}$ is symmetric monoidal, the group completion map $\eta_A$ in \eqref{eq:fib-completion} is a map of commutative $\cI$-space monoids when $A$ is commutative. However,  $\eta_A$ admits a variant which is more natural from the point of view of the positive $\cI$-model structure on $\cC\cS^{\cI}$ and which is the map figuring in Theorem \ref{thm:intro-bar-completion}.
Let $B(A)\to B(A)^{\cI\textrm{-fib}}$ be a fibrant replacement in the positive $\cI$-model structure on $\cC\cS^{\cI}$ and let $\eta_A^{\cI}$ be the composite map of commutative $\cI$-space monoids
\begin{equation}\label{eq:grp-cpl-via-Ifibrant}
\eta_A^{\cI} \colon A \to \Omega(B(A))\to
 \Omega\left(B(A)^{\cI\textrm{-fib}}\right).
\end{equation}
We shall refer to this as the the \emph{derived unit} of the $(B,\Omega)$-adjunction with respect to the positive
$\cI$-model structure. A priori, $\eta_A^{\cI}$ is more
difficult to understand than $\eta_A$ since it involves a fibrant replacement in 
$\cC\cS^{\cI}$. However,
$\eta_A^{\cI}$ has the advantage that it provides a group completion
for objects that are not necessarily semistable as we show in the next proposition.

\begin{proposition}\label{prop:I-fib-space-completion}
If $A$ is a commutative $\cI$-space monoid with underlying flat
$\cI$-space, then there is a chain of natural weak equivalences relating
$\Omega(B(A)^{\textrm{$\cI$-$\mathrm{fib}$}})_{h\cI}$ and 
$\Omega(B(A_{h\cI})^{\mathrm{fib}})$ such that the diagram
\[
\xymatrix@-1pc{
& A_{hI}\ar[dl]_{(\eta^{\cI}_A)_{h\cI}} \ar[dr]^{\eta_{A_{h\cI}}} &\\
\Omega(B(A)^{\textrm{$\cI$-$\mathrm{fib}$}})_{h\cI} \ar[rr]^{\sim}& & 
\Omega(B(A_{h\cI})^{\mathrm{fib}})
}
\]
is commutative in the homotopy category.
\end{proposition}
\begin{proof}
Using $B(A)^{\textrm{$\cI$-{fib}}}$ instead of $B(A)^{\textrm{fib}}$ and 
$\eta_A^{\cI}$ instead of $\eta_A$, the proof is completely analogous to the proof of Proposition~\ref{prop:fib-space-completion}. The point is that the semistability condition on $A$ can be dropped since 
$B(A)^{\textrm{$\cI$-{fib}}}$ is semistable by definition.
\end{proof}

\begin{example}\label{ex:gp_of_free_comm_ispace-monoid}
As in Example \ref{ex:free_comm_ispace-monoid}, let $C_1$ be the
free commutative $\cI$-space monoid on a point in degree one and recall that the underlying simplicial monoid $(C_1)_{h\cI}$ is equivalent to 
$\coprod_{n\geq 0}B\Sigma_n$. By the previous proposition, 
$\eta_{C_1}^{\cI}\colon C_1 \to
  \Omega\left(B(C_1)^{\cI\textrm{-fib}}\right)$ lifts the group
  completion of $(C_1)_{h\cI}$ to a map in $\cC\cS^{\cI}$, so the commutative $\cI$-space monoid   
$\Omega\left(B(C_1)^{\cI\textrm{-fib}}\right)$ is a model of $Q(S^0)$ 
by the Barratt-Priddy-Quillen theorem.
\end{example}

We record some useful consequences of the proposition.

\begin{lemma}\label{lem:group-compl-of-grouplike}
  If $A$ is a cofibrant and grouplike commutative
  $\cI$-space monoid, then $\eta_A^{\cI} \colon A \to
  \Omega\left(B(A)^{\cI\textrm{-fib}}\right)$ is an $\cI$-equivalence.
\end{lemma}
\begin{proof}
This follows from Proposition \ref{prop:I-fib-space-completion} since
the group completion of the underlying simplicial monoid 
$A_{h\cI} \to\Omega\left(B(A_{h\cI})^{\textrm{fib}}\right)$ is a weak
equivalence if $A_{h\cI}$ is grouplike.
\end{proof}

\begin{lemma}\label{lem:bar-of-group-compl}
Let $A$ be a cofibrant commutative $\cI$-space monoid, and let 
\begin{equation}\label{eq:eta-factorization}
\xymatrix@1{A\;\ar@{>->}[r] & \;C\; \ar@{->>}[r]^-{\sim} &
    \;\Omega\left(B(A)^{\textrm{$\cI$-$\mathrm{fib}$}}\right)}
\end{equation}
be a factorization of $\eta^{\cI}_A$ into a positive $\cI$-cofibration followed by an acyclic positive $\cI$-fibration. Then the bar construction $B(-)$ maps 
$A \to C$ to an $\cI$-equivalence.
\end{lemma}
\begin{proof}
  Since the spaces $B(A)_{h\cI}\simeq B(A_{h\cI})$ and
  $B(C)_{h\cI}\simeq B(C_{h\cI})$ are connected, it suffices by
  Proposition \ref{prop:I-fib-space-completion} to show that $B(A) \to
  B(C)$ becomes an $\cI$-equivalence after applying the functor
  $\Omega((-)^{\textrm{$\cI$-$\mathrm{fib}$}})$.  Applying $B(-)$ to
  the second map in \eqref{eq:eta-factorization} and composing with
  the counit of the $(B,\Omega)$-adjunction, we get a map
\begin{equation}\label{eq:counit-compose}
B(C)\to B\Omega(B(A)^{\textrm{$\cI$-$\mathrm{fib}$}})\to 
B(A)^{\textrm{$\cI$-$\mathrm{fib}$}}
\end{equation}
such that the composition with $B(A)\to B(C)$ is the $\cI$-fibrant
replacement of $B(A)$. By the above and the two out of three property
for $\cI$-equivalences, it remains to prove that the composite map in
\eqref{eq:counit-compose} becomes an $\cI$-equivalence after applying
$\Omega((-)^{\textrm{$\cI$-$\mathrm{fib}$}})$. Composing with the
group completion map $\eta_C^{\cI}\colon C \to
\Omega(B(C)^{\textrm{$\cI$-$\mathrm{fib}$}})$, the map so obtained can
be identified with the composition
\[
C\to\Omega(B(A)^{\textrm{$\cI$-$\mathrm{fib}$}})
\to\Omega((B(A)^{\textrm{$\cI$-$\mathrm{fib}$}})^{\textrm{$\cI$-$\mathrm{fib}$}} )
\]
which is an $\cI$-equivalence by assumption. The result now follows from Lemma~\ref{lem:group-compl-of-grouplike} which implies that $\eta_C^{\cI}$ is an $\cI$-equivalence since $C$ is grouplike.  
\end{proof}
\begin{proof}[Proof of Theorem~\ref{thm:intro-bar-completion}]
  Since the underlying $\cI$-space of a cofibrant commutative
  $\cI$-space monoid is flat by
  Proposition~\ref{prop:underlying-flat-com-I}, the previous lemma and
  Proposition~\ref{bar-bar} imply that $(\eta_A^{\cI})_{h\cI}$ induces
  a weak equivalence when applying the bar construction.  This verifies
  that $\eta_A^{\cI}$ is a group completion in the sense of
  Section~\ref{subsec:grp-cpl}.
\end{proof}
 \section{The group completion model structure}\label{sec:csi-gp}

To set up the group completion model structure on $\cC\cS^{\cI}$, we use the description of group completion for commutative $\cI$-space monoids provided by Proposition~\ref{prop:I-fib-space-completion}. Given a commutative $\cI$-space monoid $A$, we write $\Gamma(A)$ for the commutative $\cI$-space monoid $\Omega\left( B(A)^{\cI-\textrm{fib}}\right)$ and 
$\eta^{\cI}_A \colon A \to \Gamma(A)$ for the group completion map introduced in \eqref{eq:grp-cpl-via-Ifibrant}. 

\subsection{Group completion as a fibrant replacement}
\label{sec:gp-as-fibrant}
Let $C_1 =\mC F_{\bld{1}}^{\cI}(*)$ be the free commutative $\cI$-space monoid on a point in degree one, cf.\ Example~\ref{ex:free_comm_ispace-monoid}. We begin by choosing a factorization 
\begin{equation}\label{eq:xi-factorization}
\xymatrix@1{
C_1 \;\ar@{>->}[r]^{\xi} & \;  C_1^{\gp} \;
  \ar@{->>}[r]^-{\sim}& \;\Gamma(C_1)
  }
\end{equation}
of the group completion map $\eta_{C_1}^{\cI}$ as a cofibration 
$\xi$ followed by an acyclic fibration in the positive 
$\cI$-model structure. The next lemma justifies thinking of $\xi\colon C_1\to C_1^{\gp}$ as a ``group completion in the universal example''.

\begin{lemma}\label{lem:extension-grouplike}
Let $A$ be a commutative $\cI$-space monoid which is positive $\cI$-fibrant.
Then $A$ is grouplike if and only if every map $C_1\to A$ extends to a map $C_1^{\gp}\to A$.
\end{lemma} 
\begin{proof}
Suppose first that every map $C_1\to A$ extends to $C_1^{\mathrm{gp}}$ and let $[a]$ be a class in $\pi_0(A_{h\cI})$ represented by a vertex $a$ in 
$A(\bld 1)$. As explained in Example \ref{ex:free_comm_ispace-monoid} , we can represent $a$ by a map $C_1\to A$. Extending this map to $C_1^{\mathrm{gp}}$ and identifying
$\pi_0((C_1^{\mathrm{gp}})_{h\cI})\cong\mathbb Z$ with the group completion of $\pi_0((C_1)_{h\cI})$, we see that $[a]$ is indeed invertible.

Next assume that $A$ is grouplike and consider a map $C_1\to A$. Choosing a cofibrant replacement of $A$, we may assume that $A$ is cofibrant as well as positive $\cI$-fibrant. We now proceed as in \eqref{eq:xi-factorization} by considering a  factorization of the group completion map $\eta_A^{\cI}\colon A\to \Gamma (A)$ as a cofibration $A\to A^{\mathrm{gp}}$ followed by an acyclic $\cI$-fibration  $A^{\mathrm{gp}}\to \Gamma (A)$. From this we obtain a commutative diagram
\[
\xymatrix@-1pc{C_1\ar[r]\ar[d]  & A\ar[d]^{\simeq}\\
C_1^{\mathrm{gp}} \ar[r] & A^{\mathrm{gp}}. 
}
\]
Since $\eta_A^{\cI}$ is an $\cI$-equivalence by 
Lemma~\ref{lem:group-compl-of-grouplike}, it follows that also 
$A\to A^{\mathrm{gp}}$ is an $\cI$-equivalence as indicated in the diagram.
The assumption that $A$ be positive $\cI$-fibrant therefore implies that that the latter map admits a left inverse and the composition with the bottom map in the diagram gives the required extension.
\end{proof}

The group completion model structure will be defined as the left Bousfield localization (see \cite[Chapter 3]{Hirschhorn_model}) of the positive 
$\cI$-model structure on $\cC\cS^{\cI}$ with respect to the map $\xi$. Following \cite{Hirschhorn_model}, we say that a commutative $\cI$-space monoid $W$ is \emph{$\xi$-local} if it is positive $\cI$-fibrant and the induced map
\[
\xi^*\colon \Map(C_1^{\gp},W)\to \Map(C_1,W) 
\]
is a weak equivalence. A map $A\to B$ in $\cC\cS^{\cI}$ is said to be a
\emph{$\xi$-local equivalence} if after choosing a cofibrant replacement 
$\tilde f\colon \widetilde A\to \widetilde B$ (see \cite[Chapter 8]{Hirschhorn_model}) the induced map
\[
\tilde f^*\colon\Map(\widetilde B,W)\to \Map(\widetilde A,W) 
\] 
is a weak equivalence for all $\xi$-local objects $W$. Notice, that any $\cI$-equivalence is a $\xi$-local equivalence.

\begin{proposition}
There exists a model structure on the category $\cC\cS^{\cI}$ such that
\begin{itemize}
\item
the weak equivalences are the $\xi$-local equivalences,
\item
the cofibrations are the same as for the positive $\cI$-model structure, and
\item
the fibrations are the maps that have the right lifting property with respect to the cofibrations that are also $\xi$-local equivalences.
\end{itemize}
\end{proposition}
\begin{proof}
The existence of a model structure with the stated properties follows from Hirschhorn's existence theorem for left Bousfield localizations \cite[Theorem 4.1.1]{Hirschhorn_model}. This uses that $\cC\cS^{\cI}$ is  \emph{cellular} which we verify in 
Proposition~\ref{prop:csi-cellular}. 
\end{proof}

We shall refer to the model structure in the above proposition as the \emph{group completion model structure}. It will be convenient to adapt the notation $\cC\cS^{\cI}$ and $\cC\cS^{\cI}_{\gp}$ for the category of commutative $\cI$-space monoids equipped with the positive $\cI$-model structure and the group completion model structure, respectively. With this notation, the identity functor defines a Quillen adjunction 
\[
L_{\textrm{gp}} \colon \cC\cS^{\cI} \rightleftarrows \cC\cS^{\cI}_{\mathrm{gp}} \colon R_{\mathrm{gp}}
\]  
and the fact that $\cC\cS^{\cI}_{\mathrm{gp}}$ is defined as a left Bousfield localization implies that it has the following universal property: Given a model category $\mathcal M$ and a left Quillen functor 
$F\colon \cC\cS^{\cI}\to \mathcal M$ such that $F(\xi)$ is a weak equivalence in $\cM$, then $F$ is also a left Quillen functor when viewed as a functor on
$\cC\cS^{\cI}_{\mathrm{gp}}$. 

We record some further formal consequences of the definitions.

\begin{proposition}
The group completion model structure $\cC\cS^{\cI}_{\mathrm{gp}}$ is a cofibrantly generated, left proper, and simplicial model structure. 
\end{proposition}
\begin{proof}
This all follows from \cite[Theorem 4.1.1]{Hirschhorn_model}.
\end{proof}

\begin{proposition}\label{prop:B-Omega-completion-adjunction}
The $(B,\Omega)$-adjunction defines a Quillen adjunction
\[
B\colon\cC\cS^{\cI}_{\mathrm{gp}}\rightleftarrows\cC\cS^{\cI}\colon\Omega.
\]
\end{proposition}
\begin{proof}
Since $B(\xi)\colon B(C_1)\to B(C_1^{\mathrm{gp}})$ is an $\cI$-equivalence by Lemma~\ref{lem:bar-of-group-compl}, this follows from the universal property of $\cC\cS^{\cI}_{\mathrm{gp}}$ as a left Bousfield localization.
\end{proof}

\begin{lemma}\label{lem:gp-fibrant}
A commutative $\cI$-space monoid is fibrant in $\cC\cS^{\cI}_{\mathrm{gp}}$ if and only if it is positive $\cI$-fibrant and grouplike.
\end{lemma}
\begin{proof}
Since $\cC\cS^{\cI}$ is left proper by \cite[Proposition~3.2]{Sagave-S_diagram}, it follows from \cite[Proposition 3.4.1] {Hirschhorn_model} that the fibrant objects in $\cC\cS^{\cI}_{\mathrm{gp}}$
are the objects that are $\xi$-local. Suppose first that $A$ is $\xi$-local. This means that $\xi$ induces an acyclic fibration of the simplicial mapping spaces
\[
\xi^*\colon\Map(C_1^{\mathrm{gp}},A)\to \Map(C_1,A)
\]
and in particular that the map of vertices 
$\cC\cS^{\cI}(C_1^{\mathrm{gp}},A)\to \cC\cS^{\cI}(C_1,A)$ is surjective. By  
Lemma~\ref{lem:extension-grouplike} this is equivalent to $A$ being grouplike. 

Next, assuming that $A$ is positive $\cI$-fibrant and grouplike, we must show that $A$ is $\xi$-local. Choosing a cofibrant replacement, we may assume without loss of generality that $A$ is also cofibrant. Then $\eta_A^{\cI}\colon A\to \Gamma(A)$ is an $\cI$-equivalence by 
Lemma~\ref{lem:group-compl-of-grouplike} so that it suffices to show that the upper horizontal map  in the diagram 
\[
\xymatrix@-1pc{\Map(C_1^{\mathrm{gp}},\Gamma(A)) \ar[r] \ar[d]^{\cong} 
& \Map(C_1,\Gamma(A))\ar[d]^{\cong}\\
\Map(B(C_1^{\mathrm{gp}}), B(A)^{\cI-\textrm{fib}}) \ar[r] &
\Map(B(C_1), B(A)^{\cI-\textrm{fib}})
}
\]
is a weak equivalence. Here the vertical isomorphisms are induced by the 
$(B,\Omega)$-adjunction. It follows from 
Lemma~\ref{lem:bar-of-group-compl} that the induced map 
$B(C_1)\to B(C_1^{\mathrm{gp}})$ is an $\cI$-equivalence which in turn implies that the horizontal map in the bottom of the the diagram is a weak equivalence. This gives the statement of the lemma.
\end{proof}

In general, given a commutative $\cI$-space monoid $A$, we write 
$A\to A^{\mathrm{gp}}$ for a fibrant replacement in 
$\cC\cS^{\cI}_{\mathrm{gp}}$ (that is, $ A^{\mathrm{gp}}$ is fibrant in 
$\cC\cS^{\cI}_{\mathrm{gp}}$ and the map is both a cofibration and a 
$\xi$-local equivalence). The terminology is justified by the next lemma.

\begin{lemma}\label{lem:fibrant-replacement-completion}
Fibrant replacement $A\to A^{\mathrm{gp}}$ in $\cC\cS^{\cI}_{\mathrm{gp}}$ models the group completion.
\end{lemma}
\begin{proof}
  Let $A$ be a commutative $\cI$-space monoid and let $C\to A$ be a
  cofibrant replacement in the positive $\cI$-model
  structure. Applying the functor $\Gamma$ to a fibrant replacement
  $C\to C^{\mathrm{gp}}$ in $\cC\cS^{\cI}_{\mathrm{gp}}$, we get a
  commutative diagram
\[ 
\xymatrix@-1pc{ A \ar[d]  & C \ar[r] \ar[d] \ar[l]& \Gamma(C) \ar[d] \\
    A^{\textrm{gp}} & C^{\textrm{gp}} \ar[r] \ar[l]& \Gamma
    (C^{\textrm{gp}})
}
\] 
in which $C^{\mathrm{gp}}\to A^{\mathrm{gp}}$ is an $\cI$-equivalence 
by \cite[Theorem 3.2.18]{Hirschhorn_model} and $C^{\mathrm{gp}}\to\Gamma(C^{\mathrm{gp}})$ is an $\cI$-equivalence by 
Lemma~\ref{lem:group-compl-of-grouplike}. Furthermore, since the 
map $B(C)\to B(C^{\mathrm{gp}})$ is an $\cI$-equivalence by 
Proposition~\ref{prop:B-Omega-completion-adjunction}, it follows that also 
$\Gamma(C)\to \Gamma(C^{\mathrm{gp}})$ is an $\cI$-equivalence. In conclusion, $A^{\mathrm{gp}}$ is $\cI$-equivalent to the group completion of a cofibrant replacement of $A$.
\end{proof}

\begin{remark}
While the group completion $\eta_A\colon A \to \Omega
(B(A)^{\textrm{fib}})$ in \eqref{eq:fib-completion}
needed the underlying $\cI$-space of $A$ to be flat and semistable and the group completion $\eta_A^{\cI} \colon A \to
  \Omega\left(B(A)^{\cI\textrm{-fib}}\right)$ in
  \eqref{eq:grp-cpl-via-Ifibrant} needed $A$ to be cofibrant,
  the fibrant replacement $A \to A^{\textrm{gp}}$ gives a functorial
  group completion on all commutative $\cI$-space monoids.
\end{remark}

Having identified the fibrant objects in $\cC\cS^{\cI}_{\mathrm{gp}}$, we can describe the weak equivalences more explicitly.

\begin{lemma}\label{lem:csi-group-equivalence}
A map of commutative $\cI$-space monoids $C\to D$ is a weak equivalence in $\cC\cS^{\cI}_{\mathrm{gp}}$ if and only if the induced map of group completions 
\[
\Omega(B(C_{h\cI})^{\mathrm{fib}})\to 
\Omega(B(D_{h\cI})^{\mathrm{fib}})
\] 
is a weak equivalence.
\end{lemma}
\begin{proof}
It follows from \cite[Theorem 3.2.18]{Hirschhorn_model} that $C\to D$ is a $\xi$-local equivalence if and only if the induced map of fibrant replacements $C^{\mathrm{gp}}\to D^{\mathrm{gp}}$ is an $\cI$-equivalence. Choosing cofibrant replacements, it suffices to prove the statement in the lemma under the additional hypothesis that $C$ and $D$ be cofibrant. By 
Lemma~\ref{lem:fibrant-replacement-completion} and its proof, the map $C^{\mathrm{gp}}\to D^{\mathrm{gp}}$ can then be identified with the map 
$\Gamma(C)\to \Gamma(D)$ up to $\cI$-equivalence, hence the result follows from Proposition~\ref{prop:I-fib-space-completion}.
\end{proof}

\begin{proof}[Proof of Theorem \ref{thm:csi-gp}]
The characterization of the weak equivalences in $\cC\cS^{\cI}_{\mathrm{gp}}$ provided by Lemma~\ref{lem:csi-group-equivalence} can easily be translated to give the characterization in the theorem. The remaining statements follow from 
Lemmas~\ref{lem:gp-fibrant} and \ref{lem:csi-group-equivalence}.
\end{proof}

\subsection{Group completion and repletion}\label{sec:repletion}
In this section, we study the fibrations in $\cC\cS^{\cI}_{\textrm{gp}}$
and relate them to the \emph{replete} maps introduced in~\cite{Rognes_TLS}. 

Right properness is a desirable feature of a model category. In
general, a model structure that arises through a left Bousfield
localization may or may not be right proper. In the case of 
$\cC\cS^{\cI}_{\textrm{gp}}$, it is not:

\begin{example}
We employ an example from~\cite[\S 5.7]{Bousfield-F_Gamma-bisimplicial} to show that the group completion model structure $\cC\cS^{\cI}_{\textrm{gp}}$ is not right proper. Let $M = \mN_0 \cup \{0'\}$ be
the commutative monoid with $0'+0'=0$, $0+0'=0'$, and $0'+n=n$ for
$n\geq 1$. Its group completion is $\mZ$. Viewing the pullback
diagram of commutative monoids
\[\xymatrix@-1pc{ \{0,0'\} \ar[d] \ar[r] & \{0\} \ar[d] \\ M \ar[r] &
\mZ}\]
as a pullback diagram of constant commutative $\cI$-space monoids,
we see that $\cC\cS^{\cI}_{\textrm{gp}}$ is not right proper: The map 
$M \to \mZ$ is a group completion and hence a weak equivalence in
$\cC\cS^{\cI}_{\textrm{gp}}$, while the map $0 \to \mZ$ is a (positive) 
$\cI$-fibration of grouplike commutative monoids, hence a fibration in 
$\cC\cS^{\cI}_{\textrm{gp}}$.
However, $\{0,0'\} \to \{0\}$ is not a weak equivalence in 
$\cC\cS^{\cI}_{\textrm{gp}}$.
\end{example}

The following definition enables us to state and prove a weakened
form of right properness.

\begin{definition}[{{\cite[Definition 8.1]{Rognes_TLS}}}]
  A map $A \to B$ in $\cC\cS^{\cI}$ is \emph{
    virtual surjective} if the induced map of commutative monoids 
    $\pi_0(A_{h\cI})\to \pi_0(B_{h\cI})$ becomes surjective after group completion.
\end{definition}

\begin{proposition}\label{prop:weak-right-prop-csigp}
Consider a pullback square of commutative
  $\cI$-space monoids
\[\xymatrix@-1pc{A \ar[d]\ar[r] & C \ar[d] \\ B \ar[r] & D}\]
with $B \to D$ a weak equivalence in $\cC\cS^{\cI}_{\mathrm{gp}}$ and
$C \to D$ a virtual surjective fibration in
$\cC\cS^{\cI}_{\mathrm{gp}}$. Then $A \to C$ is a weak equivalence in
$\cC\cS^{\cI}_{\mathrm{gp}}$.
\end{proposition}
\begin{proof}
  Our argument is similar to that used in the proof of~\cite[Proposition 8.3]{Rognes_TLS}, but we make explicit some fibrancy conditions suppressed there. Arguing as in the proof of~\cite[Lemma 9.4]{Bousfield_telescopic}, we may assume without loss of generality that
  $C$ and $D$ are fibrant in $\cC\cS^{\cI}_{\textrm{gp}}$ and that
  $\pi_0(C_{h\cI}) \to \pi_0(D_{h\cI})$ is surjective. Furthermore, since the vertical maps in the diagram are positive $\cI$-fibrations, it follows from 
\cite[Corollary~11.4]{Sagave-S_diagram} that the diagram of simplicial monoids 
\[\xymatrix@-1pc{A_{h\cI} \ar[d]\ar[r] & C_{h\cI} \ar[d] \\ B_{h\cI}
  \ar[r] & D_{h\cI}}\] is a homotopy pullback. We know from
Theorem~\ref{thm:csi-gp} that in the latter diagram the bottom
horizontal map becomes a weak equivalence after applying the bar
construction and we must show that the same holds for the upper
horizontal map. For this we shall shall use the Bousfield-Friedlander
Theorem~\cite[Theorem B.4]{Bousfield-F_Gamma-bisimplicial} to show
that the diagram obtained by applying the bar construction to each of
the simplicial monoids is again a homotopy pullback. It is clear that
we have a homotopy pullback for each fixed degree of the bar
construction. We also observe that in general, given a grouplike
simplicial monoid $M$, an argument similar to proving that the bar
construction on a group is a Kan complex shows that the bar
construction on $M$ satisfy the $\pi_*$-Kan condition
\cite[B.3]{Bousfield-F_Gamma-bisimplicial}. This applies in particular
to the grouplike simplicial monoids $C_{h\cI}$ and $D_{h\cI}$.
Finally, since $\pi_0(C_{h\cI}) \to \pi_0(D_{h\cI})$ is a surjective
group homomorphism, it induces a Kan fibration after applying the bar
construction. This is all we need to apply the Bousfield-Friedlander
Theorem.
\end{proof}

\begin{definition}[{\cite[Definition 8.1]{Rognes_TLS}}]
Let $f\colon A \to B$ be a map in $\cC\cS^{\cI}$. 
\begin{enumerate}[(i)]
\item The map $f$ is \emph{exact} if the commutative square
\[\xymatrix@-1pc{A \ar[r] \ar[d] & A^{\textrm{gp}} \ar[d] \\ B \ar[r] & B^{\textrm{gp}}}\]
is a homotopy pullback in the positive $\cI$-model structure.
\item The map $f$ is \emph{replete} if it is exact and virtual
  surjective.
\end{enumerate}
\end{definition}
\begin{remark}\label{rem:repletion}
For a virtual surjective map $A\to B$, its repletion 
$A^{\textrm{rep}}\to B$ is defined to be the homotopy pullback of the diagram
$B \to B^{\textrm{gp}} \ot A^{\textrm{gp}}$, compare to ~\cite[Definition
8.2]{Rognes_TLS}. By~\cite[Proposition 8.3]{Rognes_TLS} or
  the next proposition,
  $A^{\textrm{rep}} \to B$ is indeed replete.
\end{remark}
\begin{proposition}\label{prop:csigp-fib-vs-repletion}
Let  $f\colon A \to B$ be a map in $\cC\cS^{\cI}$. 
\begin{enumerate}[(i)]
\item If $f$ is exact and a positive $\cI$-fibration, then $f$ is a fibration
in $\cC\cS^{\cI}_{\mathrm{gp}}$. 
\item If $f$ is a fibration in $\cC\cS^{\cI}_{\mathrm{gp}}$ and virtual
surjective, then $f$ is replete.
\end{enumerate} 
\end{proposition}
\begin{proof}
  Part (i) is a formal consequence of~\cite[Proposition
  3.4.7]{Hirschhorn_model} and the right properness of
  $\cC\cS^{\cI}$.  
  For (ii) we build a commutative diagram
\[\xymatrix@-1pc{A
  \ar[dd]_{f} \ar[dr] \ar[rr] & & A^{\textrm{gp}} \ar@{>->}[d]^{\sim}\\
  & P \ar[r] \ar[dl] & C \ar@{->>}[d]\\ B \ar[rr] & & B^{\textrm{gp}}
}\] by factoring $A^{\textrm{gp}} \to B^{\textrm{gp}}$ into an acyclic
cofibration followed by a fibration in the positive $\cI$-model
structure and forming the pullback $P$.
By~\cite[Proposition 3.3.16]{Hirschhorn_model}, $C \to B^{\textrm{gp}}$ is
fibration in $\cC\cS^{\cI}_{\textrm{gp}}$, hence 
Proposition~\ref{prop:weak-right-prop-csigp} implies that $P \to C$ is a weak
equivalence in $\cC\cS^{\cI}_{\textrm{gp}}$. By the 2-out-of-3 property for weak equivalences, $A \to P$
is therefore a weak equivalence in $\cC\cS^{\cI}_{\textrm{gp}}$.  Since $A \to
B$ and $P \to B$ are fibrations in $\cC\cS^{\cI}_{\textrm{gp}}$, it
follows by~\cite[Proposition 3.3.5]{Hirschhorn_model} that $A \to P$
is an $\cI$-equivalence.
\end{proof}

\begin{remark}
  We may summarize the situation as follows: Exactness is a desirable
  property for a map of commutative $\cI$-space monoids, and one wants
  to have an ``exactification''.  The pullback construction of Remark
  \ref{rem:repletion} gives an ``exactification'' $A^{\textrm{rep}}\to
  B$ for virtual surjective maps $A \to B$. For a general map,
  $A^{\textrm{rep}}\to B$ may fail to be exact because $A \to
  A^{\textrm{rep}}$ does not induce an $\cI$-equivalence after group
  completion. However, given any map $A\to B$, a factorization
  $\xymatrix@1{A \;\ar@{>->}[r]^{\sim} & \;A'\;\ar@{->>}[r] & \;B}$
  into an acyclic cofibration followed by a fibration in
  $\cC\cS^{\cI}_{\textrm{gp}}$ gives a commutative $\cI$-space
  monoid $A'$ with maps from $A$ and to $B$ such that $A'$ is well
  defined up to $\cI$-equivalence (see \cite[Proposition
  3.3.5]{Hirschhorn_model}) and $A \to A'$ induces an
  $\cI$-equivalence after group completion. Furthermore, if $A\to B$
  is virtual surjective, then $A'\to B$ coincides with the repletion
  $A^{\textrm{rep}}\to B$ up to $\cI$-equivalence.  In this way the
  fibrations in $\cC\cS^{\cI}_{\textrm{gp}}$ generalize the replete
  maps, and the above factorization in $\cC\cS^{\cI}_{\mathrm{gp}}$
  generalizes repletion.
\end{remark}

\section{The Quillen equivalence to \texorpdfstring{$\Gamma$}{Gamma}-spaces}\label{sec:Gamma-csigp}
In this section we set up the Quillen equivalence relating $\Gamma$-spaces to the group completion model structure on commutative $\cI$-space monoids. We also discuss various ways to explicitly realize the induced equivalence of homotopy categories.

\subsection{Reminder on
  \texorpdfstring{$\Gamma$}{Gamma}-spaces}\label{subs:Gamma-reminder}
Let $\Gamma^{\op}$ be the category of based finite sets (this is the opposite of the category $\Gamma$ considered by Segal~\cite{Segal_categories}) and let us write $k^+$ for the set $\{0,\dots,k\}$ with basepoint $0$. Following Bousfield and Friedlander~\cite{Bousfield-F_Gamma-bisimplicial}, a \emph{$\Gamma$-space} is a  functor $X\colon \Gamma^{\op}\to\cS_*$ such that $X(0^+)=*$. We write $\Gamma^{\op}\cS_*$ for the category of 
$\Gamma$-spaces. Since the full subcategory of $\Gamma^{\op}$ generated by the objects $k^+$ is skeletal, we often define functors out of $\Gamma^{\op}$ only on the sets $k^+$.

It will be convenient to work with the Q-model structures on 
$\Gamma^{\op}\cS_*$ established by Schwede~\cite{Schwede_Gamma-spaces}. In the \emph{level Q-model structure}, a map of $\Gamma$-spaces is a weak equivalence (or fibration) if its evaluation at every object in $\Gamma^{\op}$ is a weak equivalence (or fibration) in $\cS_*$. The existence of such a model structure is a special case of a general construction for based diagram categories which we recall in Proposition~\ref{prop:based-projective-model}. The category 
$\Gamma^{\op}\cS_*$ also has a \emph{stable Q-model structure} with the key feature that its homotopy category is equivalent to the homotopy category of connective spectra: We say that a map of $\Gamma$-spaces $X\to Y$ is a \emph{stable Q-equivalence} if the induced map of spectra (see 
\cite[Section 4]{Bousfield-F_Gamma-bisimplicial}) is a stable equivalence. A map is a \emph{stable Q-cofibration} if it is a cofibration in the level Q-model structure, and a \emph{stable Q-fibration} if it has the right lifting property with respect to maps that are both stable Q-equivalences and stable Q-cofibrations. 
Recall from \cite{Bousfield-F_Gamma-bisimplicial} that a $\Gamma$-space $X$ is \emph{special} if the canonical map 
$X(k^+\vee l^+)\to X(k^+)\times X(l^+)$ is a weak equivalence for all 
$k,l\geq 0$, and \emph{very special} if in addition the induced monoid structure on $\pi_0(X(1^+))$ is a group structure. 

\begin{theorem}[{\cite[Theorem 1.5]{Schwede_Gamma-spaces}}]
The stable Q-equivalences, Q-fibrations, and Q-cofibrations specify a model structure on $\Gamma^{\op}\cS_*$ in which a $\Gamma$-space $X$ is fibrant if and only if it is very special and $X(k^+)$ is a fibrant simplicial set for all 
$k\geq0$. \qed
\end{theorem}

\begin{remark}
The letter $Q$ refers to Quillen and emphasizes that this is not the same as the model structure on $\Gamma$-spaces introduced by 
Bousfield-Friedlander~\cite{Bousfield-F_Gamma-bisimplicial}. As noted in 
\cite[Remark 1.6]{Schwede_Gamma-spaces}, the identity functor on 
$\Gamma^{\op}\cS_*$ defines a left Quillen functor from the stable Q-model structure to the stable Bousfield-Friedlander structure and this is a Quillen equivalence. Hence it follows from 
\cite[Theorem 5.8]{Bousfield-F_Gamma-bisimplicial} that the homotopy category associated to the stable Q-model structure is equivalent to the homotopy category of connective spectra. Our reason for working with the stable Q-model structure is very simple: It has more fibrant objects, so it is easier for a right adjoint functor mapping into $\Gamma$-spaces to be a right Quillen functor. 
\end{remark}

Next we recall from~\cite{Schwede_Gamma-spaces} how the stable Q-model
structure is built from the level Q-model structure. For this, we
write $\Gamma^k$ for the $\Gamma$-space $\Gamma^{\op}(k^+,-)$ and recall that $\mS=\Gamma^1$ plays the role of the sphere spectrum in 
$\Gamma$-spaces.  The projection maps
\[ p_1 \colon (k+l)^+ \to k^+ \quad \text{ and } \quad
p_2\colon(k+l)^+ \to l^+ \] induce maps of $\Gamma$-spaces 
\[ p_1^* \wdg p_2^* \colon \Gamma^k \wdg \Gamma^l \to \Gamma^{k+l}.\]
Similarly, the fold map $\nabla\colon 2^+ \to 1^+$ (given by $1\mapsto 1,
2\mapsto 1$), and the projection $p_1 \colon 2^+ \to 1^+$ induce a map
\[ p_1^* \wdg \nabla^* \colon \Gamma^1 \wdg \Gamma^1 \to \Gamma^2.\]
Consider the set $S$ of maps in $\Gamma^{\op}\cS_*$ defined by
\begin{equation}\label{eq:gamma_localization} 
S=\big\{p_1^* \wdg p_2^*\colon \Gamma^k \wdg \Gamma^l \to \Gamma^{k+l} | k,l\geq 1\big\} \cup
\big\{ p_1^* \wdg \nabla^* \colon \Gamma^1 \wdg \Gamma^1 \to \Gamma^2\big\}.
\end{equation} 
The statement of the next proposition is implicit in the proof of \cite[Theorem 1.5]{Schwede_Gamma-spaces}.
\begin{proposition}[\cite{Schwede_Gamma-spaces}]
The stable Q-model structure on $\Gamma^{\op}\cS_*$ is the left Bousfield localization of the level Q-model structure with respect to the maps in $S$. 
\qed
\end{proposition}
In effect, the first type of maps in $S$ ensure that 
fibrant objects in the stable Q-model structure are special $\Gamma$-spaces, and adding the map $p_1^* \wdg \nabla^*$ forces the very
special condition.

\subsection{Construction of the adjunction}\label{subs:construction-adj}
We first discuss a general principle for constructing left Quillen functors out of $\Gamma^{\op}\cS_*$. Let $\cC$ be a based simplicial category with base object 
$*$. We write $\Map(-,-)$ for the based mapping spaces and $C\otimes K$ for the tensor of an object $C$ in $\cC$ with a based simplicial set $K$. Suppose that we are given a functor $F\colon \Gamma\to\cC$ with $F(0^+)=*$. From this we get a pair of adjoint functors
\begin{equation}\label{eq:Lambda-Phi-adjunction} 
\Lambda_F \colon \Gamma^{\op}\cS_* \rightleftarrows
\cC \colon \Phi_F
\end{equation}
defined by 
\[
\Lambda_F(X)=\int^{k^+ \in \Gamma^{\op}}\!\! F(k^+) \tensor X(k^+)
\qquad  \text{and} \qquad  
\Phi_F(A)=\Map(F(-),A).
\]
It is easy to see that the left adjoint $\Lambda_F$ preserves tensors with based simplicial sets and that there is a natural isomorphism 
$\Lambda_F(\Gamma^k)\cong F(k^+)$. Conversely, any functor $\Lambda\colon\Gamma^{\op}\cS_*\to\cC$ that preserves colimits and tensors becomes naturally isomorphic to $\Lambda_F$ when we define the functor $F$ by 
$F(k^+)=\Lambda(\Gamma^k)$. (This is a consequence of the fact that a 
$\Gamma$-space $X$ can be written as the coend of the $(\Gamma\times\Gamma^{\op})$-diagram $(k^+,l^+)\mapsto \Gamma^k\wedge X(l^+)$.) 

\begin{lemma}\label{lem:Lambda-Phi-Quillen-adjunction}
The $(\Lambda_F,\Phi_F)$-adjunction is a Quillen adjunction with respect to the level Q-model structure on $\Gamma^{\op}\cS_*$ if and only if 
$F(k^+)$ is cofibrant in $\cC$ for $k\geq 0$. 
\end{lemma}
\begin{proof}
The assumption that $\cC$ be a simplicial model category implies that the functor $\Map(F(k^+),-)$ preserves fibrations and acyclic fibrations provided that $F(k^+)$ is cofibrant.
\end{proof}

Given an object $E$ in $\cC$, the above discussion shows how to set up an adjunction $\Gamma^{\op}\cS_*\rightleftarrows \cC$ taking $\mathbb S$ to the prescribed value $E$: letting $F(k^+)=E^{\times k}$, the 
$(\Lambda_F,\Phi_F)$-adjunction has this property.
\begin{example}
Let $\cC$ be the category of spectra (in the sense of \cite{Bousfield-F_Gamma-bisimplicial}), let $\mathbb S$ be the sphere spectrum, and let $F$ be the functor defined by $F(k^+)=\mathbb S^{\times k}$. In this case the adjunction 
\eqref{eq:Lambda-Phi-adjunction} is the one used by Segal~\cite{Segal_categories} and 
Bousfield-Friedlander~\cite{Bousfield-F_Gamma-bisimplicial} to establish an equivalence between the stable homotopy category of $\Gamma$-spaces and the homotopy category of connective spectra. 
\end{example}

\begin{example}
Let $\cC$ be the category of simplicial abelian groups and let $F$ be the functor defined by $F(k^+)=\mathbb Z^{\times k}$, where we think of $\mathbb Z$ as a constant simplicial group. In this case the functor $\Phi_F$ can be identified with the usual Eilenberg-Mac~Lane functor that takes a simplicial abelian group $M$ to the $\Gamma$-space $k^+\mapsto M^{\times k}$. The resulting adjunction \eqref{eq:Lambda-Phi-adjunction} has been analyzed by 
Schwede~\cite{Schwede_Gamma-spaces}.
\end{example}

Now let $\cC$ be the based simplicial category $\cC\cS^{\cI}$, and let $C_1^{\gp}$ be the commutative $\cI$-space monoid considered in 
Section~\ref{sec:gp-as-fibrant}. We know from 
Example~\ref{ex:gp_of_free_comm_ispace-monoid} that $C_1^{\gp}$ represents the infinite loop space $Q(S^0)$ and the idea is to consider the adjunction \eqref{eq:Lambda-Phi-adjunction} associated to the functor $k^+\mapsto (C_1^{\gp})^{\times k}$. However, this adjunction fails to be a Quillen adjunction since the cartesian products $(C_1^{\gp})^{\times k}$ are not cofibrant in $\cC\cS^{\cI}$. In order to overcome this difficulty we appeal to the following general result on based diagram categories whose proof is analogous to the unbased case considered for instance in 
\cite[Section 11.6]{Hirschhorn_model}.

\begin{proposition}\label{prop:based-projective-model}
Let $\cK$ be a based small category and let $\cC$ be a based cofibrantly generated model category. Then there is a cofibrantly generated model structure on the category of based $\cK$-diagrams in $\cC$ such that a map of diagrams is a weak equivalence (respectively a fibration) if and only if it is an objectwise weak equivalence (respectively fibration) in $\cC$. In this model structure the cofibrant diagrams are objectwise cofibrant. \qed 
\end{proposition}

This proposition applies in particular to the category of based $\Gamma$-diagrams in $\cC\cS^{\cI}$, and we use the model structure to choose a cofibrant replacement of the diagram $(C_1^{\gp})^{\times}\colon k^+\mapsto (C_1^{\gp})^{\times k}$. This means that we have a cofibrant $\Gamma$-diagram $C$ together with a natural transformation $C\to (C_1^{\gp})^{\times}$ such that $C(k^+)\to 
(C_1^{\gp})^{\times k}$ is an acyclic fibration in $C\cS^{\cI}$ for 
$k\geq 0$ (hence a positive level equivalence).
By Lemma~\ref{lem:Lambda-Phi-Quillen-adjunction}, the diagram $C$ gives rise to a Quillen adjunction
\begin{equation}\label{eq:GammaS_CSI_adj} 
\Lambda = \Lambda_C \colon \Gamma^{\op}\cS_*
\rightleftarrows \cC\cS^{\cI} \colon \Phi = \Phi_C,
\end{equation} with respect to the level Q-model structure on
$\Gamma^{\op}\cS_*$, sending $\mS$ to the object $C(1^+)$. 

\begin{lemma} \label{lem:Lambda_maps_S_to_we} The left adjoint
  $\Lambda$ in \eqref{eq:GammaS_CSI_adj} sends all maps in the set
  $S$ defined in \eqref{eq:gamma_localization} to $\cI$-equivalences.
\end{lemma}
\begin{proof}
It follows from the discussion at the beginning of 
Section~\ref{subs:construction-adj} 
  that $\Lambda$ maps $p_1^* \wdg p_2^* \colon \Gamma^k \wdg
  \Gamma^{l} \to \Gamma^{k+l}$ to the top horizontal map in the
  commutative square
\[\xymatrix@-1pc{
  C(k^+) \boxtimes C(l^+) \ar[r] \ar[d] &  C((k+l)^+) \ar[d] \\
  (C_1^{\gp})^{\times k} \boxtimes (C_1^{\gp})^{\times l} \ar[r] &
  (C_1^{\gp})^{\times (k+l)}.  }\] Here the right hand vertical map is
an $\cI$-equivalence as noted above.  The left hand vertical map is an
$\cI$-equivalence by
Propositions~\ref{prop:flat-boxtimes-preserves-I-and-N-equiv},
\ref{prop:boxtimes-times-flat}, and \ref{prop:underlying-flat-com-I}.
The bottom map is an instance of the map studied in
Proposition~\ref{prop:boxtimes-times}. It is an $\cI$-equivalence
because it follows from Propositions~\ref{prop:underlying-flat-com-I},
\ref{prop:boxtimes-times-flat}, and
\ref{semistabilityprop}(iii) that cartesian powers of
$C_1^{\gp}$ have underlying flat and semistable $\cI$-spaces. Hence $\Lambda$ sends $p_1^*\vee p_2^*$ to an
$\cI$-equivalence.

For the second type of map in $S$, notice that $\Lambda$ sends $p_1^* \wdg \nabla^* \colon \Gamma^1 \wdg
\Gamma^1 \to \Gamma^2$ to the map $C(1^+) \boxtimes C(1^+) \to C(2^+)$
induced by $p_1^*$ and $\nabla^*$.
As above, showing that this is an $\cI$-equivalence is equivalent to
showing that the map $C_1^{\gp} \boxtimes C_1^{\gp} \to C_1^{\gp} \times C_1^{\gp}$
induced by $p_1^*$ and $\nabla^*$ is an $\cI$-equivalence.
Applying homotopy colimits over $\cI$
and composing with the monoidal structure map and the map induced by
the diagonal $\cI \to \cI \times \cI$, we get a chain of maps
\[ (C_1^{\gp})_{h\cI} \times (C_1^{\gp})_{h\cI} \to (C_1^{\gp}\boxtimes
C_1^{\gp})_{h\cI} \to (C_1^{\gp} \times C_1^{\gp})_{h\cI} \to (C_1^{\gp})_{h\cI}
\times (C_1^{\gp})_{h\cI}. \] 
Here Lemma \ref{lem:times-boxtimes} and Lemma
\ref{lem:hI-of-product} imply that the first and the last map in the
composite are $\cI$-equivalences, so it is sufficient to show the the
composite is an $\cI$-equivalence. It follows from the proof of
Proposition~\ref{prop:boxtimes-times} that the composite is homotopic to the map given by $(x,y)\mapsto (x,\mu(x,y))$ and the claim follows because the simplicial monoid $(C_1^{\gp})_{h\cI}$ is grouplike.
\end{proof}

The last lemma and the universal property of the stable Q-model structure on $\Gamma^{\op}\cS_*$ as a left Bousfield localization (see \cite[Proposition 3.3.18]{Hirschhorn_model}) has the following consequence.
\begin{proposition}
The adjunction \eqref{eq:GammaS_CSI_adj} is a Quillen adjunction with 
respect to the stable Q-model structure on $\Gamma^{\op}\cS_*$ and the positive $\cI$-model structure on $\cC\cS^{\cI}$. \qed
\end{proposition}

Composing with the canonical Quillen adjunction $\cC\cS^{\cI}\rightleftarrows \cC\cS^{\cI}_{\gp}$, we get the Quillen adjunction
\begin{equation}\label{eq:GammaS_CSIgrp_adj} 
\Lambda \colon \Gamma^{\op}\cS_*
\rightleftarrows \cC\cS^{\cI}_{\mathrm{gp}} \colon \Phi,
\end{equation}
again with respect to the stable Q-model structure on $\Gamma^{\op}\cS_*$.

\begin{remark}
If in the definition of the adjunction $(\Lambda,\Phi)$, we had 
used the positive $\cI$-fibrant replacement 
$C_1^{\textrm{$\cI$-$\mathrm{fib}$}}$ instead of $C_1^{\gp}$, the first part of the argument in Lemma~\ref{lem:Lambda_maps_S_to_we} would apply to give a Quillen adjunction between $\cC\cS^{\cI}$ with the positive 
$\cI$-model structure and $\Gamma^{\op}\cS_*$ with a different model structure in which the fibrant objects are the special (and not necessarily very special) $\Gamma$-spaces. The latter model category of $\Gamma$-spaces is  compared to $E_{\infty}$ spaces in~\cite{Santhanam_units}.
\end{remark}

\subsection{Comparison of \texorpdfstring{$\Gamma$}{Gamma}-spaces
built from commutative \texorpdfstring{$\cI$}{I}-space monoids}
Let $A$ be a fibrant object in $\cC\cS^{\cI}_{\gp}$.
Being a right Quillen functor, $\Phi$ takes $A$ to a very special 
$\Gamma$-space $\Phi(A)$ and there is a chain of weak equivalences
  \begin{equation}\label{eq:Phi-of-1plus} 
    \Phi(A)(1^+)=\Map(C(1^+),A) \xleftarrow{\sim} \Map(C_1^{\gp}, A)
    \xrightarrow{\sim} \Map(C_1, A)\iso A(\bld{1})
\end{equation} 
induced by the $\cI$-equivalence $C(1^+)\xrightarrow{\simeq} C_1^{\gp}$ and the weak equivalence $C_1\xrightarrow{\simeq} C_1^{\gp}$ in $\cC\cS^{\cI}_{\gp}$. The last isomorphism exists because the involved 
free/forgetful adjunctions are compatible with the simplicial structure.

A priori, we do not know if the $\Gamma$-space $\Phi(A)$ captures the
``correct'' infinite loop space structure associated to $A$. To see this, we shall compare $\Phi(A)$ to the 
$\Gamma$-space constructed from $A$ by the second author 
in ~\cite{Schlichtkrull_units}.  

We first review some definitions from~\cite[\S 5.2]{Schlichtkrull_units}. For
a finite based set $S$, let $\ovl{S}$ be the unbased set obtained by removing the basepoint and let $\cP(\ovl{S})$ be the category of subsets
and inclusions in $\ovl{S}$. A map $\alpha\colon S \to T$ of based finite sets induces a functor $\alpha^*\colon
\cP(\ovl{T})\to\cP(\ovl{S})$ with $\alpha^*(U) = \alpha^{-1}(U)$.  The
category $\cD(S)$ of $\ovl{S}$-indexed sum diagrams in $\cI$ is
defined to be the full subcategory of the functor category 
$\Fun(\cP(\ovl{S}),\cI)$ whose objects are functors $\theta$ that take disjoint unions to coproducts of finite sets, i.e., if $U \cap V = \emptyset$, then 
$\theta_U \to \theta_{U\cup V} \ot \theta_V$ exhibits $\theta_{U\cup V}$ as a
coproduct of finite sets. An object $\theta$ in $\cD(S)$ is determined by its
values $\theta_s$ for $s \in \ovl{S}$ and a choice of an injection
$\theta_s \to \theta_U$ whenever $s \in U$, such that the induced map
$\concat_{s \in U}\theta_s \to \theta_U$ (with any ordering of the
summands) is an isomorphism in $\cI$.

The construction of the category $\cD(S)$ is functorial in 
$\Gamma^{\op}$: A map $\alpha \colon S
\to T$ induces a functor $\alpha_* \colon \cD(S) \to \cD(T)$ with
$\alpha_*(\theta) = \theta\alpha^*$. Notice, that the restriction to one-point
subsets induces an equivalence of categories $\cD(S) \to
\cI^{\ovl{S}}$. The reason for using $\cD(S)$ instead of $\cI^{\ovl{S}}$
is that the latter is not functorial in $S$.

Next we define a functor $\mC F_S \colon \cD(S)^{\op} \to \cC\cS^{\cI}$ by
$\mC F_S(\theta) = \boxtimes_{s \in \ovl{S}} 
\mC(F_{\theta_s}^{\cI}(*)).$ This uses that $\mC(F_{\bld{n}}^{\cI}(*))$ is
contravariantly functorial with respect to the object $\bld{n}$ in $\cI$.  
A map $\alpha \colon S \to T$ in $\Gamma^{\op}$ induces a natural transformation
$M_{\alpha}\colon \mC F_T \circ \alpha_* \to \mC F_S$. To see this,
fix an object $\theta$ of $\cD(S)$ and observe that 
$(\alpha_*\theta)_t = \theta_{\alpha^{-1}(t)}$ for $t\in \ovl{T}$. It is enough to give a
map \[\mC(F_{\theta_{\alpha^{-1}(t)}}^{\cI}(*)) \to \boxtimes_{s \in
  \alpha^{-1}(t)} \mC(F_{\theta_s}^{\cI}(*))\] for every $t \in
\ovl{T}$, and to give such a map is equivalent to specifying a point in the
evaluation of the codomain at $\theta_{\alpha^{-1}}(t)$. Choosing an
ordering of the set $\alpha^{-1}(t)$, the isomorphism
$\concat_{s\in \alpha^{-1}(t)}\theta_s \to
\theta_{\alpha^{-1}(t)}$ coming from $\theta$ together with the
canonical points $\eins_{\theta_{s}}$ in $\mC
  (F_{\theta_{s}}^{\cI}(*))(\theta_{s})$ (defined as the image of 
  $\eins_{\theta_s}\in F_{\theta_s}^{\cI}(*)(\theta_s)$ under the canonical map 
$F_{\theta_s}^{\cI}(*)\to \mathbb C(F_{\theta_s}^{\cI}(*))$ of $\cI$-spaces)  
  represent the desired element in the iterated $\boxtimes$-product. It is easy to see that this is independent of the ordering and natural in $\theta$. Moreover, if $\beta \colon T \to U$ is another map in $\Gamma^{\op}$, then we have the equality $M_{\beta  \alpha} = M_{\alpha} \circ (M_{\beta} \alpha_*)$.

Now let $A$ be a commutative $\cI$-space monoid. We use the previous
construction to define a $\Gamma$-space $A_{h\cD}$ by
\[ 
A_{h\cD}\colon \Gamma^{\op}\to \cS_*,\qquad
A_{h\cD}(S) = \Map(\mC F_S, A)_{h \cD(S)}.\] The map $A_{h\cD}(S)
\to A_{h\cD}(T)$ induced by a morphism 
$\alpha \colon S \to T$ is defined by 
\[\Map(\mC F_S, A)_{h \cD(S)} \xrightarrow{M_{\alpha}^*} \Map(\mC F_T \circ
\alpha_*, A)_{h \cD(S)} \to \Map(\mC F_T, A)_{h \cD(T)}.\]
Returning to the construction in \cite{Schlichtkrull_units}, we write $A(S)$ for the $\cD(S)$-diagram
\begin{equation}\label{eq:A(S)-definition}
A(S)\colon \cD(S)\to \cS,\qquad A(S)(\theta)=\textstyle\prod_{s\in \overline S}A(\theta_s)
\end{equation}
and observe that $A(S)$ is naturally isomorphic to the $\cD(S)$-diagram 
$\Map(\mathbb CF_S,A)$. Under this isomorphism, the natural transformation $M_{\alpha}^*$ can be identified with the natural transformation 
$A(S)\to A(T)\circ \alpha_*$ defined in \cite[\S 5.2]{Schlichtkrull_units}. This implies the statement of the next lemma.

\begin{lemma}
The $\Gamma$-space $A_{h\cD}$ is canonically isomorphic to the 
$\Gamma$-space $A_{h\cI}$ introduced in~\cite[\S 5.2]{Schlichtkrull_units}. 
\qed
\end{lemma}

The point of defining $A_{h\cD}$ in terms of the mapping spaces 
$\Map(\mathbb C F_S,A)$ is to facilitate the comparison to $\Phi(A)$ in the next proposition.

\begin{proposition}\label{prop:comparison-gamma-spaces}
  Let $A$ be a commutative $\cI$-space monoid and suppose that $A$ is grouplike and positive $\cI$-fibrant. Then there is a zig-zag chain of level equivalences  between $\Gamma$-spaces relating $A_{h\cD}$ and $\Phi(A)$.
\end{proposition}
\begin{proof}
  Viewing $\cD(S)$ as the value of a functor $\cD(-)\colon
  \Gamma^{\op} \to \textrm{Cat}$, we form the Grothendieck
  construction $\Gamma^{\op}\int \cD(-)$, 
  see \cite{Thomason_homotopy-colimt}. The objects of this category are pairs $(S,\theta)$ given by an object $S$ in 
$\Gamma^{\op}$ and an object $\theta$ in $\cD(S)$.  
A  morphism $(\alpha,f)\colon (S,\theta) \to (T,\omega)$ is given by a pair of morphisms $\alpha \colon S \to T$ in $\Gamma^{\op}$ and $f\colon
  \alpha_*(\theta) \to \omega$ in $\cD(T)$. The functors $\mC F_S$
  considered above assemble to a functor $\mC F\colon (\Gamma^{\op}
  {\textstyle \int} \cD(-))^{\op} \to \cC\cS^{\cI}$ that sends
  $(S,\theta)$ to $\mC F_S(\theta)$ and $(\alpha,f)$ to
  \[
  \mC F_T(\omega)\xrightarrow{f_*} \mC F_T(\alpha_*(\theta))
  \xrightarrow{M_{\alpha}}\mC F_S(\theta).
  \]
Using Proposition \ref{prop:based-projective-model}, we choose a fibrant replacement $\mathbb CF^{\gp}$ of this diagram in the projective model structure on $(\Gamma^{\op}{\textstyle \int} \cD(-))^{\op}$-diagrams inherited from the group completion model structure $\cC\cS_{\gp}^{\cI}$. Thus, we have a map of diagrams $\mathbb C F\to \mathbb C F^{\gp}$, such that 
$\mathbb C F^{\gp}_S(\theta)$ is fibrant and $\mC F_S(\theta)\to  \mC F_S^{\gp}(\theta)$ is an acyclic cofibration in $\cC\cS_{\gp}^{\cI}$ for all objects $(S,\theta)$.  We create a second $\Gamma$-space direction by defining
\[
  (\mC F^{\gp})^{\times} \colon (\Gamma^{\op} {\textstyle \int} \cD(-))^{\op}
  \times \Gamma\to \cC \cS^{\cI} \quad \textrm{where} \quad 
  (\mC F^{\gp})^{\times}(S,\theta,k^+) = \left(\mC F_S^{\gp}(\theta)\right)^{\times k}.
\]
This is a based diagram in each variable in the sense that for a fixed object 
$(S,\theta)$ in $(\Gamma^{\op} {\textstyle \int} \cD(-))^{\op}$ it defines a based $\Gamma$-diagram in $\cC\cS^{\cI}$ and for a fixed object $k^+$ in $\Gamma$ it defines a based $(\Gamma^{\op} {\textstyle \int} \cD(-))^{\op}$-diagram in 
$\cC\cS^{\cI}$. Since $(\mC F^{\gp})^{\times}$ does not take values in cofibrant commutative $\cI$-space monoids, we need to replace it by a diagram with cofibrant values while maintaining the property of being based in each variable. For this purpose we apply Proposition~\ref{prop:based-projective-model} with 
$\cC$ the category of based $\Gamma$-diagrams in $\cC\cS^{\cI}$. By adjointness, we may view 
$(\mC F^{\gp})^{\times}$ as a based $(\Gamma^{\op} {\textstyle \int} \cD(-))^{\op}$-diagram in $\cC$ and choosing a cofibrant replacement 
$\overline C\to (\mC F^{\gp})^{\times}$ we get a $(\Gamma^{\op} {\textstyle \int} \cD(-))^{\op}\times \Gamma$ diagram $\overline C$ with the required properties. 

  Now we are in a position to apply the results about
  bi-$\Gamma$-spaces from Appendix~\ref{app:bi-Gamma}. The commutative
  $\cI$-space monoid $A$ gives rise to a bi-$\Gamma$-space $X$ with
  \[ X(S,k^+) = \Map(\ovl{C}(S,-,k^+),A)_{h\cD(S)}.\] 
  Fixing the second variable $k^+=1^+$, the chain of weak equivalences 
 \[
 \Map(\ovl{C}(S,-,1^+),A)_{h\cD(S)}\xleftarrow{\simeq}
 \Map(\mC F_S^{\gp},A)_{h\cD(S)}\xrightarrow{\simeq}
 \Map(\mC F_S,A)_{h\cD(S)}
 \]
 defines a chain of level equivalences relating $X(-,1^+)$ and
 $A_{h\cD}$.  Next we fix the first variable $S=1^+$. An argument
 similar to that used in Example~\ref{ex:free_comm_ispace-monoid}
 shows that the $\cI^{\op}$-diagram $\bld n\mapsto \mC(F_{\bld
   n}^{\cI}(*))$ is a diagram of $\cI$-equivalences in positive
 degrees. Identifying $\cD(1^+)$ with $\cI$, this implies that
 $\ovl{C}(1^+,-,k^+)$ is a diagram of $\cI$-equivalences in positive
 degrees. Using that the cofibrant replacement $\ovl{C}$ is an
 objectwise cofibrant diagram, we obtain weak equivalences
 \[
 \Map(\ovl{C}(1^+,-,k^+),A)_{h\cI}\xleftarrow{\simeq} 
 \Map(\ovl{C}(1^+,\bld 1,k^+),A)\xrightarrow{\simeq}
 \Map(C(k^+),A),
 \]
where the first map is induced by the inclusion of $\{\bld 1\}$ in $\cI$, and the $\Gamma$-diagram $C$ is as in Section~\ref{subs:construction-adj}. These maps define a chain of level equivalences relating $X(1^+,-)$ and $\Phi(A)$.
Finally, having replaced $\mathbb CF$ by $\mathbb CF^{\mathrm{gp}}$ in the first step of the proof, we can use Proposition~\ref{prop:boxtimes-times}
(as in the proof of Lemma~\ref{lem:Lambda_maps_S_to_we}) and
\cite[Proposition 5.3]{Schlichtkrull_units} to show that $X$ is bi-special. Since $X(-,1^+)$ is very special by~\cite[Proposition 5.3]{Schlichtkrull_units}, we conclude from Lemma~\ref{lem:two-monoid-structures} that $X$ is bi-very special. Hence Proposition \ref{prop:equivalent-Gamma-spaces} provides a
zig-zag chain of level equivalences between $X(-,1^+)$ and $X(1^+,-)$.
\end{proof}

Anticipating the definition of the units model structure in Section~\ref{sec:un-csi}, we deduce from Proposition~\ref{prop:comparison-gamma-spaces} the following consistency result for the definition of the spectrum of units associated to a commutative symmetric ring spectrum. 
\begin{corollary}
  Let $R$ be a commutative symmetric ring spectrum. The $\Gamma$-space
  of units $\gl_1(R)$ defined in \eqref{eq:gl_1-def} is stably
  equivalent to the $\Gamma$-space of units associated to $R$ 
  in~\cite[Proposition 5.5]{Schlichtkrull_units}.  \qed
\end{corollary}

\begin{remark}
The notion of units for a commutative symmetric ring spectra considered here has been compared by Lind~\cite{Lind-diagram} to the corresponding notion of units for commutative $S$-algebras in the sense 
of~\cite{EKMM}.  
\end{remark}

Finally, we discuss another, conceptually simpler, way to associate a
$\Gamma$-space to a commutative $\cI$-space monoid $A$. For this we
view finite based sets as discrete based spaces. Via the resulting
functor $\Gamma^{\op}\to\cS_*$, the tensor defines a $\Gamma$-object
$S\mapsto A\otimes S$ in $\cC\cS^{\cI}$. We define a $\Gamma$-space
$A_{\Gamma}\colon\Gamma^{\op}\to \cS_*$ by evaluating the based
homotopy colimit over $\cI$ levelwise: 
$A_{\Gamma}(S)=(A\otimes S)_{h^*\cI}$. Here $(-)_{h^*\cI}$ denotes the based homotopy colimit as in Section~\ref{subsec:loop-functor}. It is necessary to pass to the based homotopy colimit in order for $A_{\Gamma}$ to be a based diagram. Notice, that $A_{\Gamma}(k^+)$ can be identified with $(A^{\boxtimes k})_{h^*\cI}$.   

We wish to compare $A_{\Gamma}$ with the $\Gamma$-space $A_{h\cD}$ considered above and for this purpose we introduce an auxiliary 
$\Gamma$-space $A_{h\mathcal E}$. The definition of the latter uses a modified version $\cE(S)$ of the categories $\cD(S)$ used in the definition of 
$A_{h\cD}$. The objects of $\mathcal E(S)$ are pairs $(\bld{m},M)$
consisting of an object $\bld{m}$ of $\cI$ and a functor $M\colon
\cP(\ovl{S}) \to (\cI\!\downarrow\!\bld{m})$ such that the underlying
functor $P_S(\bld{m},M)\colon \cP(\ovl{S}) \to \cI$, obtained by
composing with the projection $(\cI\!\downarrow\!\bld{m})\to\cI$ that
forgets the augmentation, is an object of $\cD(S)$. A morphism
$(\bld{m},M)\to(\bld{n},N)$ is an injective map
$\kappa\colon\bld{m}\to\bld{n}$ together with a natural transformation
$\kappa_*\circ M \to N$ of functors $\cP(\ovl{S}) \to
(\cI\!\downarrow\!  \bld{n})$. Notice that a choice of ordering
$\{s_1,\dots,s_k\}$ of $\ovl{S}$ determines an equivalence of
categories between $\mathcal E(S)$ and the comma category
$(\concat^{k}\! \downarrow \!\cI)$. As in the case of the $\cD(S)$, a
map $\alpha\colon S \to T$ in $\Gamma^{\op}$ induces a functor 
$\alpha^*\colon\cP(\ovl{T}) \to \cP(\ovl{S})$ and hence a functor 
$\alpha_*\colon\mathcal E(S) \to \mathcal E(T)$ sending 
$(\bld{m},M)$ to $(\bld{m},M \circ \alpha^*)$. Writing $P_S\colon \cE(S)\to
\cD(S)$ for the functor introduced above, we have the equality 
$\alpha_*P_S=P_T\alpha_*$.

Let again $A$ be a commutative $\cI$-space monoid and let us view the 
$\cD(S)$-diagram $A(S)$ in \eqref{eq:A(S)-definition} as a diagram of based spaces with base points specified by the unit of $A$. Hence we can form the 
$\Gamma$-space $A_{h^*\cD}$ defined by the based homotopy colimits 
$A_{h^*\cD}(S)=A(S)_{h^*\cD(S)}$. The canonical map of $\Gamma$-spaces 
$A_{h\cD}\to A_{h^*\cD}$ is a levelwise equivalence since the categories $\cD(S)$ have contractible classifying spaces. Using the same notation $A(S)$ for the $\cE(S)$-diagram obtained by composing with $P_S$, we similarly define a $\Gamma$-space $A_{h^*\cE}$ with $A_{h^*\cE}(S)=A(S)_{h^*\cE(S)}$. As for $A_{h^*\cD}$, the structure map 
induced by a morphism $\alpha\colon S \to T$ in $\Gamma^{\op}$ is defined by 
\[ A_{h^* \mathcal E}(S)=A(S)_{h^*\mathcal E(S)} \to
(A(T)\circ\alpha_*)_{h^*\mathcal E(S)} \to A(T)_{h^*\mathcal E(T)}= 
A_{h^* \mathcal E}(T).\]

\begin{lemma}\label{lem:E-D-equivalence}
  The functors $P_S$ induce a natural level equivalence of
  $\Gamma$-spaces $A_{h^*\mathcal E} \to A_{h^*\mathcal D}$.
\end{lemma}
\begin{proof}
  The map $A(S)_{h^*\mathcal E(S)} \to A(S)_{h^*\cD(S)}$ induced by $P_S$
  is natural in $S$. To see that it is a weak equivalence, we check that $P_S$ is
  \emph{homotopy right cofinal} in the sense of 
  \cite[Theorem~19.6.13]{Hirschhorn_model}.
  By the above equivalences of categories, it is enough to prove this for the 
  projection  $P_k\colon (\concat^{k} \!\downarrow\! \cI) \to \cI^{k}$. 
  Given an object $(\bld{n_1},\dots,\bld{n_k})$ in $\cI^k$, the conclusion now follows from the fact that $((\bld{n_1},\dots,\bld{n_k})\!\downarrow\! P_k )$ has
  an initial object and therefore a contractible classifying space.
\end{proof}
The point in defining $A_{h*\mathcal E}$ is that it admits a map of
$\Gamma$-spaces to $A_{\Gamma}$. For this we let $\varepsilon_S \colon
\mathcal E(S)\to\cI$ be the functor sending $(\bld{m},M)$ to
$\bld{m}$ and claim that there is a natural transformation $A(S) \to
\varepsilon_S^*(A\tensor S)$
of functors $\mathcal E(S)\to\cS$. Indeed, writing $P_S(\bld m,M)=\theta$, we define such a natural transformation by mapping an element $\{a_s\in A(\theta_s)\}_{s\in \overline S}$ in $A(S)$ to the element of $(A\otimes S)(\bld m)$ specified by the 
$\overline S$-tuple of objects $\{\theta_s\}_{s\in \overline S}$ in $\cI$, the morphism 
$\concat_{s\in\overline S}\theta_s\to \theta_{\overline S}\to \bld m$,  and 
the $\overline S$-tuple of elements $\{a_s\}_{s\in\overline S}$. Here the map 
$\theta_{\overline S}\to \bld m$ is part of the structure defining $(\bld m,M)$ and the resulting element of $(A\otimes S)(\bld m)$ does not depend on the ordering of 
$\overline S$ used to define $\concat_{s\in \overline S}\theta_s$. On homotopy colimits we obtain a map
\[ A_{h^*\mathcal E}(S) \iso A(S)_{h^*\mathcal E(S)} \to
\left(\varepsilon_S^*(A\tensor S)\right)_{h^*\mathcal E(S)} \to
(A\tensor S)_{h^*\cI}.
\]
\begin{lemma}
  The above map induces a natural map of $\Gamma$-spaces
  $A_{h^*\mathcal E} \to A_{\Gamma}$. If the underlying $\cI$-space of $A$ is flat, 
  then the $\Gamma$-space $A_{\Gamma}$ is special and 
  $A_{h^*\mathcal E} \to A_{\Gamma}$ is a level equivalence.
\end{lemma}
\begin{proof}
Unraveling the definitions, we see that a morphism $\alpha\colon S\to T$ in $\Gamma^{\op}$ gives rise to the commutative diagram
\[\xymatrix@-1pc{A(S)\ar[r] \ar[d] & \varepsilon^*_S(A\tensor S) \ar[d] \\
  \alpha^*(A(T)) \ar[r] & \alpha^*(\varepsilon^*_T(A\tensor T))}\]
of $\mathcal E(S)$-spaces. From this it easily follows that $A_{h\mathcal E} \to
A_{\Gamma}$ is a map of $\Gamma$-spaces. The $\Gamma$-space $A_{h^*\mathcal E}$ is special by Lemma~\ref{lem:E-D-equivalence} and the fact that $A_{h^*\cD}$ is special (see \cite[Proposition 5.3]{Schlichtkrull_units}). Furthermore, assuming that the underlying $\cI$-space of $A$ is flat, it follows from Corollary~\ref{cor:flat-boxtimes-times} that $A_{\Gamma}$ is also special. So it is enough to show that 
$A(1^+)_{h^*\mathcal E(1^+)}\to (A\tensor 1^+)_{h^*\cI}$ is a
weak equivalence, and this follows from the cofinality argument in the proof of 
Lemma~\ref{lem:E-D-equivalence}.
\end{proof}
Combining the last two lemmas, we get the desired comparison of $\Gamma$-spaces. 
\begin{proposition}\label{prop:D-E-Gamma}
  For a flat commutative $\cI$-space monoid $A$, there are
  level equivalences $A_{h\cD} \to A_{h^*\cD}\ot A_{h^*\mathcal E} \to A_{\Gamma}$
  between special $\Gamma$-spaces.\qed
\end{proposition}

\begin{remark}
Whereas the $\Gamma$-space $\Phi(A)$ requires $A$ to be a fibrant object in 
$\cC\cS^{\cI}_{\mathrm{gp}}$, the $\Gamma$-space $A_{\Gamma}$ requires $A$ to have an underlying flat $\cI$-space in order to be homotopically well-behaved. In contrast, the 
$\Gamma$-space $A_{h\cD}$ represents the correct stable homotopy type for all commutative $\cI$-space monoids $A$.
\end{remark}

We shall be particularly interested in the case where $A$ is the commutative $\cI$-space monoid $C_1$. Recall from Example~\ref{ex:free_comm_ispace-monoid} that $(C_1)_{h\cI}$ can be identified with the disjoint union of the spaces $B(\bld n\downarrow \cI)/\Sigma_n$. Let $\mathbb S\to (C_1)_{\Gamma}$ be the map of $\Gamma$-spaces specified by the vertex $\id_{\bld 1}$ in $B(\bld 1\downarrow \cI)$. 

\begin{lemma}\label{lem:S-C-1-Gamma}
The map of $\Gamma$-spaces $\mathbb S\to (C_1)_{\Gamma}$ specified above is a stable equivalence. 
\end{lemma}
\begin{proof}
Consider in general a based simplicial set $K$. It follows from the universal property of the tensor that $C_1\otimes K$ can be identified with the commutative $\cI$-space monoid $K^{\bullet}\colon \bld n\mapsto K^n$ (see \cite{Schlichtkrull_infinite}). Thus, the map $K\to (C_1\otimes K)_{h^*\cI}$ can be identified with the map 
$K\to (K^{\bullet})_{h^*\cI}$ induced by the inclusion $\{\bld 1\}\to \cI$. The result now follows from \cite[Lemma 3.5]{Schlichtkrull_infinite} which states that this map is 
$(2n-1)$-connected for $K=S^n$.
\end{proof}

In the next corollary we consider the map of $\Gamma$-spaces $\mS\to 
\Phi(C(1^+))$ induced by the map $1^+\to \Map(C(1^+),C(1^+))$ sending the non-base point to the identity of $C(1^+)$.

\begin{corollary}\label{cor:Phi-of-S}
The above map of $\Gamma$-spaces $\mS\to \Phi(C(1^+))$ is a stable equivalence.
\end{corollary}
\begin{proof}
The acyclic $\cI$-fibration $C(1^+)\to C_1^{\mathrm{gp}}$ induces a level equivalence 
of $\Gamma$-spaces $\Phi(C(1^+))\to \Phi(C_1^{\mathrm{gp}})$, so it suffices to show that the composite map $\mathbb S\to \Phi(C_1^{\mathrm{gp}})$ is a stable equivalence.
Combining Propositions~\ref{prop:comparison-gamma-spaces} and 
\ref{prop:D-E-Gamma}, we get a zig-zag chain of level equivalences relating the 
$\Gamma$-spaces $\Phi(C_1^{\mathrm{gp}})$ and $(C_1^{\mathrm{gp}})_{\Gamma}$. Furthermore, it follows from Lemma~\ref{lem:bar-of-group-compl} that the map 
$C_1\to C_1^{\mathrm{gp}}$ induces a stable equivalence of $\Gamma$-spaces 
$(C_1)_{\Gamma}\to (C_1^{\mathrm{gp}})_{\Gamma}$. Thus, composing with the stable equivalence $\mathbb S\to (C_1)_{\Gamma}$ in Lemma~\ref{lem:S-C-1-Gamma}, we get a stable equivalence  $\mathbb S\to (C_1^{\mathrm{gp}})_{\Gamma}$. Since by definition the map $\mathbb S\to \Phi(C_1^{\mathrm{gp}})$ takes the non-base point in $1^+$ to a generator for the infinite cyclic group of components of $\Phi(C_1^{\mathrm{gp}})(1^+)$, this suffices to prove the result.
\end{proof}

\subsection{Stabilization and the proof of Theorem~\ref{thm:Gamma-csigp}.}
\label{subsec:stabilization-proof}
In this section we finish the proof of Theorem \ref{thm:Gamma-csigp} by showing that the Quillen adjunction \eqref{eq:GammaS_CSIgrp_adj}  is in fact a Quillen equivalence with respect to the stable $Q$-model structure on $\Gamma^{\op}\cS_*$.

We begin by making some general remarks on Quillen adjunctions and derived units. Thus, consider a Quillen adjunction $F\colon \cC\rightleftarrows \cD \!: \!G$
relating the model categories $\cC$ and $\cD$. For an object $X$ of $\cC$, the adjoint of a functorial fibrant replacement $F(X) \to F(X)^{\textrm{fib}}$ in $\cD$ is a map $\epsilon_X \colon X \to G(F(X)^{\textrm{fib}})$. These maps
assemble to a natural transformation $\epsilon$ that we call the 
\emph{derived unit} of the adjunction. It is often interesting to know if
$\epsilon$ is a natural weak equivalence on the full subcategory of
cofibrant objects in $\cC$. This is equivalent to asking for the unit of the induced adjunction of homotopy categories to be a natural isomorphism. Recall that the right adjoint $G$ is said to \emph{reflect} weak equivalences between fibrant objects if a morphism $f\colon Y\to Y'$ between fibrant objects in $\cD$ is a weak equivalence provided that $G(f)$ is a weak equivalence in 
$\cC$. By \cite[Corollary 1.3.16]{Hovey_model}, a Quillen adjunction $(F,G)$ as above is a Quillen equivalence if and only if $G$ reflects weak equivalences between fibrant objects and the derived unit $\epsilon$ is a weak equivalence on cofibrant objects. For future reference, we analyze the derived unit of the composition of a pair a Quillen adjunctions  
\[ \xymatrix@-1pc{ \cC \ar@<.5ex>[r]^{F} & \cD \ar@<.5ex>[r]^{H}
  \ar@<.5ex>[l]^{G} & \mathcal{E} \ar@<.5ex>[l]^{K}}.
\]
Let again $\epsilon$ be the derived unit of the $(F,G)$-adjunction, and let us write $\nu$ for the derived unit of the $(H,K)$-adjunction. Checking from the definitions, it is easy to verify the statement in the next lemma.
We will later use various 2-out-of-3 statements derived from this.

\begin{lemma}\label{lem:derived_units_composite}
  The derived unit of the composed Quillen adjunction $(HF,GK)$
  evaluated at $X$ is weakly equivalent to the map
\[ X \xrightarrow{\epsilon_X} G( F(X)^{\mathrm{fib}}) \xrightarrow{G(
  \nu_{F(X)^{\mathrm{fib}}})} G(K(H(F(X)^{\mathrm{fib}})^{\mathrm{fib}}
)).\eqno\qed
\] 
\end{lemma}

Now let us return to the $(\Lambda,\Phi)$-adjunction in \eqref{eq:GammaS_CSIgrp_adj}. It is easy to verify the first part of the condition for this to be a Quillen equivalence.

\begin{lemma}\label{lem:we_between_fibrants}
The right adjoint $\Phi$ in \eqref{eq:GammaS_CSIgrp_adj} reflects
weak equivalences between fibrant objects in $\cC\cS^{\cI}_{\gp}$.
\end{lemma}
\begin{proof}
  Let $f \colon A \to B$ be a map between fibrant objects in
  $\cC\cS^{\cI}_{\textrm{gp}}$ and assume that $\Phi(f)$ is a weak
  equivalence.  Since $\Phi(A)$ and $\Phi(B)$ are fibrant and hence very special, the map $\Phi(f)(1^+)$ is a weak equivalence of spaces. The weak
  equivalences in \eqref{eq:Phi-of-1plus} therefore imply that $f(\bld{1})\colon A(\bld 1)\to B(\bld 1)$ is also a weak equivalence. Hence $f$ is an $\cI$-equivalence since $A$ and $B$ are positive $\cI$-fibrant.
\end{proof}

The next aim is to show that the derived unit of the adjunction is a
weak equivalence on cofibrant objects. For this we proceed as in 
\cite{Hovey_symmetric-general} by forming the stabilizations of the categories 
$\Gamma^{\op}\cS_*$ and $\cC\cS^{\cI}$ (with respect to the tensor with $S^1$) to get the categories of spectra $\Sp^{\mN}(\Gamma^{\op}\cS_*)$ and 
$\Sp^{\mN}(\cC\cS^{\cI})$. We equip these categories of spectra with the stable model structures defined in \cite[\S 3]{Hovey_symmetric-general}. It is then an immediate consequence of Lemma~\ref{lem:bar-of-group-compl} that the suspension spectrum functor $F_0\colon \cC\cS^{\cI}\to \Sp^{\mN}(\cC\cS^{\cI})$ takes the map $\xi$ in \eqref{eq:xi-factorization} to a stable equivalence. Hence $F_0$ defines a left Quillen functor 
$F_0\colon \cC\cS^{\cI}_{\gp}\to \Sp^{\mN}(\cC\cS^{\cI})$ and we obtain a diagram of Quillen adjunctions

\begin{equation}\label{eq:stabilization_square} 
\xymatrix@-.5pc{\Gamma^{\op}\cS_*
\ar@<.5ex>[r]^{\Lambda} \ar@<-.5ex>[d]_{F_0} &
\cC\cS^{\cI}_{\textrm{gp}}
\ar@<.5ex>[l]^{\Phi}\ar@<-.5ex>[d]_{F_0} \\ Sp^{\mN}(\Gamma^{\op}\cS_*)
\ar@<.5ex>[r]^{\Lambda^{\textrm{st}}} \ar@<-.5ex>[u]_{\Ev_0}
& Sp^{\mN}(\cC\cS^{\cI}). \ar@<.5ex>[l]^{\Phi^{\textrm{st}}}
\ar@<-.5ex>[u]_{\Ev_0}}
\end{equation}
The functors $\Lambda^{\textrm{st}}$ and $\Phi^{\textrm{st}}$ are the prolongations of $\Lambda$ and $\Phi$ to maps of spectra. For 
$\Lambda^{\textrm{st}}$ this uses that $\Lambda$ preserves the tensor with $S^1$ and for $\Phi^{\textrm{st}}$ we use the natural transformation
\[
\Phi(A)\otimes S^1\to \Phi\Lambda(\Phi(A)\otimes S^1)\cong 
\Phi(\Lambda\Phi(A)\otimes S^1)\to \Phi(A\otimes S^1)
\] 
defined by the unit and counit of the $(\Lambda,\Phi)$-adjunction. It follows from \cite[Proposition~5.5]{Hovey_symmetric-general} that the 
$(\Lambda^{\textrm{st}},\Phi^{\textrm{st}})$-adjunction is a Quillen adjunction. 
We recall from \cite[Chapter 7]{Hovey_model} that the homotopy categories of 
$\Sp^{\mN}(\Gamma^{\op}\cS_*)$ and $\Sp^{\mN}(\cC\cS^{\cI})$ are triangulated and that the total derived functors $L\Lambda^{\textrm{st}}$ and 
$R\Phi^{\textrm{st}}$ are exact.
 
\begin{remark}
There is a diagram analogous to that in \eqref{eq:stabilization_square} but with $\Sp^{\mN}(\cC\cS^{\cI}_{\gp})$ instead of $\Sp^{\mN}(\cC\cS^{\cI})$. Our preference for the latter category is that this makes it easier to derive the compactness statement in Lemma~\ref{lem:compactness} below.
 \end{remark}
 
\begin{lemma}\label{lem:F-Ev-adjunctions}
The derived units of the $(F_0,\Ev_0)$-adjunctions in \eqref{eq:stabilization_square} are weak equivalences on cofibrant objects.
\end{lemma}
\begin{proof}
Consider first the case of $\Gamma^{\op}\cS_*$. Given a $\Gamma$-space $X$ we claim that a fibrant replacement of $F_0(X)$ in the level model structure on $\Sp^{\mN}(\Gamma^{\op}\cS_*)$ (see  \cite[\S 1]{Hovey_symmetric-general}) is in fact an $\Omega$-spectrum, hence also a fibrant replacement in the stable model structure. Assuming this, the derived unit of the adjunction can be identified with a fibrant replacement in $\Gamma^{\op}\cS_*$ and is therefore a weak equivalence. To prove the claim, let $X\to \widehat X$ be a fibrant replacement in $\Gamma^{\op}\cS_*$ and observe that by 
\cite[Lemma 4.1]{Bousfield-F_Gamma-bisimplicial}, the $\Gamma$-space 
$k^+\mapsto \widehat X(k^+\wedge S^n)^{\textrm{fib}}$ 
(where $(-)^{\textrm{fib}}$ denotes fibrant replacement in $\cS_*$) is a fibrant replacement of $X\wedge S^n$ for all $n\geq 0$. The claim made above now follows from \cite[Theorem~4.2]{Bousfield-F_Gamma-bisimplicial}.

Next consider the case of $\cC\cS^{\cI}$ and let $A$ be a cofibrant commutative $\cI$-space monoid. Using 
Lemma~\ref{lem:group-compl-of-grouplike} we see that a fibrant replacement of $F_0(A)$ in the level model structure on $\Sp^{\mN}(\cC\cS^{\cI})$ is in fact an $\Omega$-spectrum in positive degrees. This implies that the derived unit of the adjunction can be identified with a fibrant replacement in 
$\cC\cS^{\cI}_{\gp}$, hence is a weak equivalence. 
\end{proof} 
 
\begin{lemma}\label{lem:unit_we_on_sphere}
The derived unit of the adjunction $(\Lambda^{\mathrm{st}},\Phi^{\mathrm{st}})$ is a weak equivalence when evaluated at $F_0(\mS)$.
\end{lemma} 
\begin{proof}
Since $\Lambda$ preserves tensors and $\Lambda(\mS)\cong C(1^+)$, we may identify $\Lambda^{\textrm{st}}(F_0(\mS))$ with $F_0(C(1^+))$. Let 
$V$ be a fibrant replacement of $F_0(C(1^+))$ in the level model structure on $\Sp^{\mN}(\cC\cS^{\cI})$. Then Lemma~\ref{lem:group-compl-of-grouplike} implies that $V$ is in fact an $\Omega$-spectrum and hence fibrant in the stable model structure on $\Sp^{\mN}(\cC\cS^{\cI})$. Consequently, 
$F_0(\mS)\to \Phi^{\textrm{st}}(V)$ is a model of the derived unit and 
Corollary~\ref{cor:Phi-of-S} shows it to be a weak equivalence in spectrum degree 0. 
Now choose a fibrant replacement $F_0(\mS)\to U$ in the level model structure on $\Sp^{\mN}(\Gamma^{\op}\cS_*)$ and notice that $U$ is an 
$\Omega$-spectrum. The above map factors through $U$ as a map of 
$\Omega$-spectra $U\to \Phi^{\mathrm{st}}(V)$ which is a weak equivalence in spectrum degree $0$. This implies the result since both spectra $U$ and 
$\Phi^{\mathrm{st}}(V)$ have the property that the very special $\Gamma$-space $X$ in spectrum degree $n$ has an $(n-1)$-connected space $X(1^+)$ .
\end{proof}

Let us in general write $[-,-]$ for the abelian groups of morphisms in a triangulated category. Recall that an object $K$ in a triangulated category is said to be \emph{compact} if for any family of objects $\{X_i:i\in I\}$, indexed by a set $I$ and with coproduct $\coprod_{i\in I}X_i$, the canonical map 
$\bigoplus_{i\in I}[K,X_i]\to [K,\coprod_{i\in I} X_i]$ is an isomorphism. 
 
 \begin{lemma}\label{lem:compactness}
 The suspension spectrum $F_0(C(1^+))$ is compact in $\Ho(\Sp^{\mN}(\cC\cS^{\cI}))$.
 \end{lemma}
 \begin{proof}
 First observe that there are stable equivalences of spectra 
 \[
 F_0(C(1^+))\xr{\sim} F_0(C_1^{\gp})\xl{F(\xi)} F_0(C_1)
 \]
 so it suffices to show that $F_0(C_1)$ is compact. Let $X$ be an object in 
 $\Sp^{\mN}(\cC\cS^{\cI})$. Using that the positive $\cI$-model structure on 
 $\cC\cS^{\cI}$ is finitely generated in the sense of 
 \cite[Definition 4.1]{Hovey_symmetric-general}, we deduce from \cite[Corollary 4.13]{Hovey_symmetric-general} that $[F_0(C_1),X]$ is isomorphic to the colimit of the morphism sets $[C_1\otimes S^n, X_n]$ in $\Ho(\cC\cS^{\cI})$. Thus, the result is a consequence of the fact that the tensors $C_1\otimes S^n$ are compact in the homotopy category of $\cC\cS^{\cI}$ by 
 \cite[Theorem 7.4.3]{Hovey_model}.
 \end{proof}

\begin{lemma}\label{lem:stable-derived-unit}
The derived unit of the adjunction $(\Lambda^{\textrm{st}},
\Phi^{\textrm{st}})$ is a weak equivalence on cofibrant
  objects.
\end{lemma}
\begin{proof}
This argument uses Morita theory. It follows the proof
of~\cite[Theorem 5.3]{Schwede-S_uniqueness}.
The first thing to show is that $R\Phi^{\textrm{st}}$ preserves
infinite coproducts. We look at the composition of Quillen adjunctions
\[ \ Sp^{\mN} = Sp^{\mN}(\cS_*) \rightleftarrows Sp^{\mN}(\Gamma\cS_*)
\rightleftarrows Sp^{\mN}(\cC\cS^{\cI}). \] The left hand adjunction
is a Quillen equivalence, so it is enough to show that the total
derived functor of the composition $\widetilde{\Phi}$ of the two right
adjoints preserves infinite coproducts. Its left adjoint
$\widetilde{\Lambda}$ sends $\mS$ to $F_0(C(1^+))$. Writing $(-)[n]$ for the $n$-fold shift functor on spectra, this gives 
\[ \pi_n(R\widetilde{\Phi}(Y)) \iso [
\mS[n],R\widetilde{\Phi}(Y)]^{\Ho(Sp^{\mN}(\cS_*))} \iso
[F_0(C(1^+))[n], Y]^{\Ho(Sp^{\mN}(\cC\cS^{\cI}))}.\]
The object $F_0(C(1^+))$ is compact by Lemma~\ref{lem:compactness} which implies that $\pi_n(R\widetilde{\Phi}(-))$ preserves infinite coproducts and the same therefore holds for $R\widetilde{\Phi}$.

Now consider the full subcategory $\cT$ of $\Ho(Sp^{\mN}(\Gamma\cS_*))$
on objects $X$ for which the unit of the derived adjunction $X \to (R\Phi^{\textrm{st}})(L \Lambda^{\textrm{st}})(X)$ is an isomorphism. This is a
triangulated subcategory of $\Ho(Sp^{\mN}(\Gamma\cS_*))$ since both $L
\Lambda^{\textrm{st}} $ and $R\Phi^{\textrm{st}}$ are exact functors of
triangulated categories and it contains $F_0(\mS)$ by Lemma \ref{lem:unit_we_on_sphere}. Furthermore, since $R\Phi^{\textrm{st}}$ preserves infinite coproducts, it follows that $\cT$ is closed under infinite coproducts. Because $F_0(\mS)$ is a compact generator for
$\Ho(Sp^{\mN}(\Gamma\cS_*))$, we deduce that $\cT$ is the whole category
$\Ho(Sp^{\mN}(\Gamma\cS_*))$
\end{proof}

\begin{proposition}
The derived unit of the $(\Lambda, \Phi)$-adjunction is a
  weak equivalence on cofibrant objects.
\end{proposition}
\begin{proof}
Notice first that there is a natural isomorphism 
$\Lambda^{\textrm{st}}F_0\cong F_0\Lambda$ since $\Lambda$ preserves tensors. We know from Lemmas~\ref{lem:F-Ev-adjunctions} and \ref{lem:stable-derived-unit} that the derived units of the adjunctions 
$(F_0,\Ev_0)$ and $(\Lambda^{\textrm{st}},\Phi^{\textrm{st}})$ are weak equivalences on cofibrant objects. By Lemma~\ref{lem:derived_units_composite} this implies that the same holds for the composite adjunction $(F_0\Lambda,\Phi\Ev_0)$ and using 
Lemma~\ref{lem:derived_units_composite} once more we get the result.
\end{proof}

\begin{proof}[Proof of Theorem~\ref{thm:Gamma-csigp}]
Using~\cite[Corollary 1.3.16]{Hovey_model}, 
Lemma \ref{lem:we_between_fibrants} together with the last proposition show that the
$(\Lambda,\Phi)$-adjunction is a Quillen equivalence.
\end{proof}

\begin{proof}[Proof of Corollary~\ref{cor:B-infty-A}]
By Theorem~\ref{thm:Gamma-csigp} and the equivalence between the homotopy categories of $\Gamma$-spaces and connective spectra defined in 
\cite[Theorem 5.8]{Bousfield-F_Gamma-bisimplicial}, this is now a consequence of Propositions~\ref{prop:comparison-gamma-spaces} and \ref{prop:D-E-Gamma}.
\end{proof}

\section{The units model structure}\label{sec:un-csi}
The aim of this section is to prove Theorem \ref{thm:intro-units} from the
introduction which constructs the units of a commutative
$\cI$-space monoid as a cofibrant replacement in a ``units'' model
structure $\cC\cS^{\cI}_{\un}$ on commutative $\cI$-space monoids.

\begin{construction}
  Let $A$ be a commutative $\cI$-space monoid. We define $A^{\times}$
  to be the sub commutative $\cI$-space monoid of invertible path
  components of $A$. That is, $A^{\times}(\bld{n})$ is the sub
  simplicial set consisting of all those path components of
  $A(\bld{n})$ whose vertices represent units in the monoid $\pi_0(A_{h\cI})$. 
  This construction is functorial in $A$, and the inclusion defines a natural map $A^{\times} \to A$ of commutative $\cI$-space monoids.  
\end{construction}
The definition of the units is clearly homotopy invariant in the sense that an 
$\cI$-equivalence $A\to B$ of commutative $\cI$-space monoids induces an $\cI$-equivalence $A^{\times}\to B^{\times}$. 
Recall that in Section \ref{sec:monoid-path-comp} we defined a
commutative $\cI$-space monoid $A$ to be grouplike if the commutative monoid
of path components $\pi_0(A_{h\cI})$ is a group. This condition is equivalent to the equality $A^{\times}=A$. Consequently, a map $A \to B$ with $A$ grouplike factors uniquely as 
$A \to B^{\times} \to B$.

\subsection{Units and cobase change}
For the construction of the units model structure, we need the following result
about how forming the units behaves with respect to cobase change:

\begin{lemma}\label{lem:cobase-change-units}
Let $A$ and $B'$ be commutative $\cI$-space monoids with $B'$ grouplike,
let $A^{\times} \to B'$ be a map in $\cC\cS^{\cI}$, and let $B$ be the pushout
of $B' \ot A^{\times} \to A$ in $\cC\cS^{\cI}$. Then the induced map $B' \to B$
is isomorphic to $B^{\times} \to B$. 
\end{lemma}

To prove the lemma, we need to analyze how a commutative
$\cI$-space monoid decomposes into its units and ``non-units''. It is
convenient to use the following terminology which is analogous to the terminology for
non-unital $\mS$-algebras (`nucas') studied by Basterra
\cite{Basterra_Andre-Quillen}.
\begin{definition}
  Let $C$ be a commutative $\cI$-space monoid. A non-unital
  commutative $C$-algebra ($C$-nuca for short) is a non-unital commutative monoid
  in the category of $C$-modules.  We write
  $C\text{-}\cN\cS^{\cI}$ for the category of $C$-nucas.
\end{definition}

This definition uses that the category of modules $\Mod_C$ over a commutative
$\cI$-space monoid $C$ inherits a symmetric monoidal
product $\boxtimes_{C}$ with $E\boxtimes_C F$ being defined as a
coequalizer of the diagram $E\boxtimes C \boxtimes F \rightrightarrows E\boxtimes
F$ defined by the $C$-actions on $E$ and $F$. A $C$-nuca is a $C$-module $E$ with an associative and commutative
multiplication $E\boxtimes_{C} E \to E$. Equivalently, it is an algebra
over the monad $\mN$ in $\Mod_C$ defined by $\mN X = \coprod_{n > 0}
X^{\boxtimes_C^n}/\Sigma_n$.

For $A$ a commutative $\cI$-space monoid, we write $\widehat A$ for
the sub $\cI$-space of $A$ given by the complement of $A^{\times}$. In
other words, $\widehat A$ consists of the non-invertible path
components of $A$.

\begin{lemma}
  The multiplication of $A$ induces the structure of an
  $A^{\times}$-nuca on $\widehat A$.
\end{lemma}
\begin{proof}
  The product of an element of $A$ with a non-unit is a non-unit, so $A \boxtimes
  A \to A$ restricts to an $A$-module structure on $\widehat A$. This in turn restricts to an $A^{\times}$-module structure and induces a map 
$\widehat A \boxtimes_{A^{\times}} \widehat A\to
  \widehat A$. The commutativity and associativity of $A$ imply that
  $\widehat A$ is an $A^{\times}$-nuca.
\end{proof}

For a $C$-nuca $E$, we can equip the coproduct of the underlying
$\cI$-spaces $C {\textstyle\coprod} E$ with the structure of a
commutative $\cI$-space monoid by defining
\[ (C {\textstyle\coprod} E)\boxtimes(C {\textstyle\coprod} E)
\iso (C\boxtimes C) {\textstyle\coprod} (C\boxtimes E)
{\textstyle\coprod} (E\boxtimes C) {\textstyle\coprod} (E\boxtimes E)
\to C\boxtimes E \] using the monoid structure of $C$, the $C$-module
structure of $E$, and the composition $E \boxtimes E \to E
\boxtimes_{C} E \to E$.

\begin{lemma}\label{lem:units-nonunits-decomp}
  If $A$ is a commutative $\cI$-space monoid, then $A$ and $A^{\times}
  {\textstyle\coprod} \widehat A$ are isomorphic as commutative
  $\cI$-space monoids.\qed
\end{lemma}

Contrary to the square-zero extensions in algebra or the corresponding
constructions for $\mS$-algebras~\cite{Basterra_Andre-Quillen}, 
$C{\textstyle\coprod} E$ fails to be augmented over $C$ because
$\Mod_C$ has no zero object. However, we can use the projection $E \to
*$ to the terminal $\cI$-space to view $C {\textstyle\coprod} E$ as
an object in the category 
of commutative $C$-algebras over $C{\textstyle\coprod}*$.

\begin{proof}[Proof of Lemma \ref{lem:cobase-change-units}]
  Using Lemma~\ref{lem:units-nonunits-decomp}, we define a map of commutative $\cI$-space monoids  $A\cong A^{\times}{\textstyle\coprod}\widehat A\to B'{\textstyle\coprod}*$ by projecting $\widehat A$ onto $*$. By the universal property of the pushout, this induces a map $p\colon B\to B'{\textstyle\coprod}*$, so we may view $B$ as a commutative $B'$-algebra augmented over $B'{\textstyle\coprod}*$. Since $p$ must map the units of $B$ to $B'$, it therefore suffices to show that the map $B'\to B$ maps $B'$ isomorphically onto $p^{-1}(B')$. For this purpose we identify the underlying $\cI$-space of $B$ with the coequalizer $B'\boxtimes_{A^{\times}}\!A$ of the diagram $B'\boxtimes A^{\times}\boxtimes A\rightrightarrows B'\boxtimes A$ defined by the $A^{\times}$-actions on $B'$ and $A$. The result now follows from the chain of isomorphisms
\[
B'\boxtimes_{A^{\times}}\!A\cong B'\boxtimes_{A^{\times}}\!(A^{\times}
{\textstyle\coprod}\widehat A)\cong (B'\boxtimes_{A^{\times}}\!A^{\times}){\textstyle
\coprod}(B'\boxtimes_{A^{\times}}\!\widehat A)\cong B'{\textstyle\coprod}
(B'\boxtimes_{A^{\times}}\!\widehat A)  
\]     
of $\cI$-spaces over $B'{\textstyle\coprod}*$. 
\end{proof}

\subsection{The model structure}
To prove the existence of the units model structure, we will use the
method of ``$Q$-structures'' introduced by Bousfield and
Friedlander~\cite{Bousfield-F_Gamma-bisimplicial} and refined by
Bousfield~\cite[\S 9]{Bousfield_telescopic}. In the formulation
of~\cite[\S 9]{Bousfield_telescopic}, this method produces a left
Bousfield localization (that has fewer fibrant objects than the
original model structure). However, to prove Theorem
\ref{thm:intro-units} we need to construct a right Bousfield
localization (with fewer cofibrant objects than the original model
structure). The key point here is that the results of~\cite[\S
9]{Bousfield_telescopic} dualize to give right Bousfield localizations
because they do not make use of generating (acyclic) cofibrations. (We
thank Jens Hornbostel for pointing us to this.) This is not the case
with most other localization techniques. For the readers convenience,
we formulate the dual version of Bousfield's localization
theorem~\cite[Theorem 9.3]{Bousfield_telescopic} and the necessary
prerequisites.

Let $\cC$ be a model category, let $P\colon \cC \to \cC$ be a functor, 
and let $\beta \colon P \to \id_{\cC}$ be a natural transformation satisfying
the following axioms:
\begin{enumerate}[(B1)]
\item If $f \colon A \to B$ is a weak equivalence, then so is $P(f)$.
\item For every object $A$ of $\cC$, the maps $P(\beta_A)$ and
  $\beta_{P(A)}$ are weak equivalences.
\item For a pushout square
\[\xymatrix@-1pc{ C \ar[r]^{g} \ar[d]_{f} & A \ar[d] \\ D \ar[r]^{h} & B}\]
in $\cC$ with $f$ a cofibration of cofibrant objects and
$\beta_{C}$, $\beta_{D}$, and $P(g)$ weak equivalences, the map
$P(h)$ is also a weak equivalence.
\end{enumerate}
We call a map $f$ in $\cC$ a \emph{$P$-equivalence} if $P(f)$ is a weak
equivalence in $\cC$, a \emph{$P$-fibration} if it is a fibration in
$\cC$, and a \emph{$P$-cofibration} if it has the left lifting
property with respect to all maps that are $P$-fibrations and
$P$-equivalences. Dualizing the statement and the proof
of~\cite[Theorem 9.3]{Bousfield_telescopic}, we get the following result.
\begin{proposition}\label{prop:Q-structures-dualized}
  Let $\cC$ be a proper model category with a functor
  $P\colon \cC \to \cC$ and a natural transformation $\beta \colon P
  \to \id_{\cC}$ satisfying (B1)-(B3). Then the $P$-equivalences,
  $P$-cofibrations, and $P$-fibrations form a proper model structure
  on $\cC$, and a map $f \colon C \to D$ is a $P$-cofibration if and
  only if it is a cofibration in $\cC$ and
\[\xymatrix@-1pc{
  P(C) \ar[rr]^{\beta_C} \ar[d]_{P(f)} && C \ar[d]^{f} \\
  P(D) \ar[rr]^{\beta_D} && D }\] is a homotopy pushout square in $\cC$.\qed
\end{proposition}

\begin{proposition}\label{prop:units-structure}
The category $\cC\cS^{\cI}$ equipped with the functor $(-)^{\times}\colon \cC\cS^{\cI}\to
\cC\cS^{\cI}$ and the natural inclusion $\beta_A\colon A^{\times}\to A$ satisfies (B1)-(B3).
\end{proposition}
\begin{proof}
The condition in (B1) is clearly satisfied and the maps in (B2) are even isomorphisms.  
For (B3) we first consider the outer pushout square in the diagram
  \[\xymatrix@-1pc{
    C \ar[d] \ar@{-->}[rr] \ar@/^.6pc/[rrrr] & & A^{\times} \ar[d]
    \ar[rr]
    & & A \ar[d]\\
    D \ar[rr] \ar@/_.6pc/[rrrr]& & B' \ar[rr] & & B.  }\] 
  Because
  $\beta_C$ and $\beta_D$ are $\cI$-equivalences, both $C$ and $D$ are
  grouplike. Hence $C \to A$ factors through $A^{\times}$
  as indicated by the dotted arrow and we define $B'$ to be
  the pushout of $D \ot C \to A^{\times}$. It follows that the pushout
  in the outer square is isomorphic to the iterated pushouts defined by the
  two inner squares. By assumption, $C=C^{\times} \to A^{\times}$ is an
  $\cI$-equivalence and as $C\to D$ is assumed to be a cofibration,
  left properness of the positive $\cI$-model structure implies that $D \to B'$
  is an $\cI$-equivalence. Since $B' \to B$ is isomorphic to
  $\beta_B$ by Lemma~\ref{lem:cobase-change-units}, this
  implies that $D^{\times} \to B^{\times}$ is an $\cI$-equivalence
  and (B3) is proved.
\end{proof}

We write $\cC\cS^{\cI}_{\un}$ for the ``units'' model structure on commutative 
$\cI$-space monoids specified by Propositions~\ref{prop:Q-structures-dualized} and 
\ref{prop:units-structure}. (Here we use that that the positive $\cI$-model structure on $\cC\cS^{\cI}$ is proper by \cite[Proposition~3.5]
{Sagave-S_diagram}).
  
The units model structure can also be characterized as a right Bousfield localization of the positive $\cI$-model structure. We refer the reader to 
\cite[Section 3.3]{Hirschhorn_model} for the definition and properties of right Bousfield localizations.  

\begin{proposition}\label{prop:units-right-Bousfield}
The units model structure $\cC\cS^{\cI}_{\un}$ is the right Bousfield localization of the positive $\cI$-model structure with respect to the class of maps $A^{\times}\to A$ for $A$ positive $\cI$-fibrant.  
\end{proposition}

\begin{proof}[Proof of Proposition~\ref{prop:units-right-Bousfield}]
We use the standard terminology (as in \cite[Section~3]{Hirschhorn_model}) concerning colocal objects and colocal equivalences with respect to the class of maps specified in the proposition. If $W$ is a cofibrant and grouplike object in 
$\cC\cS^{\cI}$, then the inclusion $A^{\times}\to A$ induces an isomorphism
$
\Map(W,A^{\times})\to \Map(W,A)
$
for all $A$, which implies that $W$ is colocal. For a general object $A$, we consider the composition $\widetilde A^{\times}\to A^{\times}\to A$, where the first map is a functorial cofibrant replacement of $A^{\times}$ in the positive $\cI$-model structure. Since this is a functorial colocalization of $A$, we conclude from 
\cite[Theorem 3.2.18]{Hirschhorn_model} that a map $A\to B$ of commutative $\cI$-space monoids is a colocal equivalence if and only if the map of units $A^{\times}\to B^{\times}$ is an $\cI$-equivalence. This implies the statement of the proposition.   
\end{proof}

\begin{proof}[Proof of Theorem~\ref{thm:intro-units}]
It is clear that the natural inclusion $A^{\times}\to A$ induces an isomorphism 
$(A^{\times})_{h\cI}\to (A_{h\cI})^{\times}$ which gives the description of the weak equivalences in the theorem. The characterization of the cofibrant objects follows from 
Proposition~\ref{prop:Q-structures-dualized} and this in turn implies the last statement in the theorem. 
\end{proof}

\begin{appendix}

\section{Cellularity of the positive \texorpdfstring{$\cI$}{I}-model structure}\label{app:cellular}
In this section we verify that the positive $\cI$-model structure on $\cC\cS^{\cI}$ is cellular both in the simplicial and in the topological setting. This is needed for the construction of the group completion model structure as a left Bousfield localization. Recall from~\cite[Definition~12.1.1]{Hirschhorn_model} that a
cellular model category is a cofibrantly generated model category
$\cC$ with generating cofibrations $I$ and generating acyclic
cofibrations $J$ such that the domains and codomains of the maps in
$I$ are compact relative to $I$~\cite[Definition
10.8.1]{Hirschhorn_model}, the domains of the maps in $J$ are small
relative to the subcategory of relative $I$-cell
complexes~\cite[Definition 10.4.1]{Hirschhorn_model}, and the
cofibrations are effective monomorphisms~\cite[Definition
10.9.1]{Hirschhorn_model}.
\begin{proposition}\label{prop:csi-cellular}
Let $\cS$ be either the category of simplicial sets or the category of compactly generated weak Hausdorff topological spaces. Then the positive $\cI$-model structure on 
$\cC\cS^{\cI}$ is cellular.
\end{proposition}
The proof is based on the next lemma. Consider a map $A\to B$ of commutative $\cI$-space monoids and the diagram $B\ot A\to B$. We write $B\boxtimes_A B$ for the pushout of this as a diagram in $\cC\cS^{\cI}$ and $B\coprod_A B$ for the pushout of the underlying diagram of $\cI$-spaces.  

\begin{lemma}
If $A\to B$ is a cofibration in $\cC\cS^{\cI}$, then the canonical map of $\cI$-spaces $B\coprod_AB\to B\boxtimes_AB$ is a monomorphism.
\end{lemma}
\begin{proof}
For the proof we need the (absolute) flat model structure on $\cS^{\cI}$ introduced in 
\cite[Section 3]{Sagave-S_diagram}. This is a monoidal model structure which satisfies the monoid axiom, so by~\cite[Theorem 4.1(2)]{Schwede-S_algebras} it lifts to a
``flat'' model structure on the category of $A$-modules for the commutative $\cI$-space monoid $A$. Furthermore, the flat model structure on $A$-modules is monoidal with respect to the symmetric monoidal product $\boxtimes_A$ inherited from the 
$\boxtimes$-product. Arguing as in the proof of 
\cite[Proposition~12.5]{Sagave-S_diagram}, one shows that the cofibrancy assumption on $A\to B$ as a map in $\cC\cS^{\cI}$ implies that it is a cofibration as a map of 
$A$-modules with respect to the flat model structure.  Now observe that the map in the lemma can be identified with the pushout-product of the map $A\to B$ with itself (with respect to the $\boxtimes_A$-product as a map of $A$-modules). 
Hence it follows from the pushout-product axiom that the map in question is a cofibration in the flat model structure on $A$-modules. By an argument similar to that used in the proof of \cite[Proposition~12.7]{Sagave-S_diagram} this shows it to be an $h$-cofibration in the sense of \cite[Section~7]{Sagave-S_diagram}, which in turn implies that it is levelwise injective and hence a monomorphism. 
\end{proof}

\begin{proof}[Proof of Proposition \ref{prop:csi-cellular}]
For the compactness and smallness assertions, we recall from \cite[Section~6]{Sagave-S_diagram} that the objects in question are obtained by applying free functors (that is, left adjoints of evaluation functors) to compact objects in $\cS$. The assertions therefore hold because sequential colimits in $\cC\cS^{\cI}$ are created in $\cS^{\cI}$ and cofibrations in $\cC\cS^{\cI}$ are $h$-cofibrations in $\cS^{\cI}$ by \cite[Section~7]{Sagave-S_diagram}. 

By definition of an effective monomorphism, we have to show that if $A\to B$ is a cofibration in $\cC\cS^{\cI}$, then it is an equalizer of the canonical maps 
$B\rightrightarrows B\boxtimes_A B$. Since equalizers are created in the underlying category $\cS^{\cI}$, it suffices to show that this holds for the underlying maps of $\cI$-spaces.  As maps of $\cI$-spaces we have the factorization $B\rightrightarrows B\coprod_A B\to B\boxtimes_A B$ where the second map is a monomorphism by the previous lemma. It therefore suffices to show that $A\to B\rightrightarrows B\coprod_A B$ is an equalizer diagram and this again follows from the fact that the map of $\cI$-spaces underlying $A\to B$ is an $h$-cofibration by \cite[Proposition~12.7]{Sagave-S_diagram}.   
\end{proof}

\section{Bi-\texorpdfstring{$\Gamma$}{Gamma}-spaces}\label{app:bi-Gamma}
As in Section \ref{subs:Gamma-reminder}, let $\Gamma^{\op}$ be the
category of based finite sets, and let $\cS_*$ denote the category of based simplicial sets. 

\begin{definition}
  A \emph{bi-$\Gamma$-space} is a functor $\Gamma^{\op}\times
  \Gamma^{\op} \to \cS_*$ with $X(k^+,0^+) = *$ and $X(0^+,k^+)=*$ for all
  objects $k^+$ of $\Gamma^{\op}$.
\end{definition}
For $i=1,2$, the projections $(k_1 + k_2)^+ \to k_i^+$ and $(l_1 +
l_2)^+ \to l_i^+$ induce a map
\begin{equation}\label{eq:bi-sp-map}
  X((k_1 + k_2)^+, (l_1 + l_2)^+) \to X(k_1^+,l_1^+)\times X( k_1^+,l_2^+)\times X(k_2^+, l_1^+)\times X(k_2^+,l_2^+)
\end{equation}

\begin{definition}
The bi-$\Gamma$-space $X$ is \emph{bi-special} if the map
\eqref{eq:bi-sp-map} is a weak equivalence for all 
$k_1,k_2,l_1,l_2 \geq 0$.
\end{definition}

It is clear from the definition that $X$ is bi-special if and only if the 
$\Gamma$-spaces $X(k^+,-)$ and $X(-,k^+)$ are special in the sense of Bousfield-Friedlander \cite{Bousfield-F_Gamma-bisimplicial} for each 
$k\geq 0$. If $X$ is bi-special, then the two projections $2^+ \to 1^+$ and the
fold map $2^+ \to 1^+$ induce a monoid structure on $\pi_0(X(1^+,1^+))$ via
\[ \pi_0(X(1^+,1^+)) \times \pi_0(X(1^+,1^+)) \xleftarrow{\iso}
\pi_0(X(2^+,1^+)) \to \pi_0(X(1^+,1^+)).\] A second monoid structure
arises from the induced maps in the second variable.
\begin{lemma}\label{lem:two-monoid-structures}
The two monoid structures on $\pi_0(X(1^+,1^+))$ coincide. 
\end{lemma}
\begin{proof}
This follows by a version of the Eckmann-Hilton argument:  The two
ways to multiply 4 elements coincide because they are both
given by 
\[\pi_0(X(1^+,1^+))^{\times 4} \xleftarrow{\iso} \pi_0(X(2^+,2^+)) \to  
\pi_0(X(1^+,1^+))
\]
where the first isomorphism is induced by the weak equivalence in \eqref{eq:bi-sp-map}. 
 \end{proof}
\begin{definition}
  A bi-$\Gamma$-space $X$ is bi-very special if it is bi-special and the
  monoid $\pi_0(X(1^+,1^+))$ is a group.
\end{definition}
Again the condition of being bi-very special is equivalent to each of the 
$\Gamma$-spaces $X(k^+,-)$ and $X(-,k^+)$ being very special in the sense of \cite{Bousfield-F_Gamma-bisimplicial} for all 
$k\geq 0$. The following is the bi-$\Gamma$-space analogue of the construction in~\cite[\S 4]{Bousfield-F_Gamma-bisimplicial}. One can prolong a bi-$\Gamma$-space $X$ to a functor on pairs of based (not
necessarily finite) sets by a left Kan extension. If $K$ and $L$ are
based simplicial sets, we may then evaluate the prolonged functor in each bisimplicial degree $(K_k, L_l)$. Forming the diagonal of the resulting tri-simplicial set $X(K_k,L_l)_n$, we get a simplicial set $X(K,L)$. As in the case of $\Gamma$-spaces, there are natural assembly maps
\[X(K,L)\sm P \to X(K\sm P,L) \quad \textrm{ and }\quad X(K,L)\sm Q \to X(K,L \sm Q),\] and similar assembly maps acting from the left.
Setting $X_{m,n} = X(S^m,S^n)$ for $m,n\geq 0$, we get a bispectrum with
structure maps 
\[ X_{m,n} \sm S^1 \to X_{m+1,n} \quad \textrm{ and }\quad X_{m,n} \sm
S^1 \to X_{m,n+1}\] induced by the assembly maps. Let us write 
$(-)^{\mathrm{fib}}$ for the fibrant replacement functor on $\cS_*$ that takes a based simplicial set to the simplicial complex of its topological realization.

\begin{definition}
A bispectrum $X$ is a \emph{bi-$\Omega$-spectrum} if the structure maps
induce weak equivalences
\[ X^{\textrm{fib}}_{m,n} \to \Omega(X^{\textrm{fib}}_{m+1,n}) \quad \textrm{ and }\quad  X^{\textrm{fib}}_{m,n} \to 
\Omega(X^{\textrm{fib}}_{m,n+1})\]
in all bi-degrees $(m,n)$.
\end{definition}

\begin{lemma}\label{lem:bi-Gamma-Omega}
If the bi-$\Gamma$-space $X$ is bi-very special, then the associated bispectrum is a bi-$\Omega$-spectrum.
\end{lemma}
\begin{proof}
The condition that $X$ be bi-very special implies that the 
$\Gamma$-spaces $X(S^n,-)$ and $X(-,S^n)$ are very special for all 
$n\geq 0$. By \cite[Theorem~4.2]{Bousfield-F_Gamma-bisimplicial} this implies the statement of the lemma.
\end{proof}

Given a bi-$\Gamma$-space $X$, we let $X_0$ be the $\Gamma$-space with
$X_0(s^+) = X(s^+,1^+)$. We define a new $\Gamma$-space $X_1$ 
by setting 
\[ X_1(s^+) = \hocolim_{(m,n)\in\cN \times \cN}
\Omega^{m+n}(X(s^+\sm S^m,S^n)^{\textrm{fib}})\]
where the category $\cN$ is as in Section~\ref{subsec:semistability}. 
The map from the initial vertex into the
homotopy colimit induces a map of $\Gamma$-spaces $X_0 \to
X_1$. Similarly, we define a $\Gamma$-space $X_2$ by
\[ X_2(s^+) = \hocolim_{(m,n)\in\cN \times \cN}\Omega^{m+n
}(s^+\sm X(S^m,S^n)^{\textrm{fib}})\] 
and notice that the assembly map for the
left action defines a map of $\Gamma$-spaces $X_2 \to X_1$.
\begin{lemma}
If $X$ is bi-very special, then both maps $X_0 \to X_1$ and  $X_2 \to X_1$
are level equivalences. 
\end{lemma}
\begin{proof}
The condition that $X$ be bi-very special implies that each of the 
bi-$\Gamma$-spaces $X(s^+\wedge-,-)$ is also bi-very special. Therefore 
Lemma~\ref{lem:bi-Gamma-Omega} implies that $X_1$ is the homotopy colimit of a sequence of level equivalences which gives the result for 
$X_0\to X_1$.   
For $X_2 \to X_1$, the claim follows because the map of spectra induced by the assembly map $s^+ \sm X(S^{-},S^n) \to X(s^+ \sm S^{-},S^n)$ is a stable
equivalence for each $n\geq 0$ by 
\cite[Lemma  4.1]{Bousfield-F_Gamma-bisimplicial}. 
\end{proof}
Since $X_2$ is symmetric in $m$ and $n$, the last lemma and its dual version have the following consequence.
\begin{proposition}\label{prop:equivalent-Gamma-spaces}
If the bi-$\Gamma$-space $X$ is bi-very special, then there is a
zig-zag chain of level equivalences of $\Gamma$-spaces between 
$X(-,1^+)$ and $X(1^+,-)$. \qed
\end{proposition}

\section{Group completion and units in the topological context}
\label{app:topological}
In this section we show how to deduce the theorems from the introduction in the topological context. Thus, let $\cU$ denote the category of compactly generated weak Hausdorff topological spaces, and consider the corresponding category of topological $\cI$-spaces $\cU^{\cI}$ which was also studied in~\cite{Sagave-S_diagram}. We continue to let $\cS$ denote the category of simplicial sets. The first observation is that geometric realization $\geor{-}$ and the singular complex $\Sing$ define a pair of adjoint functors 
$\geor{-}\colon \cS^{\cI}\rightleftarrows \cU^{\cI}\! :\!\Sing$, which is a Quillen equivalence with respect to the positive $\cI$-model structures on $\cS^{\cI}$ and $\cU^{\cI}$. Viewing these categories as symmetric monoidal categories under the 
$\boxtimes$-product, the functor $\geor{-}$ is strong symmetric monoidal and $\Sing$ is (lax) symmetric monoidal. Since the unit and counit for the adjunction are monoidal natural transformations, this implies that there is an induced Quillen equivalence between the categories of commutative monoids 
$\geor{-}\colon \cC\cS^{\cI}\rightleftarrows \cC\cU^{\cI}\! :\!\Sing$, again with respect to the positive $\cI$-model structures.    

\subsection{Group completion in the topological context}
The main technical difference encountered in the topological setting is that the geometric realization functor $\geor{-}$ from simplicial spaces to spaces is homotopy invariant only for simplicial spaces that are ``good'' in the sense of 
Segal~\cite{Segal_categories}. Thus, when forming the geometric realization we should either stipulate that the simplicial spaces be good, or otherwise use the ``fat'' realization $\lVert-\rVert$ considered in \cite[Appendix A]{Segal_categories}. For a topological monoid $M$ this means that when forming the bar construction we should either require the unit of $M$ to be a non-degenerate base point, or otherwise use the fat realization of the usual simplicial bar construction $B_{\bullet}(M)$. On the other hand, since all objects in $\cU$ are fibrant, the levelwise fibrant replacement $(-)^{\mathrm{fib}}$ used in Section~\ref{sec:bar-construction} can be dropped in the topological setting. 
 
For our results on group completion, the above discussion has the following implications.
The assumption that the commutative $\cI$-space monoid $A$ in 
Theorem~\ref{thm:intro-bar-completion} be cofibrant implies that the underlying $E_{\infty}$ space $A_{h\cI}$ is non-degenerately based, and the statement of the theorem therefore remains valid in the topological setting. Since Appendix~\ref{app:cellular} about cellularity includes the topological case, the construction of the group completion model structure $\cC\cU^{\cI}_{\gp}$ proceeds as in the simplicial case. The topological version of Theorem~\ref{thm:csi-gp} holds with the understanding that the notation $B(A_{h\cI})\to B(A'_{h\cI})$ indicates the fat realization of the simplicial map 
$B_{\bullet}(A_{h\cI})\to B_{\bullet}(A'_{h\cI})$. Consequently, we have the Quillen equivalence
\[
\geor{-}\colon \cC\cS^{\cI}_{\gp}\rightleftarrows \cC\cU^{\cI}_{\gp} :\!\Sing.
\]  
This in turn implies that the discussion of group completion and repletion in 
Section~\ref{sec:repletion} carries over to the topological setting.  

\subsection{The relation to topological \texorpdfstring{$\Gamma$}{Gamma}-spaces}
It follows from \cite[Theorem B1]{Schwede_Gamma-spaces}, that the category of topological $\Gamma$-spaces $\Gamma^{\op}\cU_*$ has a stable $Q$-model structure such that the geometric realization and singular complex functors induce a Quillen equivalence 
$\geor{-}\colon \Gamma^{\op}\cS_*\rightleftarrows \Gamma^{\op}\cU_*\! :\!\Sing$. 
We wish to relate the categories $\Gamma^{\op}\cU_*$ and $\cC\cU^{\cI}$ directly, and for this we observe that the general criteria in \cite[Proposition~VII~2.10]{EKMM} and \cite[Proposition~4.2.19]{Hovey_model} imply that the category $\cC\cU^{\cI}$ is enriched, tensored, and cotensored over the category of based spaces $\cU_*$, such that the positive $\cI$-model structure makes it a based topological model category 
(a $\cU_*$-category in the sense of \cite[Definition~4.2.18]{Hovey_model}). Thus, we can imitate the definition of the adjoint functor pair $(\Lambda,\Phi)$ in 
Section~\ref{subs:construction-adj} to get a Quillen adjunction
\begin{equation}\label{Gamma-top-csigp}
\Lambda\colon \Gamma^{\op}\cU_*\rightleftarrows \cC\cU^{\cI}_{\gp} :\!\Phi.
\end{equation}
This fits in a commutative diagram of Quillen adjunctions
\[
\xymatrix@-.5pc{\Gamma^{\op}\cS_*
\ar@<.5ex>[r]^{\Lambda} \ar@<-.5ex>[d]_{\geor{-}} &
\cC\cS^{\cI}_{\textrm{gp}}
\ar@<.5ex>[l]^{\Phi}\ar@<-.5ex>[d]_{\geor{-}} \\ 
\Gamma^{\op}\cU_*
\ar@<.5ex>[r]^{\Lambda} \ar@<-.5ex>[u]_{\Sing}
& \cC\cU^{\cI}_{\gp} \ar@<.5ex>[l]^{\Phi}
\ar@<-.5ex>[u]_{\Sing}}
\]
and it therefore follows from Theorem~\ref{thm:Gamma-csigp} and the 2-out-of-3 property for Quillen equivalences that the $(\Lambda,\Phi)$-adjunction in 
\eqref{Gamma-top-csigp} is a Quillen equivalence. Using this, we also get a topological analogue of Corollary~\ref{cor:B-infty-A}. 

\subsection{Units in the topological context}
The units $A^{\times}$ of an object $A$ in $\cC\cU^{\cI}$ is defined as in the simplicial setting by letting $A^{\times}(\bld n)$ be the union of the path components in $A(\bld n)$ that represent units in the commutative monoid $\pi_0(A_{h\cI})$. Thus, we have a natural inclusion $A^{\times}\to A$ of commutative $\cI$-space monoids and $A$ is grouplike if and only if this is an equality. In analogy with the definition of $\cC\cS^{\cI}_{\un}$, we define the units model structure $\cC\cU^{\cI}_{\un}$ to be the right Bousfield localization of the positive $\cI$-model structure with respect to the inclusions $A^{\times}\to A$ for $A$ positive $\cI$-fibrant. The weak equivalences and the cofibrant objects in 
$\cC\cU^{\cI}_{\un}$ can then be described as in Theorem~\ref{thm:intro-units} and 
we have the Quillen equivalence 
$
\geor{-}\colon \cC\cS^{\cI}_{\un}\rightleftarrows \cC\cU^{\cI}_{\un} :\!\Sing.
$
However, since it is not clear that the decomposition $A=A^{\times}\coprod \widehat A$ in Lemma~\ref{lem:units-nonunits-decomp} holds topologically, an additional argument is needed in order to see that the axiom (B3) required for using the Bousfield-Friedlander localization principle in Proposition~\ref{prop:Q-structures-dualized} is satisfied. The problem with the topological decomposition of a commutative $\cI$-space monoid in its units and non-units is that this is a decomposition in path components and these may not agree with the connected components. However, it is clear that the analogue of 
Lemma~\ref{lem:units-nonunits-decomp} holds for an object in $\cC\cU^{\cI}$ that is obtained by geometric realization from an object in $\cC\cS^{\cI}$. Arguing as in the proof of Proposition~\ref{prop:units-structure} it follows that axiom (B3) is satisfied for such objects. The general case of axiom (B3) follows from this by considering the diagram
\[
\xymatrix@-1pc{
D & C \ar[r]\ar[l]  & A\\
D' \ar[u]_{\simeq}& \geor{\Sing C} \ar[u]_{\simeq} \ar[r] \ar[l] & \ar[u]_{\simeq}\geor{\Sing A,}
} 
\]
where the left hand square is defined by factoring the map $\geor{\Sing C}\to C\to D$ as a cofibration $\geor{\Sing C}\to D'$ followed by an acyclic fibration $D'\to D$. (The objects 
$\geor{\Sing C}$ and $D'$ may not be cofibrant but this does not affect the proof.) Since (B3) holds for the pushout defined by the bottom diagram, it follows from left properness that it also holds for the pushout defined by the upper diagram.

Summarizing the above discussion, there is a topological analogue of the diagram of Quillen equivalences in \eqref{diag:all-adjunctions}. From this we then get a spectrum of units functor which to a topological commutative symmetric ring spectrum $R$ associates the topological $\Gamma$-space $\gl_1(R)$. 

\end{appendix}

\end{document}